\documentclass[11pt, a4paper]{amsart}
\usepackage{amssymb}
\usepackage{amsmath}
\usepackage{amssymb}
\usepackage{amsthm}
\usepackage{bbm}
\usepackage{bm}
\usepackage{mathtools}
\usepackage{mathrsfs}
\usepackage[T1]{fontenc}
\usepackage[usenames,dvipsnames,svgnames,table]{xcolor}
\usepackage{esint}
\usepackage{soul}
\allowdisplaybreaks
\usepackage{combelow}
\usepackage{enumitem}
\setenumerate{label={\rm (\alph{*})}}
\usepackage[colorlinks=true,citecolor=black,urlcolor=black,
linkcolor=black]{hyperref}

\usepackage{graphicx}
\usepackage{tikz}

\usepackage{geometry}
\usepackage{verbatim}

\mathtoolsset{showonlyrefs}

\newtheorem{theorem}{Theorem}[section]
\newtheorem{bigthm}{Theorem}

\newtheorem{lemma}[theorem]{Lemma}
\newtheorem{proposition}[theorem]{Proposition}

\newtheorem{openproblem}[theorem]{Open Problem}

\theoremstyle{remark}
\newtheorem{remark}[theorem]{Remark}

\numberwithin{equation}{section}

\theoremstyle{definition}
\newtheorem{definition}[theorem]{Definition}

\hypersetup{bookmarksdepth=2}

\usepackage{booktabs}

\normalsize

\def\XXint#1#2#3{{\setbox0=\hbox{$#1{#2#3}{\int}$ }
		\vcenter{\hbox{$#2#3$ }}\kern-.6\wd0}}

\newcommand{\mres}{\mathbin{\vrule height 1.6ex depth 0pt width
		0.13ex\vrule height 0.13ex depth 0pt width 1.3ex}}

		\makeatletter
\newcommand*\bigcdot{\mathpalette\bigcdot@{.5}}
\newcommand*\bigcdot@[2]{\mathbin{\vcenter{\hbox{\scalebox{#2}{$\m@th#1\bullet$}}}}}
\makeatother

\newcommand{\bv}{\operatorname{BV}}
\newcommand{\di}{\operatorname{div}}

\newcommand{\curl}{\operatorname{curl}}
\newcommand{\dif}{\operatorname{d}\!}
\newcommand{\spt}{\operatorname{spt}}

\newcommand{\besov}{\operatorname{B}}

\newcommand{\ccinfty}{\operatorname{C}_{\mathrm{c}}^{\infty}}

\renewcommand{\geq}{\geqslant}
\newcommand{\hold}{\operatorname{C}}

\newcommand{\Rea}{\mathrm{Re\,}}
\newcommand{\Ima}{\mathrm{Im\,}}
\newcommand{\ima}{\mathrm{im\,}}

\newcommand{\lebe}{\operatorname{L}}
\newcommand{\lin}{\operatorname{Lin}}
\newcommand{\slin}{\operatorname{SLin}}
\newcommand{\locc}{\operatorname{loc}}
\renewcommand{\leq}{\leqslant}

\newcommand{\proj}{\mathrm{Proj}}
\newcommand{\rank}{\operatorname{rank}}

\newcommand{\rmim}{\mathrm{im\,}}
\newcommand{\sobo}{\operatorname{W}}
\newcommand{\spann}{\operatorname{span}}

\newcommand{\A}{\mathscr{A}}
\newcommand{\B}{\mathscr{B}}

\newcommand{\C}{\mathscr{C}}
\newcommand{\CC}{\mathrm{C}}

\newcommand{\eps}{\varepsilon}

\newcommand{\EE}{\mathscr{E}}
\newcommand{\F}{\mathcal{F}}

\newcommand{\Le}{\mathscr{L}}

\newcommand{\N}{\mathbb{N}}

\newcommand{\R}{\mathbb{R}}

\newcommand{\pp}{A}
\newcommand{\qq}{B}

\makeatletter
\DeclareFontFamily{OMX}{MnSymbolE}{}
\DeclareSymbolFont{MnLargeSymbols}{OMX}{MnSymbolE}{m}{n}
\SetSymbolFont{MnLargeSymbols}{bold}{OMX}{MnSymbolE}{b}{n}
\DeclareFontShape{OMX}{MnSymbolE}{m}{n}{
    <-6>  MnSymbolE5
   <6-7>  MnSymbolE6
   <7-8>  MnSymbolE7
   <8-9>  MnSymbolE8
   <9-10> MnSymbolE9
  <10-12> MnSymbolE10
  <12->   MnSymbolE12
}{}
\DeclareFontShape{OMX}{MnSymbolE}{b}{n}{
    <-6>  MnSymbolE-Bold5
   <6-7>  MnSymbolE-Bold6
   <7-8>  MnSymbolE-Bold7
   <8-9>  MnSymbolE-Bold8
   <9-10> MnSymbolE-Bold9
  <10-12> MnSymbolE-Bold10
  <12->   MnSymbolE-Bold12
}{}

\let\llangle\@undefined
\let\rrangle\@undefined
\DeclareMathDelimiter{\llangle}{\mathopen}%
                     {MnLargeSymbols}{'164}{MnLargeSymbols}{'164}
\DeclareMathDelimiter{\rrangle}{\mathclose}%
                     {MnLargeSymbols}{'171}{MnLargeSymbols}{'171}
\makeatother

\newcommand{\abs}[1]{\left\lvert #1 \right\rvert}
\newcommand{\abss}[1]{\lvert #1 \rvert}
\newcommand{\brac}[1]{\left(#1\right)}

\newcommand{\brangle}[1]{\left\langle#1\right\rangle}
\newcommand{\Brangle}[1]{\langle\kern-2.5pt\langle #1 \rangle\kern-2.5pt\rangle}

\newcommand{\col}{\colon\,}

\providecommand{\eps}{\varepsilon}
\newcommand{\indi}{\mathbbm{1}}

\newcommand{\normm}[1]{\lVert#1\rVert}
\newcommand{\normiii}[1]{{\left\vert\kern-0.25ex\left\vert\kern-0.25ex\left\vert #1 \right\vert\kern-0.25ex\right\vert\kern-0.25ex\right\vert}}

\newcommand{\weaksto}{\overset{*}{\rightharpoonup}}

\newcommand{\dd}{\mathrm{d}}
\usepackage{todonotes}

\newcommand{\rmc}{\mathrm{c}}
\newcommand{\rma}{\mathrm{a}}
\newcommand{\rms}{\mathrm{s}}

\AtEndDocument{\bigskip{\footnotesize%
    \noindent\textsc{Max Planck Institute for Mathematics in the Sciences,\\ Inselstrasse 22, Leipzig, 04103, Germany}\\
    \textit{E-mail address}, Z.~Li: \texttt{zhuolin.li@mis.mpg.de}
		
		\vspace{2.4pt}
		\noindent\textsc{Department of Mathematics and Statistics, Georgetown University,
                  \\3700 Reservoir Road NW, Washington, D.C., 20057, USA }   \\
		\noindent\textit{E-mail address}, C.~Irving: \texttt{ci152@georgetown.edu}; B.~Rai\cb{t}\u{a}: \texttt{br607@georgetown.edu}
}}

\usepackage[
backend=biber,
style=alphabetic,
maxnames=99,
  giveninits=true,
  url=false,
  doi=true
]{biblatex}

\begin{document}
	\title[Regularity for $\A$-quasiconvex integrals]{Partial regularity for $\A$-quasiconvex variational problems of linear growth}
    \author[C.~Irving]{Christopher~Irving}
\author[Z.~Li]{Zhuolin~Li}
\author[B.~Rai\cb{t}\u{a}]{Bogdan~Rai\cb{t}\u{a}}

\begin{abstract}
We prove that minimizers of variational integrals 
$$
\mathcal E(v)=\int_\Omega f(v)\quad\text{for }v\in\mathcal M(\Omega)\text{ such that } \mathscr{A}
v=0, 
$$
  are partially continuous provided that the integrands $f$ are strongly
$\mathscr{A}$-quasiconvex in a suitable sense. We consider linear growth problems, linear pde operators $\mathscr{A}$ of constant rank, and
variations of the
form $v+\varphi$ with $\mathscr{A}$-free $\varphi\in \mathrm{C}_{\rmc}^\infty(\Omega)$.
Our analysis also covers the ``potentials case'' 
$$ 
\mathcal F(u)=\int_\Omega f( \mathscr{B} u)\quad\text{for }
u\in\mathscr D'(\Omega)\text{ such that }\B u\in \mathcal M(\Omega), 
$$ 
where $\mathscr{B}$ is a different linear pde operator of constant rank. 
Both our main results extend to $x$-dependent integrands.
 \end{abstract}
 \maketitle

\section{Introduction}
We  consider integral functionals of linear growth on bounded open sets $\Omega\subset\R^n$
\begin{align}\label{eq:P_Omega_Afree}
 \mathcal E(v,\Omega)=\int_\Omega f(x, v)\qquad\text{for }v\in\mathcal M(\Omega,V)\text{ such that }\A v=0
\end{align}
and their \textit{local minimizers} with respect to $\A$-free perturbations,
 \begin{align}\label{eq:local_min_A}
\mathcal E(v,\Omega)\leq \mathcal E(v+\varphi,\Omega)\quad\text{for $\varphi\in \hold_c^\infty(\Omega,V)$ with }\A \varphi=0.
\end{align}

 Here $\A$ is a homogeneous vectorial differential operator with constant coefficients defined on $\R^n$ from $V$ to $W$, which are all finite dimensional linear spaces. 

Our main result is the following:
  \begin{bigthm}\label{thm:main_A-free}
Let $\A$ satisfy the constant rank \eqref{eq:CR} condition below, $f\colon \Omega\times V\to\R$ satisfy the  growth, smoothness, and strong $\A$-quasiconvexity  {assumptions} 
\textup{\ref{it:fgrowth}--\ref{it:frecession}} below. Then for any local minimizer $v$ of \eqref{eq:P_Omega_Afree}, there exists an open set $\Omega'\subset\Omega$  such that $\mathscr{L}^n(\Omega\setminus \Omega^{\prime})=0$ and $v\in \hold_{\locc}^{0,\alpha}(\Omega')$ for any $\alpha\in(0,{\tfrac12})$.
\end{bigthm}

We next clarify our assumptions.
The differential operator $\A$ has characteristic polynomial $\A(\xi)\in \lin (V,W)$ for $\xi\in\R^n$ and  will be assumed to have \textit{constant rank}
\begin{align}\tag{CR}\label{eq:CR}
    \rank\A(\xi)=\text{const.}\quad\text{for }\xi\in\R^n\setminus\{0\}.
\end{align}
The map $f\colon \Omega\times V\to\R$ will be assumed continuous of linear growth, so that we can define
\begin{align*}
    \int_\Omega f(x,v)=\int_\Omega f\left(x,\frac{\dif v}{\dif\mathscr L^n}(x)\right)\dif x+\int_\Omega f^\infty \left(x,\frac{\dif v}{\dif |v|}(x)\right)\dif v^s(x),
\end{align*}
where $v=\dfrac{\dif v}{\dif \mathscr L^n}\mathscr L^n\mres \Omega+v^s$ is the Radon--Nikod\'ym decomposition of $v$ with respect to $\mathscr L^n$, so $v^s\perp\mathscr L^n$, and $|v|$ denotes the total variation measure of $v$. Here, we define the \textit{recession function} $f^\infty\colon\Omega\times V\to \R$  as
\begin{align*}
    f^\infty(x,z)=\limsup_{\substack{ x'\to x,\,t\to\infty,\,z'\to z}} t^{-1}{f(x',tz')}.
\end{align*}

Particularly relevant to minimization are quasiconvex integrands \cite{Morrey52}, and more generally \textit{$\A$-quasiconvex} \cite{FonMul99}, i.e., $f\colon V\to \R$ such that for a cube $Q\subset\R^n$,
\begin{align*}
    f\left(\fint_Q v(x)\dif x\right)\leq \fint_Q f(v(x))\dif x\quad\text{for }v\in \hold^\infty_{\mathrm{per}}(Q,V),\text{ with }\A v=0.
\end{align*}
\par We now list our 
assumptions on  integrands $f\col \Omega \times V \to \R$, which lead to regularity: 
    \begin{enumerate}[label=\rm(\textbf{H\theenumi}),leftmargin=3\parindent, itemsep=.2em]
        \item\label{it:fgrowth} $\abss{f(x,z)}\leq L(1+\abss{z})$ for any $(x,z)\in \Omega \times V$, where $L>0$;
        \item\label{it:fcontinuity} For any $x\in \Omega$ the function $f(x,\cdot)$ is $\CC^{2}$, and 
        the functions $f(\cdot, \cdot)$, $\partial_z f(\cdot,\cdot)$ and $\partial_z^2f(\cdot,\cdot)$ are all jointly continuous on $\Omega \times V$ with
            \begin{equation*}
                \abss{\partial_z f(x,z)-\partial_z f(y,z)}\leq L \abss{x-y}
            \end{equation*} for any $z\in V$ and any $x,y\in \Omega$;
        \item\label{it:fqcstrong} The function $f$ is {\it strongly $\A$-quasiconvex} in the sense that $f(x,\,\cdot\,)-\ell E$ is $\A$-quasiconvex with some $\ell>0$ for $\mathscr L^n$-a.e $x\in\Omega$. 
        \item\label{it:frecession} The recession function $f^{\infty}(x,z)$ is continuous in $x$ for each $z \in V$.
\end{enumerate}
Here $E\colon V\to\R$ is the reference integrand $E(z)=\sqrt{1+|z|^2}-1$.

As in \cite{LiRaita24}, we reduce this study to  functionals defined on potential operator level
\begin{align}\label{eq:P_Omega}
 \mathcal F(u,\Omega)=\int_\Omega f(x,\B u)\quad\text{for }u\in \mathscr{D}'(\Omega,U)\text{ such that } \B u\in \mathcal M(\Omega,V),
\end{align}
where $\B$ is a homogeneous differential operator with constant coefficients defined on $\R^n$ from $U$ to $V$.  The notion of \textit{local minimality} is given by
\begin{align}\label{eq:local_min_B'}
    \mathcal{F}(u,\Omega)\leq \mathcal F(u+\phi,\Omega)\quad\text{for }\phi\in \hold_c^\infty(\Omega,U).
\end{align}

We will  use  the  notion of $f\colon V\to\R$ being
 \textit{quasiconvex with respect to $\B$-gradients}, i.e., for a cube $Q\subset \R^n$ and any $z\in V$,
\begin{align*}
    f\left(z\right)\leq \fint_Q f(z+\B u(x))\dif x\quad\text{for any }u\in \hold_c^\infty(Q,U).
\end{align*}
See \cite{Raita19,harkonen2024syzygies} for the connection between this notion and $\A$-quasiconvexity.

In this case, we replace assumption \ref{it:fqcstrong} with
\begin{enumerate}[label=(\rm\textbf{H3}$^{\prime}$),leftmargin=3\parindent, itemsep=.2em]
        \item\label{it:fBqcstrong}  The function $f$ is {\it strongly quasiconvex with respect to $\B$-gradients}, i.e., $f(x,\,\cdot\,)-\ell E$ is quasiconvex with respect to $\B$-gradients  for $\mathscr L^n$-a.e $x\in\Omega$.
    \end{enumerate}

Our main result in this direction is as follows:
\begin{bigthm}\label{thm:main}
   Let $\B$ be a $k$-th order differential operator satisfying \eqref{eq:CR}. Suppose that $f\colon \Omega\times V\to\R$ satisfies the  growth, smoothness, and strong quasiconvexity {assumptions} \textup{\ref{it:fgrowth}}, \textup{\ref{it:fcontinuity}},
   \textup{\ref{it:fBqcstrong}}, and \ref{it:frecession}.
   Then for any local minimizer $\B u $ of \eqref{eq:P_Omega}, there is an open subset $\Omega'\subset\Omega$ such that $\mathscr{L}^n(\Omega\setminus \Omega^{\prime})=0$ and $\B u\in \CC_{\locc}^{0,\alpha}(\Omega')$ for any $\alpha\in(0,1)$. Moreover, for each $\omega\Subset\Omega$, there exists $\tilde u\in \sobo^{k-1,1}(\Omega)$ such that $\B u = \B \tilde{u}$ in $\omega$ and $\tilde u\in \CC_{\locc}^{k,\alpha}(\omega\cap\Omega')$. 
\end{bigthm}

There are three major sources of difficulty in establishing Theorems~\ref{thm:main_A-free} and~\ref{thm:main}: 
\begin{enumerate}
    \item\label{itm:difficulty_a} the continuity of minimizers $v$ or $\B u$ is inferred from a non-unique representative $u$, in sharp contrast to the classic case $\B=D$;
    \item\label{itm:difficulty_c} no reduction to the more familiar case when $\B=D^k$ is possible due to Ornstein's non-inequality \cite{ornstein1962non,conti2005new,KK16,kazaniecki2017anisotropic}.
     \item\label{itm:difficulty_b} no basic implementation of known methods of linearization and approximation with solutions of linear elliptic systems can be performed.
\end{enumerate}

Difficulty~\ref{itm:difficulty_a} reflects the complete failure of coercivity of $\mathcal F$ in \eqref{eq:P_Omega} in any variation of the space
$
    \bv^{\B}(\Omega)\coloneqq \{u\in \sobo^{k-1,1}(\Omega,V)\colon \B u\in \mathcal M(\Omega,W)\}
$
\cite{gmeineder2019embeddings,raictua2019critical,raita2020continuity,breit2020trace,gmeineder2021limiting,gmeineder2024boundary}.
Coercivity holds for  \textit{elliptic} $\B$, i.e.,  $\ker \B(\xi)=\{0\}$ for $\xi\in\R^n\setminus\{0\}$, in which case the estimate
\begin{align}\label{eq:intro_ell_est}
\|\phi\|_{\sobo^{k-1,1}(\Omega)}\leq c\|\B \phi\|_{\lebe^1(\Omega)}\qquad\text{for }\phi\in \hold_c^\infty(\Omega,U)
\end{align}
holds for the variations $\phi$ in \eqref{eq:local_min_B'}.
Thus this difficulty is not present in the classic gradient case $\B=D$ in  neither the superlinear growth case \cite{Evans86,AF87,AF89,GiaMod86,marcellini1989regularity,CFM98,diening2012partial} or the linear growth one \cite{GMS79,GmeKri_BV}. In the set-up of general elliptic operators, strengthened versions of \eqref{eq:intro_ell_est} were crucial for contributions for superlinear \cite{CG22,Schiffer24} and linear growth integrands \cite{Federico2023,GK19,Gme20,Gme21,BK22}.

Estimate~\eqref{eq:intro_ell_est} fails brutally for constant rank operators that are not elliptic: there exist nonzero $\phi\in\CC_{\rmc}^\infty(\Omega,U)$ such that $\B \phi=0$. Examples are $\phi=\C \psi$ for nonzero $\psi$ in $\CC_{\rmc}^\infty(\Omega,U)$, where $\C\neq0$ is a potential operator for $\B$ \cite{Raita_pot_new}. In \cite{LiRaita24} this difficulty was overcome for integral functionals \eqref{eq:P_Omega} of superlinear growth by showing that for each subdomain $\omega\Subset\Omega$, one can find a representative $u$ of a local minimizer $\B u$ such that $\C^* u=0$ in $\omega$, making local estimates in the spirit of \eqref{eq:intro_ell_est} available.

In fact, in \cite{LiRaita24} it was possible to derive Korn-type inequalities for  $1<p<\infty$  such as
\begin{align}\label{eq:korn}
    \|D^k\phi\|_{\lebe^p(\Omega)}\leq c\|\B \phi\|_{\lebe^p(\Omega)}\qquad\text{for }\phi\in \hold^\infty_c(\Omega,U)\text{ such that }\C^*\phi=0
\end{align}
and stronger modular formulations \cite[Proposition~5.1]{LiRaita24}. Such estimates are known to fail in the limiting case $p=1$ by Ornstein's non-inequality \cite{KK16}. In particular, not only do we lack the higher integrability available in the superlinear growth case \cite{LiRaita24} and for certain convex problems of linear growth \cite{Gme20,beck2024gradient}, but it is also not possible to reduce to a more familiar situation of functionals defined on $\bv^k$ (Difficulty~\ref{itm:difficulty_c}).

Difficulty~\ref{itm:difficulty_b} refers to the fact that common strategies to establish partial regularity results such as Theorem~\ref{thm:main} involve  local approximation of the minimizers with solutions to linear elliptic systems. These were implemented indirectly by 
 blow-up arguments \cite{Evans86,AF87,CFM98} or using the $A$-harmonic approximation method  \cite{DGG00,DGK05}. 
 The $A$-harmonic approximation lemma can also be proved directly \cite{diening2012partial,Cruz-UribeDiening2019} by using explicit solvability of linearized elliptic systems. In our case, if $\B$ is not elliptic, it is unclear  if an auxiliary  elliptic system as in \cite{GK19,Federico2023} can even be formulated.

Indirect arguments require passing to the limit in the nonlinearity $E(\cdot)$, a difficulty already seen in the subquadratic setting, in which both strategies \cite{CFM98,DGK05} crucially use a Poincar\'e-Sobolev inequality adapted for  similar nonlinear  quantities. 
Such estimates are well understood for $\B=D^k$; in the superlinear growth case, which can be reduced to $\B=D^k$ using Korn-type inequalities \cite[Proposition~5.1]{LiRaita24}. 

Establishing a linear growth analogue is challenging and thus our main new tool:

\begin{bigthm}\label{thm:main_chris'_inequ}
    Suppose that $\B, \C$ are constant rank operators such that $\ima \C(\xi) = \ker \B(\xi)$ for all $\xi \in \mathbb R^n\setminus\{0\}$.
    Let $u \in \mathscr D'(B_R,U)$ such that $\B u \in \mathcal{M}(B_R,V)$ and $\C^{\ast}u = 0$ in $B_R$, and fix $\theta \in (0,1)$. Then there exists $\pi \in \CC^{\infty}(B_{\theta R},U)$ such that $\B \pi = 0$, $\C^{\ast}\pi=0$ in $B_{\theta R}$ and for all $1 \leq p < \frac{n}{n-1}$ we have the modular estimate
\begin{align}\label{eq:main_chris'_inequ}
    \sum_{j=0}^{k-1}  \left(\int_{B_{\theta R}} E\bigg(\frac{D^j(u-\pi)}{R^{k-j}}\bigg)^p \,\dd x\right)^{\frac1p} &\leq c\int_{B_R} E(\B u), 
\end{align} 
where $c=c(n,q,\theta,\B,\C)>0$ and the association $u \mapsto \pi$ is linear.
\end{bigthm}
Here the balls are concentric and recall that $E(z)=\sqrt{1+|z|^2}-1$. 
Theorem~\ref{thm:main_chris'_inequ} is new for elliptic operators $\B$ (i.e., when $\C=0$), in which case the kernel of $\B$ may be infinite dimensional  in general, so it is unclear how to construct  $\pi$ such that \eqref{eq:main_chris'_inequ} holds (cf. \cite{breit2020trace,gmeineder2019embeddings}).
This result is to be contrasted with \cite{Federico2023}, where the open mapping theorem is employed via a qualitative solvability result, which gives an analogous norm estimate ({i.e.,} replacing $E$ by $\lvert\,\cdot\,\rvert$).
In our setting, such linear techniques do not apply due to the nonlinear nature of the modular estimate. Instead we employ a  version of the qualitative solvability theory due to H\"ormander \cite{Hormander2000}, adapted to the $\mathrm{H}^s$-scale. Estimates implied by such general theory however are highly suboptimal, so we combine them with the ellipticity of the pair $(\B,\C^{\ast})$. The approach thus differs completely from \cite[Proposition 5.1]{LiRaita24}, as we must allow for lower order terms in arbitrarily negative scales. 

Our results also apply in the superlinear growth case, {substantially} refining \cite[Proposition 5.1]{LiRaita24}: if we assume   $\B u \in \lebe^p(B_R,V)$ for some $1<p<\infty$
in the setting of Theorem~\ref{thm:main_chris'_inequ}, then we obtain $u-\pi \in \sobo^{k,p}(B_{\theta R},V)$  and the  modular estimate
\begin{equation}\label{eq:vp_intro}
    \sum_{j=0}^k  \int_{B_{\theta R}} \left\lvert V_p\left(\frac{D^j(u-\pi)}{R^{k-j}}\right)\right\rvert^2 \,\dd x \leq c \int_{B_R} \lvert V_p(\B u)\rvert^2\,\dd x,
\end{equation}
where $V_p(z)=(1+\lvert z \rvert^2)^{\frac{p-2}4}z$. A precise statement is given in Theorem~\ref{thm:p_korn}.

One would expect by analogy with Poincar\'e's inequality, that  Inequality~\eqref{eq:main_chris'_inequ} also holds for $\theta=1$. This turns out to be an outstanding open problem even for linear $\lebe^p$-estimates:
\begin{openproblem}
    Let $\B$ be a $k$-th order constant rank operator, $1\leq p\leq \infty$, and $j=1,\ldots,k-1$. Does the following estimate hold?
    \begin{align*}
        \inf\left\{\|D^j(u-\pi)\|_{\lebe^p(B)}\colon \pi\in\mathscr D'(B,U)\text{ with }\B\pi=0\right\}\leq c\|\B u\|_{\lebe^p(B)}\quad\text{for }u\in \hold^\infty(\bar B,U).
    \end{align*}
\end{openproblem}
One might expect that this result is classical. 
In fact, all reasonable variants of this question  (such as with $j=k$ in the left-hand-side or with $\lebe^q$-norms on the left-hand-side for $q$ below the Sobolev exponent of $p$) are open in general--even in the class of elliptic operators. Partial results or examples can be found in \cite{gmeineder2019embeddings,arroyo2021new,iwaniec1999nonlinear}. Interestingly, the rank condition is necessary in some instances \cite{guerra2020necessity}.

\subsection*{Further technical remarks}

Theorem~\ref{thm:main_A-free} will follow from Theorem~\ref{thm:main} by using local representations of $\A$-free fields taking the form $v=\B u+S$, where $\B$ is a potential operator of $\A$ (see \cite{Raita19,Raita_pot_new}) and $S$  is a smooth vector field satisfying the elliptic system $\A S=0,\,\B^*S=0$. This idea has origins in \cite{Murat81,FonMul99} and was used recently in \cite{Raita19,GR22,KriRai22,guerra2022compensated,LiRaita24} to prove various properties related to variational integrals defined on $\A$-free vector fields, such as \eqref{eq:P_Omega_Afree}.
We remark that if $\B$ arises as a potential operator of $\A$, then \ref{it:fqcstrong} and \ref{it:fBqcstrong} are equivalent by \cite[Section 3]{Raita19}.

While Theorems~\ref{thm:main_A-free} and \ref{thm:main} are already new in the autonomous case when $f=f(z)$, our treatment of the non-autonomous case is necessitated by the nature of this reduction from $\A$-free fields to $\B$-gradients. 
Indeed writing $v = \B u + S$ we obtain
\begin{equation}\label{eq:autonomous_reduction}
    \int_{\Omega} f(v) = \int_{\Omega} f(S(x)+\B u),
\end{equation}
which reduces the minimization of \eqref{eq:P_Omega_Afree} to that of \eqref{eq:P_Omega}, at the expense of introducing an additional $x$-dependence.
While not our original aim, our treatment of the $x$-dependence  is new in the gradient setting--i.e., Theorem~\ref{thm:main} is new for $\B=D$.

Also, while it is a routine exercise to relax \ref{it:fcontinuity} to allow for integrands which are $\beta$-H\"older continuous in $x$ for any $\beta \in (0,1)$, we have refrained from treating this generality to simplify the exposition.
We also emphasize that the condition \ref{it:frecession} arises only due the $x$-dependence, and is vacuous for autonomous integrands; related conditions can be found in \cite{KriRin10a,ARDPR}.

In contrast to \cite{LiRaita24} where the superlinear growth case is treated, we establish partial regularity using the method of $A$-harmonic approximation. 
This technique was first used in the context of geometric problems \cite{Simon1983,Simon1996}, which was later adapted to the non-parametric setting in \cite{DuzaarGrotowski2000,DGG00,DuSt02,DGK05,diening2012partial}. 
We remark however, that equipped with Theorem~\ref{thm:main_chris'_inequ}, one could equally employ a blow-up argument as in \cite{LiRaita24} to establish Theorems~\ref{thm:main_A-free} and \ref{thm:main}.

Under constant rank constraints however, the method of $A$-harmonic approximation 
requires substantial modifications.
The main difficulty here is twofold, one being the solvability of elliptic systems as \eqref{eq:Aharmonic} under suitable boundary conditions, and the other being the lack of a suitable Poincar\'e inequality in the modular form. These are overcome by an $\lebe^2$-based solvability result (Lemma \ref{lem:l2_linear_solvability}) and Theorem \ref{thm:main_chris'_inequ}, respectively.

It appears that one cannot directly reduce from Theorem~\ref{thm:main_A-free} to Theorem~\ref{thm:main} using \eqref{eq:autonomous_reduction}, unless one additionally imposes that $\partial_z^2f$ is bounded. For this reason in \cite{LiRaita24} a controlled $p$-growth is assumed, however we will show this additional condition is unnecessary for partial regularity.
We found however, that this leads to $\alpha$-H\"older continuity of $v$ in the restricted range $\alpha < \frac12$.

While there has been work in the setting of exact or closed differential forms \cite{Dacorogna_JEMS,SIL3}, the regularity theory has until now been confined to the superlinear growth setting and convex integrands; see for instance \cite{uhlenbeck1977regularity,hamburger1992regularity,beck2013regularity}. We therefore believe that Theorems~\ref{thm:main_A-free} and \ref{thm:main} are new in this setting, even for $\A = \di$ and $\B = \curl$.

Finally, there are many natural extensions which arise in the context of variational problems with constant rank PDE constraints.
Further investigations may include considering the setting of $(p,q)$-growth \cite{Schmidt2008,schmidt2009regularity,GmeinederKristensen2024}, degenerate problems \cite{DM04}, nonlinear potentials \cite{kuusi2016partial,defilippis2022quasiconvexity,defilippis2023singular}, singular set considerations \cite{KriMin07}, Sobolev regularity \cite{Bild03,BS13,}, general growth \cite{Marcellini1996}, etc.
Related works in the ($\mathbb C$-)elliptic setting include \cite{GK19,Gme20,beck2024gradient,EitlerLewintan2025} for Sobolev regularity, and \cite{CG22,Federico2023,BK22,Stephan2024} for partial regularity.

\subsection*{Organization of the paper}
In Section~\ref{sec:prel} we collect preliminaries regarding differential operators, and function spaces relevant for our problems. Section~\ref{sec:estimates} consists of the elaborate proof of Theorem~\ref{thm:main_chris'_inequ}. In Section~\ref{sec:partialreg} we prove the main partial regularity result, and in Section~\ref{sec:proof_of_main_thm_A-free} we will show how Theorems~\ref{thm:main_A-free} and \ref{thm:main} follow by means of a reduction argument.


\section{Preliminaries}\label{sec:prel}

\subsection{Notation}

\par Throughout this paper, we denote by $\Omega$a bounded open subset of $\R^n$ unless otherwise specified. The Lebesgue measure on $\R^n$ is denoted by $\Le^n = \abss{\,\cdot\,}$, and $\indi_{E}$ for any $E\subset \R^n$ is the indicator function of $E$. For a measure $\mu \in \mathcal M(\Omega)$, we will denote its Lebesgue--Radon--Nikod\'ym decomposition as $\mu = \mu^{a} \mathscr L^n + \mu^{s}$. If $\mu = \B u$ is a vectorial measure arising from a potential ({i.e.,} $\B$ is a differential operator), we will write $\mu^a = \B^a u$ and $\mu^s = \B^su$. For any $E\subset \R^n$ measurable with $\Le^n(E)\in (0,\infty)$ and any $f \in \lebe^1(E)$, the average notation is defined as follows
    \[ (f)_E= \fint_{E}f \,\dif x:= \frac{1}{\Le^n(E)}\int_E f\dif x. \] 
In particular, when $E= B_R(x)$ is a ball, we may abbreviate it as $(f)_{x,R}$. This notation also extends to vector-valued functions and measures.

\par We will assume familiarity with standard function spaces such as the Lebesgue and Sobolev spaces, along with the space of smooth and compactly supported functions $\CC^{\infty}_{\rmc}(\Omega)$ and the space of distributions $\mathscr{D}'(\Omega)$.
We will often omit the target of vector valued function spaces if it is clear from context, and write for instance $\lebe^p(\Omega) = \lebe^p(\Omega,V)$.
We will also use $\lebe^p_{\mathrm{loc}}(\Omega)$ for functions locally in $\lebe^p$ and similarly for other scales.

Throughout the text, $c$ will denote constants which change from line to line, whose dependencies may be specified by $c = c(\cdots)$. Sometimes we will write $c_1, c_2$, etc.\,to denote specific constants. We will also write $A \lesssim B$ to mean there exists a constant $c>0$ such that $A \leq cB$, and $A \sim B$ if $A \lesssim B$ and $B \lesssim A$.
\par Unless otherwise stated, $U, V, W$ will denote finite dimensional inner-product spaces over $\mathbb R$. Suppose that $X$ is such a space. For any $x,y \in X$, denote by $x\cdot y$ or $\brangle{x,y}$ the inner product of $x$ and $y$, and the induced norm by $\abss{\cdot}$. 

\par Given any two finite dimensional inner-product spaces $U,V$ as above, denote by $\lin^k(U,V)$ the space of $k$-linear maps from $U^k$ to $V$ (abbreviated as $\lin(U,V)$ if $k=1$). The space of symmetric $k$-linear maps from $U$ to $V$ is denoted by $\mathrm{SLin}^k(U,V)$. In particular, for any $u\in \CC^\infty(U,V)$, we have that $D^ku\in \mathrm{Slin}^k(U,V)$. Moreover, $y\otimes x^{\otimes k} \in \slin^k(U,V)$ is defined such that
    \[ (y\otimes x^{\otimes k})(u^1,\dots,u^k) = \brac{\prod_{i=1}^k \langle x,u^i \rangle}y, \quad \mbox{for any }u^1,\dots,u^k \in U. \]

\par The higher order gradients and exterior products above are connected by Fourier transform, 
which for $f \in \lebe^1(\R^n,V)$ is defined as
    \begin{equation}\label{eq:FT_definition} \mathscr{F} f (\xi) = \hat{f}(\xi):= \int_{\R^n} f(x)\mathrm{e}^{-\mathrm{i}2\pi  x\cdot \xi}\dif x, \quad \xi \in \R^n. \end{equation}
    If $u\in \mathscr S(\R^n,V)$, we have that $\widehat{D^ku}(\xi)=c_k\hat u(\xi)\otimes \xi^{\otimes k}$ for $\xi\in\R^n$ for a complex constant $c_k$ which we will often omit.

\subsection{Differential operators}\label{sec:differential}

\par We will work with homogeneous and constant coefficient differential operators
\begin{align*}
    \A= \sum_{|\alpha|=h} A_\alpha\partial^\alpha, \quad \B= \sum_{|\beta|=k} B_\beta\partial^\beta, 
\end{align*} 
where the multi-indicies $\alpha,\beta$ lie in $\N^n$, and $A_{\alpha} \in \lin(V,W)$, $B_{\beta} \in \lin(U,V)$ are linear maps.
The associated characteristic polynomials of $\A$ and $\B$ are respectively defined as
 
\begin{align*}
    \A(\xi) = \sum_{|\alpha|=h} A_\alpha\xi^\alpha, \quad \B(\xi)= \sum_{|\beta|=k} B_\beta\xi^\beta \quad\mbox{for $\xi \in \mathbb R^n$}, 
\end{align*}
which naturally arise under the Fourier transform.

Our primary assumption is that $\A$ and $\B$ are of constant rank \eqref{eq:CR}. An important subclass is that of \emph{(injectively-)elliptic} operators \cite{van2022injective}, i.e., operators $\B$ satisfying $\ker \B(\xi)=\{0\}$ for $\xi\in\R^n\setminus\{0\}$. Of particular importance to our work will be the following:

\begin{definition}
    Let $\A$ be a homogeneous partial differential with constant coefficients. We say that a homogeneous partial differential operator $\B$ is a \textbf{potential  operator} for $\A$ if 
    $
            \ker \A(\xi)=\rmim\B(\xi)\text{ for all }\xi\in\R^n\setminus\{0\}.
    $
\end{definition}
By \cite[Theorem 1.1]{Raita19}, we have that $\A$ admits a potential operator $\B$ if and only if $\A$ has constant rank, in which case $\B$ also has constant rank. We will also use work with a potential operator of $\B$, which we denote by $\C$. By the construction in \cite{Raita19}, we can assume that $\C$ is defined on $\R^n$ from $U$ to $U$ and is homogeneous of degree $l>k$. The inequality is justified by the fact that if $\C$ is a potential operator for $\B$, then so is $\Delta^m\C$. Therefore if $\B$ is itself a potential operator of $\A$, we obtain the exact sequence
\begin{equation}
    U \xrightarrow{\C(\xi)} U \xrightarrow{\B(\xi)} V \xrightarrow{\A(\xi)} W \quad\mbox{for each } \xi \in \mathbb R^n\setminus\{0\}.
\end{equation}

We define the \textit{wave cone} of $\A$ \cite{Tartar83,Murat81} as
\begin{equation}\label{eq:wave_cone}
    \Lambda_{\mathscr{A}} = \bigcup_{\xi \in \mathbb R^n\setminus\{0\}} \ker \mathscr{A}(\xi),
\end{equation}
which gives the directions of non-ellipticity of $\A$.
Therefore, if $\B$ arises as a potential operator of $\A$, we also have that
\begin{equation}
    \Lambda_{\mathscr{A}} = \bigcup_{\xi \in \mathbb R^n \setminus \{0\}} \ima \mathscr{B}(\xi).
\end{equation}

We will also make use of the \emph{spanning cone condition} \cite{FonMul99,GR22,KriRai22}, i.e., 
\begin{align}\tag{SC}\label{eq:SC}
    V =  \spann \Lambda_\A.
\end{align}

For a matrix $B \in \lin(U,V)$, we  define the Moore--Penrose generalized inverse $B^{\dagger}$ by 
\[ B^{\dagger}:= \left\{
    \begin{aligned}
        &(B\vert_{(\ker B)^{\perp}})^{-1},& &\mbox{on }\rmim B,\\
        &0,& &\mbox{on }(\rmim B)^{\perp}, 
    \end{aligned}\right. \]
which is the unique element in $\lin(V,U)$ satisfying
\begin{equation}
    BB^{\dagger} = \proj_{\ima B}, \quad B^{\dagger}B = \proj_{\ima B^{\ast}},
\end{equation}
where $\proj_M$ denotes the orthogonal projection onto a given subspace $M$, and $B^{\ast}$ denotes the adjoint of $B$.
We will apply this to $\B(\xi)$ in the Fourier space.
The following is a consequence of Decell's formula \cite{Decell1965}, which was used in \cite{Murat81,FonMul99,Raita19}.

\begin{lemma}\label{lem:dagger}
    Let $\B$ be an operator as above of order $k$ with constant rank. Then the mapping $\xi \mapsto \B^{\dagger}(\xi)$ is smooth on $\mathbb R^n \setminus \{0\}$ and is homogeneous of degree $-k$.
\end{lemma}

Any differential operator $\A$ as above can be written as $\A v=\di^h L v$, where $L\in \lin(V,\mathrm{SLin}^h(\R^n,W))$. 
Indeed if $v \in \mathscr D'(\R^n,V)$ we have
\begin{equation}\label{eq:div_expression}
    \A v= \sum_{|\alpha|=h} A_\alpha \partial^\alpha v = \sum_{|\alpha|=h} \partial^\alpha (A_\alpha v) =\di^h Lv,
\mbox{ where } Lz=(A_\alpha z)_{|\alpha|=h}\text{ for }z\in V.
\end{equation}
Here $\di^h$ is the formal adjoint operator of $D^h$, and in general the adjoint of $\A$ will be denoted by $\A^{\ast}$.

We will also need the following result concerning averages of $\B$-gradients.
\begin{lemma}\label{lem:B_poly}
    Let $u \in \sobo^{k-1,1}(B_R(x_0))$ such that $\B u \in \mathcal M(B_R(x_0))$.
    Then there exists a homogeneous polynomial $a \colon \mathbb R^n \to U$ of degree $k$ such that $\B a = (\B u)_{x_0,R}$.
\end{lemma}
\begin{proof}
    Write $\B u = T(D^ku)$ with $T \in \lin(\slin(\mathbb R^n,U),V)$.
    By \cite[Lemma~2.4]{LiRaita24} and by means of a density argument,   
    we have $z_0 \coloneqq (\B u)_{x_0,R} \in \ima T$. 
    Therefore we can find $\phi \in \operatorname{SLin}^k(\mathbb R^n,U)$ such that $T(\phi) = z_0$. We then define
    \begin{equation}
        a(x) = \sum_{\lvert\alpha\rvert=k} \phi_{\alpha} \frac{x^{\alpha}}{\alpha!}
    \end{equation}
    which gives the desired $k$-polynomial, since $\B a = T(D^ka) = T(\phi) = z_0$.
\end{proof}

\subsection{Integrands of linear growth}
\par Define the following reference integrand 
\begin{equation}
    E(z)\coloneqq \sqrt{1+\abss{z}^2}-1
\end{equation}
for any $z$ in any finite dimensional Banach space $(X,\lvert\cdot\rvert)$. We will record several elementary properties of this integrand, which can be found in \cite{GmeKri_BV}.

\begin{proposition}\label{prop:Eprop}
Then the function $E(\cdot)$ satisfies the following properties for any $z,w \in X$ and $t>0$:
\begin{enumerate}
    \item\label{it:Eest} $E(z) \sim \min\{\abss{z},\abss{z}^2\}$;
    \item\label{it:Ecvx} $E(\cdot)$ is convex;
    \item\label{it:Emulti} $E(tz)\lesssim \max\{t,t^2\}E(z)$;
    \item\label{it:Esum} $E(z+w) \lesssim E(z)+E(w)$.
\end{enumerate}
\end{proposition}

\begin{lemma}[Lemma 2.9 in {\cite{GmeKri_BV}}]\label{lem:Eint_est}
    For any $f \in \lebe^1(\Omega)$ with $\Omega \subset \R^n$ bounded, define $e:= \fint_{\Omega} E(\abss{f})\dif x$, then we have
        \begin{equation}
            \fint_{\Omega} \abss{f}\dif x \leq \sqrt{e^2+2e}.
        \end{equation} In particular, if $e\leq 1$, there holds
        \begin{equation}
            \fint_{\Omega} \abss{f}\dif x \leq \sqrt{3e}.
        \end{equation}
\end{lemma}
By mollification, the above lemma also holds if we replace $f$ with a measure $\mu \in \mathcal M(\Omega)$, 
where we write
\begin{equation}\label{Emu}
    \int_{\Omega} E(\mu) \coloneqq  \int_{\Omega} E(\mu^{\rma}) \,\dd x + \lvert \mu^{\rms}\rvert(\Omega).
\end{equation}

\par We will also record estimates and results concerning the integrand $f$. Here we only consider the $\A$-free framework (Theorem~\ref{thm:main_A-free}). The analogous auxiliary statements for the $\B$-framework (Theorem~\ref{thm:main}) follow analogously, but are easier to establish. This is so because in the $\B$-setting there is no spanning cone condition \cite[Lemma~2.4]{LiRaita24}.

So let $f\colon\Omega\times V\to\R$ be the integrand in \eqref{eq:P_Omega_Afree}. From \ref{it:fgrowth} and \ref{it:fqcstrong} it follows that
\begin{equation}\label{eq:fz_bounded}
    \lvert \partial_z f(x,z)v\rvert \leq cL\lvert v \rvert \quad\mbox{for all } x \in \Omega, \ z \in V,\ v \in \spann \Lambda_{\A}
\end{equation}
where the constant $c$ depends on $n$ and $\dim V$ only.
Indeed for each $(x,z) \in \Omega \times V$, we can apply \cite[Lemma 2.3]{KK16} to the mapping $v \mapsto f(x,z+v)$ restricted to $(\spann \Lambda_{\A}) \cap B_R(0)$ with $R\geq 1+\lvert z \rvert$.
Note that if \eqref{eq:SC} holds, then we simply have $\lvert \partial_zf(x,z)\rvert \leq cL$.

\par Since $\partial_z^2f(x,z)$ is jointly continuous by \ref{it:fcontinuity}, given any $B_R(x_0) \Subset \Omega$ and $M>0$ there exists $L_{x_0,R,M}>0$ and a modulus of continuity $\omega_{x_0,R,M} \colon \mathbb R_+ \to \mathbb R_+$ for which
\begin{align}
    \label{eq:loc_bound}\lvert \partial_z^2f(x,z) \rvert &\leq L_{x_0,R,M} ,\\
    \label{eq:loc_cont} \lvert \partial_z^2f(x,z)-\partial_z^2f(y,w) \rvert &\leq L_{x_0,R,M}\,\omega_{x_0,R,M}(\lvert x-y\rvert + \lvert z-w\rvert)
\end{align}
hold for all $x,y \in B_R(x_0)$ and $z,w \in V$ such that $\lvert z \rvert, \lvert w \rvert \leq M$.
Moreover, we can choose $\omega_{x_0,R,M}$ to be non-decreasing and concave such that $0 \leq \omega_{x_0,R,M} \leq 1$, $\omega_{x_0,R,M}(0)=0$ and $\omega_{x_0,R,M}(t) = 1$ for all $t>1$.

\par Given an integrand $f \colon \Omega \times V \to \mathbb R$ and $w \in V$, will also introduce the \emph{shifted integrand} which is defined for $x \in \Omega$ and $z \in V$ as \begin{equation}\label{eq:shifted_int} f_w(x,z):= f(x,z+w)-f(x,w)-\partial_z f(x,w)\cdot z.\end{equation} 
We record various properties of $f_w$ in the below two lemmas.

\begin{lemma}\label{lem:shifted_estimates} 
    Suppose that the integrand $f\col \Omega\times V \to \R$ satisfies \ref{it:fgrowth}, \ref{it:fcontinuity}, and \ref{it:fqcstrong}, and fix $M>0$.
    Then given any ball $B_R(x_0) \Subset \Omega$ and $w\in V$ with $\abss{w}\leq M$, consider the shifted integrand $f_w$ defined via \eqref{eq:shifted_int}.
    Then there exists $c_1=c_1(n,V,L,L_{x_0,R,M+1})$ such that for all $x \in B_R(x_0)$ and $z \in V$,
    \begin{equation}\label{eq:pshiftedup}
     \abs{f_w(x,z)\cdot v} \leq c_1E(z) \lvert v \rvert \quad\mbox{for all }v \in \spann \Lambda_{\A},
    \end{equation}
    and there is $c_2=c_2(\ell,M)$ such that for all $x \in B_R(x_0)$ we have
    \begin{align}
        & \partial_z^2 f(x,w)[v,v]\geq \frac{1}{c_2}\abss{v}^2 \quad\mbox{for all }v \in \Lambda_{\A}, \label{eq:ellipticity}\\
        & \int_{B} f_w(x,\varphi(y)) \dif y \geq \frac1{c_2} \int_{B} E(\varphi)\dif y \label{eq:pshiftedqc}
    \end{align} for any ball $B \subset \mathbb R^n$ and any $\varphi \in \CC_{\rmc}^{\infty}(B)$ satisfying $\A \varphi =0$.
    Moreover, by mollification, \eqref{eq:pshiftedqc} holds for all $\A$-free $\varphi \in \mathcal{M}_{\rmc}(B)$.
\end{lemma}

\par The estimate \eqref{eq:pshiftedup} can be proved by considering the two cases $\abss{z}>1$ and $\abss{z}\leq 1$ separately with \eqref{eq:loc_bound}. The other two follow from the strong quasiconvexity assumption \ref{it:fqcstrong} and the same proof in \cite[Lemma 4.1]{GmeKri_BV}.

\begin{lemma}\label{lem:shifted_continuity}
  Let $f\col \Omega \times V \to \R$ satisfy \ref{it:fgrowth} and \ref{it:fcontinuity}, fix $M>0$, and let $x_0 \in \Omega$, $R>0$ such that $B_R(x_0) \Subset \Omega$.
  Then for all $x \in B_R(x_0)$ and $w \in V$ such that $\lvert w \rvert \leq M$, the shifted integrand $f_{w}(x,z)$ defined as in \eqref{eq:shifted_int} satisfies
\begin{equation}
    \lvert (\partial_z^2f_{w}(x,z)(z-w) - \partial_zf_{w}(x,z))\cdot v \rvert \leq c \, \omega_{x_0,R,M+1}(\lvert z - w\rvert) \lvert z - w \rvert \lvert v \rvert
\end{equation}
for all $z \in V$ and $v \in \spann \Lambda_{\A}$,
where $c=c(n,V,L,L_{x_0,R,M+1})>0$, using the notation from \eqref{eq:loc_bound}, \eqref{eq:loc_cont}.
\end{lemma}
\begin{proof}
    Since $v \in \spann \Lambda_{\A}$, we can assume that \eqref{eq:SC} holds to simplify notation. The argument is analogous to \cite[Lemma~4.2]{GmeKri_BV}.
    If $\lvert z - w \rvert \leq 1$, then by the fundamental theorem of calculus and \eqref{eq:loc_cont} we have
    \begin{equation}
    \begin{split}
        \lvert\partial_z^2f_{w}(x,z)(z-w) - \partial_zf_{w}(x,z)\rvert 
        &= \bigg\lvert \int_0^1 (\partial_z^2f(x,z) - \partial_z^2f(x,w+t(z-w)))(z-w) \,\dd t\bigg\rvert\\
        &\leq L_{x_0,R,M+1} \omega_{x_0,R,M+1}(\lvert z - w\rvert) \lvert z - w\rvert,
        \end{split}
    \end{equation}
    If $\lvert z-w\rvert >1$, then using \eqref{eq:fz_bounded} and \eqref{eq:loc_bound} we can estimate
    \begin{equation}
        \begin{split}
            \lvert\partial_z^2f_{w}(x,z)(z-w) - \partial_zf_{w}(x,z)\rvert 
            &\leq L_{x_0,R,M+1} \lvert z-w\rvert + cL \lvert z-w\rvert \\
            &\leq (L_{x_0,R,M+1} + c L)\:\! \omega_{x_0,R,M+1}(\lvert z -w\rvert) \lvert z -w\rvert,
        \end{split}
    \end{equation}
    where we used that $\omega_{x_0,R,M+1}(\lvert z -w\rvert) = 1$.
    By combining the two cases, the result follows.
\end{proof}

The condition \ref{it:frecession} will be used in the following variant of \cite[Theorem 4]{KriRin10a}, adapated to the setting of $\B$-gradients. We note the proof is identical, except we use the $\Lambda_{\A}$-convexity of $f$ and the generalization \cite{dephilippis2016structure} of Alberti's rank-one theorem \cite{Alberti1993}. Related results can be found in \cite{ARDPR}.

\begin{lemma}\label{lem:strict_cont}
    Let $f \colon \Omega \times V \to \mathbb R$ be jointly continuous satisfying \ref{it:fgrowth}, \ref{it:frecession}, and such that $f(x,\cdot)$ is $\Lambda_{\A}$-convex for each $x \in \Omega$.
    Suppose that $(u_k)_k, u$ are in $\mathscr{D}'(\Omega)$ and $\B u_k \to \B u$ area-strictly in $\Omega$, then we have
    \begin{equation}
        \lim_{k \to \infty} \int_{\Omega}f(x,\B u_k) = \int_{\Omega}f(x,\B u).
    \end{equation}
\end{lemma}
Notice that the condition \ref{it:fqcstrong} implies that $f(x,\cdot)$ is $\Lambda_{\A}$-convex for each $x\in \Omega$ by \eqref{eq:ellipticity}. Thus, the above lemma applies when $f$ satisfies \ref{it:fgrowth}--\ref{it:frecession}.

\begin{proof}[Sketch of proof]
  The first half of the proof in \cite[Theorem 4]{KriRin10a}, which only assumes the linear growth bound and continuity of $f$, applies in this setting to give
  \begin{equation}\label{eq:astrict_upper_lower}
      \begin{split}
          &\int_{\Omega} f(x,\B^au) \,\dd x + \int_{\Omega} h_f\bigg(x,\frac{\B^su}{\lvert \B^su\rvert}\bigg) \,\dd\lvert \B^su \rvert \\
          &\quad\leq \liminf_{k \to \infty} \int_{\Omega} f(x,\B u_k) 
          \leq \limsup_{k \to \infty} \int_{\Omega} f(x,\B u_k) \\
          &\qquad\qquad\leq \int_{\Omega} f(x,\B^au) \,\dd x - \int_{\Omega} h_{-f}\bigg(x,\frac{\B^su}{\lvert \B^su\rvert}\bigg) \,\dd\lvert \B^su \rvert,
      \end{split}
  \end{equation}
  where
  \begin{equation}
      h_f(x,z) \coloneqq \liminf_{ x'\to x,\,t\to\infty,\,z'\to z} \frac{f(x',tz')}t.
  \end{equation}
  For any $z \in \Lambda_{\A}$, since $t \mapsto f(x,tz)$ is convex and since $\partial_zf(x,z)$ is uniformly bounded by \eqref{eq:fz_bounded}, by a similar reasoning as in \cite[(11)]{KriRin10a} we have
  \begin{equation}
      f^{\infty}(x,z) = h_f(x,z) = -h_{-f}(x,z) \quad\mbox{for all } (x,z) \in \Omega \times \Lambda_{\A}.
  \end{equation}
  Since by \cite[Theorem 1.1]{dephilippis2016structure} we have $\frac{\B^s u}{\lvert \B^s u\rvert}(x) \in \Lambda_{\A}$ for $\lvert \B^su\rvert$-a.e.\,$x\in\Omega$, the upper and lower limits in \eqref{eq:astrict_upper_lower} coincide.
\end{proof}

We also show that in our setting, \ref{it:frecession} is implied by the existence of a strong recession function.

\begin{lemma}\label{lem:recession_lipschitz}
    Let $f$ satisfy \textup{\ref{it:fgrowth}}, \textup{\ref{it:fcontinuity}} and assume the strong recession function exists in that
    \begin{equation}\label{eq:strong_recession}
    f^{\infty}(x,z) = \lim_{
    \substack{t \to \infty,\\ x' \to x, z'\to z}
    } \frac{f(x',tz')}{t}
\end{equation}
    exists for all $(x,z) \in \Omega \times V$.
    Then the recession function $f^{\infty}$ satisfies the Lipschitz estimate
    \begin{equation}
        \lvert f^{\infty}(x,z) - f^{\infty}(y,z) \vert \leq L \lvert x-y \rvert \lvert z \rvert \quad\mbox{for all } x,y \in \Omega, z \in V.
    \end{equation}
    In particular, $f$ satisfies \textup{\ref{it:frecession}}.
\end{lemma}
\begin{proof}
    Given $z\in V$ and $x\in \Omega$, by the $1$-homogeneity of $f^{\infty}(x,\cdot)$ we can write 
    \begin{equation}\label{eq:recession_limit} f^{\infty}(x,z) =f^{\infty}(x,2z)-f^{\infty}(x,z) = \lim_{t\to \infty}\frac{1}{t}(f(x,2tz)-f(x,tz)),  \end{equation} where the term in the limit can be written as 
    \[\frac{1}{t}(f(x,2tz)-f(x,tz)) =\int_0^1 \partial_z f(x,tz+stz)\cdot z\dif s. \] For any $x,y \in \Omega$, the Lipschitz continuity of $\partial_z f(\cdot,\cdot)$ in the first argument as assumed in \ref{it:fcontinuity} implies 
    \begin{align} 
    \begin{split}
    &\ \abs{\frac{1}{t}(f(x,2tz)-f(x,tz)) - \frac{1}{t}(f(y,2tz)-f(y,tz))} \\
    &\quad=\ \abs{\int_0^1 (\partial_z f(x,tz+stz)-\partial_z f(y,tz+stz))\cdot z\dif s}\\
    &\quad\leq \int_0^1 L\abss{x-y}\abss{z}\,\dd s = L\abss{x-y} \abss{z}.
    \end{split}
    \end{align} 
    Now sending $t \to 0$ and using \eqref{eq:recession_limit}, we obtain
    \[\abss{f^{\infty}(x,z)-f^{\infty}(y,z)} \leq L\abss{x-y}\abss{z}\]
    as required.
\end{proof}

\begin{remark}
    The existence of local minimisers can be established in the setting of $\A$ satisfying \eqref{eq:CR} and \eqref{eq:SC}, and $f$ satisfying \ref{it:fgrowth}--\ref{it:fqcstrong}, provided we moreover assume the existence of a strong recession function as in \eqref{eq:strong_recession} and the coercivity estimate
    \begin{equation}\label{eq:coerc_x-dependent}
        \int_{\Omega} f(x,v_0+\varphi)\,\dd x \gtrsim \int_{\Omega} \lvert \varphi \rvert \,\dd x -\int_{\Omega} 1 + \lvert v_0\rvert \,\dd x,
    \end{equation}
    valid for all $v_0 \in \lebe^1(\Omega,V)$ and $\varphi \in \CC_{\rmc}^{\infty}(\Omega,V)$ satisfying $\A \varphi = 0$. 
    This follows by a similar argument as in \cite[Section 4]{LiRaita24} by employing the Direct Method, using the lower semicontinuity results in \cite[Proposition~1.1]{KriRai22} (see also \cite{FonMul99,ARDPR}). In the case of autonomous integrands, coercivity estimates  \eqref{eq:coerc_x-dependent} follow from \ref{it:fqcstrong} \cite{ChenKris17,LiRaita24}.
\end{remark}

\subsection{Function spaces}

When working with the reference integrand $E(t)$, we will need several estimates in the modular scales associated to $E(t)$ and $E(t)^q$ with $q>1$.
For this reason it will be convenient in Section~\ref{sec:estimates} to work in the Orlicz scales, so we will recall the relevant notions here.

Briefly, an \emph{$N$-function} $\varphi \colon [0,\infty) \to [0,\infty)$ is an increasing, continuous, convex function for which
\begin{equation}
    \lim_{t \to 0} \frac{\varphi(t)}t = \lim_{t \to \infty} \frac{t}{\varphi(t)} = 0.
\end{equation}
We say $\varphi \in \Delta_2$ if there is $c>0$ such that $\varphi(2t) \leq c\varphi(t)$, and the least such $c$ is denoted by $\Delta_2(\varphi)$. We say $\varphi \in \nabla_2$ if the \emph{conjugate} of $\varphi$,
\begin{equation}
    \varphi^{\ast}(s) = \sup_{t \geq 0}(st - \varphi(t)),
\end{equation}
lies in $\Delta_2$, and write $\nabla_2(\varphi) = \Delta_2(\varphi^{\ast})$.
Given an $N$-function $\varphi \in \Delta_2 \cap \nabla_2$ and an open set $\Omega \subset \mathbb R^n$, the Orlicz space $\lebe^{\varphi}(\Omega)$ denotes the space of locally integrable functions $u$ such that $\varphi(\lvert u \rvert)$ is integrable on $\Omega$, which we equip Luxemburg norm
\begin{equation}\label{eq:luxemburg}
    \lVert u \rVert_{\lebe^{\varphi}(\Omega)} = \inf\bigg\{\lambda>0 : \int_{\Omega} \varphi\bigg(\frac{\lvert u \rvert}{\lambda}\bigg) \,\dd x \leq 1 \bigg\}.
\end{equation}
We also have $\lebe^{\varphi}(\Omega)^{\ast} \cong \lebe^{\varphi^{\ast}}(\Omega)$.

Given a Fourier multiplier $\bm{m} \in \lebe^{\infty}(\mathbb R^n)$, we will also need estimates for mappings of the form $u \mapsto \mathscr{F}^{-1}(\bm{m}(\xi)\hat u(\xi))$.
Here the Fourier transform is viewed as an isomorphism on the space of tempered distributions $\mathscr{S}'(\mathbb R^n)$, which is extended from $\mathscr{S}(\mathbb R^n)$ as defined using \eqref{eq:FT_definition} by duality.
In these scales the H\"ormander--Mikhlin multiplier theorem holds; namely if $\bm{m} \in \lebe^{\infty}(\mathbb R^n)$ satisfies the hypotheses of \cite[Theorem~7.9.5]{Hormander1} 
(for our purposes it suffices that $\bm{m}$ is smooth in $\R^n \setminus \{0\}$ and $\lvert\partial^{\alpha}\bm{m}(\xi)\rvert \leq c \lvert \xi \rvert^{-\lvert\alpha\rvert}$ for all $\lvert\alpha\rvert<n/2+1$),
we know that $u \mapsto \mathscr{F}^{-1}(\bm{m} (\xi)\hat{u}(\xi))$ is bounded on $\lebe^p(\mathbb R^n)$ for all $1<p<\infty$.
By a Marcinkiewicz-type interpolation result, namely by combining \cite[Theorems~11.7, 11.8]{Maligranda1989} and \cite[Theorem~2]{Zygmund1956}, we infer that
\begin{equation}\label{eq:orlicz_hm}
    \lVert \mathscr{F}^{-1}(\bm{m} \hat{u}) \rVert_{\lebe^{\varphi}(\mathbb R^n)} \leq c(n,\bm{m} ,\Delta_2(\varphi),\nabla_2(\varphi)) \lVert u \rVert_{\lebe^{\varphi}(\mathbb R^n)} \quad\mbox{for all } u \in \lebe^{\varphi}(\mathbb R^n).
\end{equation}
The Orlicz-Sobolev spaces $\sobo^{k,\varphi}(\Omega)$, $\dot{\sobo}^{k,\varphi}(\mathbb R^n)$ are defined analogously to the $\lebe^p$ versions for all $k \in \mathbb Z$ and $\varphi \in \Delta_2\cap\nabla_2$; if $k <0$, elements of $\sobo^{k,\varphi}(\Omega)$ are understood as distributions $v = \sum_{\lvert\alpha\rvert \leq -k} \partial^{\alpha}v_{\alpha}$ in $\mathscr{D}'(\Omega)$ with $v_{\alpha} \in \lebe^{\varphi}(\Omega)$, which we equip with the associated norm
\begin{equation}
    \lVert v \rVert_{\sobo^{k,\varphi}(\Omega)} = \inf \sum_{\lvert\alpha\rvert\leq -k} \lVert v_{\alpha} \rVert_{\lebe^{\varphi}(\Omega)},
\end{equation}
where we take the infimum over all such representations.
The definition of $\dot{\sobo}^{k,\varphi}(\mathbb R^n)$ is similar, except that we sum over $\lvert\alpha\rvert=-k$.
With respect to the distributional pairing $\langle \cdot,\cdot\rangle$, these spaces satisfy the duality relations $\sobo_0^{m,\varphi}(\Omega)^{\ast} \cong \sobo^{-m,\varphi^{\ast}}(\Omega)$ and $\dot{\sobo}^{m,\varphi}(\mathbb R^n)^{\ast} \cong \dot{\sobo}^{-m,\varphi^{\ast}}(\mathbb R^n)$ for each $m \in \mathbb N$. 
We also have the following Poincar\'e inequalities:
\begin{alignat}{3}
    \label{eq:orlicz_poincare_1}\lVert u \rVert_{\lebe^{\varphi}(B_1)} &\leq c(n,\Delta_2(\varphi)) \lVert D u \rVert_{\lebe^{\varphi}(B_1)} &&\quad\mbox{for all } u \in \sobo^{1,\varphi}_0(B_1), \\
    \label{eq:orlicz_poincare_2}\lVert u - (u)_{B_1} \rVert_{\lebe^{\varphi}(B_1)} &\leq c(n,\Delta_2(\varphi)) \lVert D u \rVert_{\lebe^{\varphi}(B_1)} &&\quad\mbox{for all } u \in \sobo^{1,\varphi}(B_1),
\end{alignat}
see for instance \cite{BhattacharyaLeonetti1991}. These can be iterated to infer higher order versions.

For $s \in \mathbb R$, we denote by $I_s$ the $s$-Riesz potential, which acts on $u \in\mathscr{S}^{\prime}(\mathbb R^n)$ via $u \mapsto I_s u = \mathscr{F}^{-1}(\lvert\xi\rvert^{-s} \hat{u}(\xi))$.
A consequence of Fourier multipliers (see e.g.\ \cite[Theorem 6.3.1]{BerghLofstrom1976} using \eqref{eq:orlicz_hm}) is that
\begin{equation}\label{eq:negative_sobolev}
    I_{-m} \colon \dot{\sobo}^{m,\varphi}(\mathbb R^n) \xrightarrow{\ \sim\ } \lebe^{\varphi}(\mathbb R^n)
\end{equation}
is an isomorphism for all $m \in \mathbb Z$ and $\varphi \in \Delta_2 \cap \nabla_2$, where the associated constants depend on $n, m, \Delta_2(\varphi), \nabla_2(\varphi)$ only.

\section{A Poincar\'e-Sobolev type estimate}\label{sec:estimates}

\par In this section we will prove Theorem~\ref{thm:main_chris'_inequ}, which as discussed in the introduction will be a crucial ingredient in establishing partial regularity of minimizers.

Throughout this section, $\B$ and $\C$ will be constant rank operators of order $k$ and $l$ respectively, satisfying $l > k$ and $\ker \B(\xi) = \ima \C(\xi)$ for all $\xi \in \mathbb R^n$, following the conventions outlined in Section~\ref{sec:differential}.

\begin{theorem}\label{thm:B_poincare2}
Let $u \in \mathscr D'(B_R)$ such that $\B u \in \mathcal{M}(B_R)$ and $\C^{\ast}u = 0$ in $B_R$, and let $\theta \in (0,1)$. Then there exists $\pi \in \CC^{\infty}(B_{\theta R})$ such that $\B \pi = 0$, $\C^{\ast}\pi=0$ in $B_{\theta R}$ and we have $u-\pi \in \sobo^{k-s,p}(B_{\theta R})$ for all $s \in (0,1)$ and $1 \leq p < \frac{n}{n-s}$ with the estimate
\begin{equation}
    \frac1{R^s}[D^{k-1}(u-\pi)]_{\sobo^{1-s,p}(B_{\theta R})} + \sum_{j=0}^{k-1} \frac1{R^{k-j}} \lVert D^j(u-\pi)\rVert_{\lebe^p(B_{\theta R})} \leq c \int_{B_R} \lvert \B u \rvert.
\end{equation}
where $c=c(n,p,s,\theta,\B,\C)>0$.
Furthermore for all $1 \leq q < \frac{n}{n-1}$ we have the modular estimate
\begin{align}
    \sum_{j=0}^{k-1}  \left(\int_{B_{\theta R}} E\bigg(\frac{D^j(u-\pi)}{R^{k-j}}\bigg)^q \,\dd x\right)^{\frac1q} &\leq c\int_{B_R} E(\B u), \label{eq:modular-poincare}
\end{align}
where $c=c(n,q,\theta,\B,\C)>0$ and the association $u \mapsto \pi$ is linear.
\end{theorem}

In the regime $1 < p < \infty$ we can establish a similar Korn-type inequality, cf.~\eqref{eq:vp_intro}.
While it is not needed in our regularity proof, we include it as it may be of independent interest.

\begin{theorem}\label{thm:p_korn}
    Let $u \in \mathscr D'(B_R)$ such that $\B u \in \lebe^1(B_R)$ and $\C^{\ast}u=0$.
    For each $\theta \in (0,1)$, there exists $\pi \in \CC^{\infty}(B_{\theta R})$ satisfying $\B\pi = 0$, $\C^{\ast}\pi=0$ in $B_{\theta R}$ and such that the following holds: if for any $1<p<\infty$ we have $\B u \in \lebe^p(B_R)$, then $u - \pi \in \sobo^{k,p}(B_{\theta R})$ with the associated estimates
    \begin{equation}\label{eq:p_poincare}
        \lVert u - \pi \rVert_{\sobo^{k,p}(B_{\theta R})} \leq c \lVert \B u \rVert_{\lebe^p(B_R)}
    \end{equation}
   and
    \begin{equation}\label{eq:p_modular_poincare}
        \sum_{j=0}^k \int_{B_{\theta R}} \lvert V_p(D^j(u-\pi))\rvert^2 \,\dd x \leq c \int_{B_R} \lvert V_p(\B u)\rvert^2 \,\dd x,
    \end{equation}
    where $V_p(z) = (1+\lvert z \rvert^2)^{\frac{p-2}4}z$ and $c = c(n,\B,\C,p,\theta,R)>0$. 
\end{theorem}

\begin{remark}
    A similar result has been established by Franceschini in \cite[Proposition 2.12]{Federico2023}, where a norm inequality of the form
    \begin{equation}
        \inf_{\pi \in \operatorname{ker} \B} \lVert u - \pi \rVert_{\lebe^1(B_{1/2})} \leq c \lVert \B u \rVert_{\lebe^1(B_1)}
    \end{equation}
    is established when $\B$ is $\mathbb R$-elliptic ({i.e.,} $\C^{\ast}=0$). 
    The proof consists of a qualitative argument involving the open mapping theorem and the Ehrenpreis fundamental principle.
    We found such arguments were insufficient to establish a modular estimate of the form \eqref{eq:modular-poincare} that is crucial in our regularity proofs.
    While one can work with the Luxemburg norm associated to 
    $\lebe^{E}$ to establish an estimate of the form
        \begin{equation}
        \inf_{\pi \in \operatorname{ker} (\B,\C^{\ast})} \lVert u - \pi \rVert_{\lebe^E(B_{1/2})} \leq c \lVert \B u \rVert_{\lebe^E(B_1)}
    \end{equation}
    provided $\B$ is constant rank and $\C^{\ast}u=0$, this is strictly weaker than the modular variant we seek.
    We contrast the case to the variable exponent $\lebe^{p(\cdot)}$ spaces, where Poincar\'e-type inequalities can be established with respect to the Luxemburg norm $\lVert \cdot \rVert_{\lebe^{p(\cdot)}(\Omega)}$, however the modular version is known to fail in general (see e.g.\ \cite[Theorem~8.2.4, Example~8.2.7]{DHHR2011} and the references therein).
\end{remark}

\subsection{Negative norm estimates}

The first result we will need is a Korn-type inequality for the pair $(\B,\C^{\ast})$, with a remainder term in arbitrarily negative scales.

\begin{lemma}\label{lem:negative_norm_estimate}
    Let $\nu \in \mathbb Z$ and $\varphi \in \Delta_2\cap\nabla_2$, and suppose $u \in \sobo^{\nu,\varphi}(B_1)$ such that $\B u \in \sobo^{-1,\varphi}(B_1)$ and $\C^{\ast}u \in \sobo^{-m-1,\varphi}(B_1)$ with $m = l-k>0$.
    Then we have $u \in \sobo^{k-1,\varphi}_{\mathrm{loc}}(B_{1})$ and for all $0<t<s<1$ we have the estimate
    \begin{equation}\label{eq:orlicz_lower}
        \lVert u \rVert_{\sobo^{k-1,\varphi}(B_{t})} \leq c \left( \lVert \B u \rVert_{\sobo^{-1,\varphi}(B_s)} + \lVert \C^{\ast}u \rVert_{\sobo^{-m-1,\varphi}(B_s)} + \lVert u \rVert_{\sobo^{\nu,\varphi}(B_s)}\right),
    \end{equation}
    where $c = c(n,\nu,\Delta_2(\varphi),\nabla_2(\varphi),\B,\C,s-t)$.
\end{lemma}

This result will be applied with $\varphi(t) = t^p$ and $\varphi(t) = \frac1{\lambda} E(t)^p$ for suitable $p \in (1,\infty)$ and $\lambda>0$.
We will establish two preliminary results that will be used in the proof.

\begin{lemma}\label{lem:Riesz_domain}
    Let $m,k \in \mathbb N$ and $\varphi \in \Delta_2 \cap \nabla_2$, and suppose $\mathcal K$ is a kernel such that $\hat{\mathcal K}$ is smooth on $\mathbb R^n\setminus\{0\}$ and is homogeneous of degree $-k$.
    Then if $v \in \sobo^{m,\varphi}_0(B_2)$, we have $\mathcal K \ast v \in \sobo^{m+k,\varphi}(B_2)$ and there is $c = c(n,m,k,\mathcal K,\Delta_2(\varphi),\nabla_2(\varphi)>0$ such that
    \begin{equation}
        \lVert \mathcal K \ast v \rVert_{\sobo^{m+k,\varphi}(B_2)} \leq c\lVert v \rVert_{\sobo^{m,\varphi}(B_2)}
    \end{equation}
\end{lemma}

Here the convolution in $\mathcal K$ is in general understood as $\mathcal K \ast v = \mathscr{F}^{-1}(\hat{\mathcal K}(\xi)\hat{v}(\xi))$.

\begin{proof}
    Let $\hat K(\xi) = \hat{\mathcal K}(\xi/\lvert \xi\rvert)$, which is smooth away from zero and homogeneous of degree $-k$, so that $\mathcal K = I_k \ast K$.
    Extending by zero, we can also view $v \in \sobo^{m,\varphi}(\mathbb R^n)$.
    Using \eqref{eq:negative_sobolev} and \eqref{eq:orlicz_hm}, for each $0 \leq j \leq m$ we can estimate
    \begin{equation}\label{eq:K_lemma_1}
    \begin{split}
        \lVert D^{k+j}(\mathcal K\ast v)\rVert_{\lebe^{\varphi}(B_2)} 
        &\leq \lVert D^k(I_k (K \ast D^jv))\rVert_{\lebe^{\varphi}(\mathbb R^n)} \\
        &\leq c\lVert K \ast D^jv \rVert_{\lebe^{\varphi}(\mathbb R^n)} \leq c \lVert D^j v\rVert_{\lebe^{\varphi}(\mathbb R^n)} \leq c \lVert v \rVert_{\sobo^{m,\varphi}(B_2)}.
        \end{split}
    \end{equation}
    Now consider $0 \leq j \leq k-1$.
    Since $\mathcal K$ is smooth in $\R^n \setminus \{0\}$ and homogeneous of degree $(k-n)$ (by \cite[Theorems 7.1.16,\!18]{Hormander1}), there exists $c>0$ such that $\lvert D^j(\mathcal K \ast v) \rvert \leq c I_{k-j} \lvert v \rvert$ holds pointwise on $\mathbb R^n \setminus \{0\}$.
    For such $j$ we can estimate
    \begin{equation*}
        \int_{B_2} \varphi(\lvert D^j(\mathcal K \ast v) \rvert) \,\dd x \leq c \int_{B_2} \varphi(\lvert I_{k-j} \lvert v \rvert) \,\dd x \leq c\int_{B_2} \varphi\bigg( \int_{B_2} \lvert x-y\rvert^{k-j-n} \lvert v(y)\rvert \,\dd y \bigg) \,\dd x.
    \end{equation*}
    Setting $\gamma_{k-j}(x) := \int_{B_2} \lvert x- y\rvert^{k-j-n} \,\dd y$, since $k-j\geq 1$ there is $c_{k-j} \geq 1$ such that $c_{k-j}^{-1} \leq \gamma_{k-j}(x)\leq c_{k-j}$ for all $x \in B_2$.
    Then applying Jensen's inequality with respect to the measure $\nu_x :=\gamma_{k-j}^{-1}(x) \lvert x - \cdot \rvert^{k-j-n} \mathscr L^n \mres B_2$ and using that $\varphi \in \Delta_2$, we can estimate
    \begin{equation}
        \int_{B_2} \varphi(\lvert D^j(\mathcal K \ast v )\rvert) \,\dd x \leq c\int_{B_2} \int_{B_2} \lvert x-y\rvert^{k-j-n} \varphi(\lvert v \rvert) \,\dd y \,\dd x \leq c \int_{B_2} \varphi(\lvert v \rvert) \,\dd y.
    \end{equation}
    As the modular estimate implies the norm estimate
    \begin{equation}\label{eq:K_lemma_2}
        \lVert  D^j(\mathcal K \ast v) \rVert_{\lebe^{\varphi}(B_2)} \leq c \lVert v \rVert_{\lebe^{\varphi}(B_2)} \quad\mbox{for each } 0 \leq j \leq k-1,
    \end{equation}
    the result follows by combining \eqref{eq:K_lemma_1} and \eqref{eq:K_lemma_2}. 
\end{proof}

\begin{lemma}\label{lem:sobolev_negative_support}
    Let $m \in \mathbb N$ and $\varphi \in \Delta_2 \cap \nabla_2$.
   \begin{enumerate}
       \item\label{it:negative_sobolev_a} If $u \in \sobo^{-m,\varphi}(B_2)$ takes the form $u = \di^m U$ for some $U \in \mathscr{D}'(\mathbb R^n, \operatorname{SLin}^m(\mathbb R^n;\mathbb R^n))$ which is supported in $B_1$, then $u \in \dot{\sobo}^{-m,\varphi}(\mathbb R^n)$ and we have
       \begin{equation}\label{eq:lemma05_inequality}
           \lVert u \rVert_{\dot\sobo^{-m,\varphi}(\mathbb R^n)} \leq c(n,m,\Delta_2(\varphi),\nabla_2(\varphi)) \lVert u \rVert_{\sobo^{-m,\varphi}(B_2)}.
       \end{equation}
       \item\label{it:negative_sobolev_b} If $u \in \dot{\sobo}^{-m,\varphi}(\mathbb R^n)$,
       then $u \in \sobo^{-m,\varphi}(B_1)$ and we have
    \begin{equation}
        \lVert u \rVert_{\sobo^{-m,\varphi}(B_1)} \leq c(n,m,\Delta_2(\varphi),\nabla_2(\varphi)) \lVert u \rVert_{\dot{\sobo}^{-m,\varphi}(\mathbb R^n)}.
    \end{equation}
    \end{enumerate}
\end{lemma}

\begin{proof}
  For \ref{it:negative_sobolev_a}, we can write
  \begin{equation}
    \lVert u \rVert_{\dot{\sobo}^{-m,\varphi}(\mathbb R^n)} = \inf\left\{ \langle u, v \rangle : v \in \sobo^{m,\varphi^{\ast}}_{\mathrm{loc}}(\mathbb R^n), \lVert \nabla^mv \rVert_{\lebe^{\varphi^{\ast}}(\mathbb R^n)} \leq 1 \right\}.
  \end{equation}
  Let $\rho \in \hold^{\infty}_c(B_2)$ such that $\mathbbm{1}_{B_1} \leq \rho \leq \mathbbm{1}_{B_2}$, and $P$ a polynomial of degree $(m-1)$ on $\mathbb R^n$ to be determined.
  Using the form $u = \di^m U$, for $v \in \ccinfty(\mathbb R^n)$ we have
  \begin{equation}
    \begin{split}
      \langle u, v \rangle 
      &= (-1)^m \int_{\mathbb R^n} \langle U , \nabla^m(v - P) \rangle \,\dd x \\
      &= (-1)^m \int_{\mathbb R^n} \langle U , \nabla^m(\rho(v-P) \rangle \,\dd x = \langle u, \rho(v-P) \rangle,
    \end{split}
  \end{equation}
  noting that $u$ is supported in $B_1$, where $\rho \equiv 1$.
  Then since $\rho(v-P) \in \sobo^{m,\varphi^{\ast}}(B_2)$ we have
  \begin{equation}
    \lvert \langle u,v \rangle \rvert \leq \lVert u \rVert_{\sobo^{-m,\varphi}(B_2)} \lVert \rho(v-P) \rVert_{\sobo^{m,\varphi^{\ast}}(B_2)}.
  \end{equation}
  We choose $P$ to satisfy
  \begin{equation}
      \int_{B_2} D^{\alpha}(v-P) \,\dd x = 0 \quad\mbox{for all $\lvert\alpha\rvert \leq m-1$},
  \end{equation}
  so then by the Poincar\'e inequality \eqref{eq:orlicz_poincare_2}, we can estimate
  \begin{equation}
    \lVert \rho(v-P) \rVert_{\sobo^{m,\varphi^{\ast}}(B_2)} 
    \lesssim \lVert v-P \rVert_{\sobo^{m,\varphi^{\ast}}(B_2)} 
      \lesssim \lVert D^mv \rVert_{\lebe^{\varphi^{\ast}}(B_2)} 
      \leq \lVert D^mv \rVert_{\lebe^{\varphi^{\ast}}(\mathbb R^n)},
  \end{equation}
  where the implicit constant depends on $n,m,\Delta_2(\varphi),\nabla_2(\varphi)$.
  Since $\ccinfty(\mathbb R^n)$ is dense in $\dot{\sobo}^{m,\varphi^{\ast}}(\mathbb R^n)$, 
  we can extend the above estimates by density to infer that
  \begin{equation}
    \lvert \langle u,v \rangle \rvert \lesssim \lVert u \rVert_{\sobo^{-m,\varphi}(B_2)}\lVert \nabla^mv \rVert_{\lebe^{\varphi^{\ast}}(\mathbb R^n)}
  \end{equation}
  holds for all $v \in \dot{\sobo}^{m,\varphi^{\ast}}(\mathbb R^n)$, which implies \eqref{eq:lemma05_inequality}.

  For \ref{it:negative_sobolev_b}, let $v \in \sobo^{m,\varphi^{\ast}}_0(B_1)$, which extending by zero lies in $\dot{\sobo}^{m,\varphi^{\ast}}(\mathbb R^n)$.
    Using the duality relation $\dot{\sobo}^{-m,\varphi}(\mathbb R^n)^{\ast} \cong \dot{\sobo}^{m,\varphi^{\ast}}(\mathbb R^n)$ and Poincar\'e inequality \eqref{eq:orlicz_poincare_1}, we have
    \begin{equation}
        \lvert\langle u, v \rangle \rvert \leq \lVert u \rVert_{\dot{\sobo}^{-m,\varphi}(\mathbb R^n)}\lVert D^mv \rVert_{\lebe^{\varphi^{\ast}}(\mathbb R^n)} \leq \lVert u \rVert_{\dot{\sobo}^{-m,\varphi}(\mathbb R^n)}\lVert v \rVert_{\sobo^{m,\varphi^{\ast}}(B_1)},
    \end{equation}
    from which the result follows.
\end{proof}

\begin{proof}[Proof of Lemma~\ref{lem:negative_norm_estimate}]
  We will show for all $\nu < k-1$ and $0 < t< s<1$ that if $u \in \sobo^{\nu,\varphi}(B_s)$ such that $\B u \in \sobo^{-1,\varphi}(B_1)$ and $\C^{\ast}u\in \sobo^{-m-1,\varphi}(B_1)$, then $u \in \sobo^{\nu+1,\varphi}(B_s)$ with the estimate
  \begin{equation}
        \lVert u \rVert_{\sobo^{\nu+1,\varphi}(B_{t})} \leq c \left( \lVert \B u \rVert_{\sobo^{-1,\varphi}(B_s)} + \lVert \C^{\ast}u \rVert_{\sobo^{-m-1,\varphi}(B_s)} + \lVert u \rVert_{\sobo^{\nu,\varphi}(B_s)}\right).
  \end{equation}
  If this holds, \eqref{eq:orlicz_lower} follows by iteratively applying the above for suitable $t,s$.
  
   Let $\rho \in \CC^{\infty}_{\rmc}(B_1)$ such that $\mathbbm{1}_{B_t} \leq \rho \leq \mathbbm{1}_{B_{s}}$ and $\lvert D^j\rho \rvert \leq c_j (s-t)^{-j}$ for all $j \in \mathbb N$.
  We will decompose
  \begin{equation}\label{eq:rhou_fourier_decomp}
      \widehat{\rho u}(\xi) = \B^{\dagger}(\xi)\widehat{\B(\rho u)}(\xi) + \C^{\ast\dagger}(\xi) \widehat{\C^{\ast}(\rho u)}(\xi) =: \hat\psi_1(\xi) + \hat\psi_2(\xi).
  \end{equation}
  Since $u \in \sobo^{\nu,\varphi}(B_s)$, we have
  \begin{align}
      \lVert \B(\rho u) \rVert_{\sobo^{\nu-k+1,\varphi}(B_2)}& \leq c(n,\nu,\B,s-t)\left(\lVert \B u\rVert_{\sobo^{\nu-k+1,\varphi}(B_s)} + \lVert u \rVert_{\sobo^{\nu,\varphi}(B_s)}\right), \label{eq:orlicz_lower_mid1}\\
      \lVert \C^{\ast}(\rho u) \rVert_{\sobo^{\nu-l+1,\varphi}(B_2)} &\leq c(n,\nu,\C,s-t)\left(\lVert\C^{\ast} u\rVert_{\sobo^{\nu-l+1,\varphi}(B_s)} + \lVert u \rVert_{\sobo^{\nu,\varphi}(B_s)}\right),\label{eq:orlicz_lower_mid2}
  \end{align}
  by a similar duality argument as in \cite[Proposition 5.1]{LiRaita24}.

  Next we claim that
  \begin{equation}\label{eq:rhou_negative_estimate}
        \lVert \rho u \rVert_{\sobo^{\nu+1,\varphi}(B_2)} \leq (\lVert \B(\rho u) \rVert_{\sobo^{\nu-k+1,\varphi}(B_2)}+\lVert \C^{\ast}(\rho u) \rVert_{\sobo^{\nu-l+1,\varphi}(B_2)}),
    \end{equation}
    holds for all $\nu$, by considering the two cases which arise depending on the sign of $\nu+1$.
  
  {\it Case $1$}: $\nu+1>0$.
  Since $\B^{\dagger}(\xi)$ is smooth away from the origin and homogeneous of order $-k$ by Lemma \ref{lem:dagger}, we can write $\B^{\dagger} = \abss{\cdot}^{-k} \ast B$, where $\hat{B}(\xi) := \B^{\dagger}(\xi/\lvert\xi\rvert)$ is smooth and homogeneous of degree $0$.
  Then the isomorphism \eqref{eq:negative_sobolev} and the H\"ormander--Mikhlin multiplier theorem \eqref{eq:orlicz_hm} imply
   \begin{align*}
       \lVert \psi_1 \rVert_{\dot{\sobo}^{\nu+1,\varphi}(\mathbb R^n)} 
      &= \lVert I_k ( B \ast \B(\rho u))\rVert_{\dot{\sobo}^{\nu+1,\varphi}(\mathbb R^n)}\\
      &\leq c \lVert B \ast \B(\rho u)\rVert_{\dot{\sobo}^{\nu-k+1,\varphi}(\mathbb R^n)} \\
      &\leq  c\lVert \B(\rho u)\rVert_{\dot{\sobo}^{\nu-k+1,\varphi}(\mathbb R^n)} \leq  c\lVert \B(\rho u)\rVert_{\sobo^{\nu-k+1,\varphi}(B_2)},
    \end{align*} where the last inequality follows from Lemma \ref{lem:sobolev_negative_support}\ref{it:negative_sobolev_a} with $m= k-(\nu+1)>0$, since by \eqref{eq:div_expression} we can write $\B (\rho u)$ in the form $\di^k(L(\rho u))$.
    Similarly, using that $\C^{\ast\dagger}$ is smooth on $\mathbb R^n \setminus\{0\}$ and homogeneous of order $-l$, we can estimate
    \begin{equation}
        \lVert \psi_2 \rVert_{\dot{\sobo}^{\nu+1,\varphi}(\mathbb R^n)} \leq c \lVert \C^{\ast}(\rho u) \rVert_{\sobo^{\nu-l+1,\varphi}(B_2)}.
    \end{equation}
    Therefore since $\rho u = \psi_1 + \psi_2$ is supported in $B_1$, by the Poincar\'e inequality \eqref{eq:orlicz_poincare_1} we have
    \begin{equation}
      \begin{split}
        \lVert \rho u \rVert_{\sobo^{\nu+1,\varphi}(B_2)} 
        \leq c\lVert \rho u \rVert_{\dot{\sobo}^{\nu+1,\varphi}(\mathbb R^n)},
      \end{split}
    \end{equation}
    so \eqref{eq:rhou_negative_estimate} follows by combining the above estimates.

    {\it Case $2$}: $\nu+1 \leq 0$. Let $v \in \CC^{\infty}_{\rmc}(B_2)$ and define $\mathcal K$ by setting $\hat{\mathcal K}(\xi) = \B^{\ast\dagger}(\xi)$. Since $\hat{\mathcal K}$ is smooth for $\xi \neq 0$ and $(-k)$-homogeneous by Lemma \ref{lem:dagger}, we can apply Lemma~\ref{lem:Riesz_domain} with $m=-(\nu+1)\geq 0$, which gives
    \begin{equation}
        \lVert \mathcal K \ast v \rVert_{\sobo^{k-(\nu+1),\varphi^{\ast}}(B_2)} \leq c \lVert v \rVert_{\sobo^{-(\nu+1),\varphi^{\ast}}(B_2)}.
    \end{equation}
    Combining this with Plancherel's theorem we can estimate
    \begin{align*}
        \brangle{\psi_1, v} =&\ \langle \widehat{\B (\rho u)}, \B^{\ast\dagger}(\xi) \hat v\rangle = \langle \B (\rho u), \mathcal K \ast v\rangle  \\
        \leq &\ c\normm{\B (\rho u)}_{\sobo^{\nu-k+1,\varphi}(B_2)} \normm{\mathcal K \ast v}_{\sobo^{k-(\nu+1),\varphi^{\ast}}(B_2)} \\
        \leq &\ c\normm{\B (\rho u)}_{\sobo^{\nu-k+1,\varphi}(B_2)} \normm{v}_{\sobo^{-(\nu+1),\varphi^{\ast}}(B_2)}.
        \end{align*} 
    Similarly by applying Lemma~\ref{lem:Riesz_domain} with $\hat{\mathcal K} = \C^{\dagger}$ we obtain the estimate

    \begin{equation}
        \lVert \psi_2 \rVert_{\sobo^{\nu+1,\varphi}(B_2)} \leq c \lVert \C^{\ast}(\rho u)\rVert_{\sobo^{\nu-l+1,\varphi}(B_2)},
    \end{equation}
    so using that $\rho u = \psi_1 + \psi_2$, the claimed estimate \eqref{eq:rhou_negative_estimate} follows.

  Finally we note that $\rho u = u$ on $B_t$ and combine the estimates \eqref{eq:orlicz_lower_mid1}, \eqref{eq:orlicz_lower_mid2} and \eqref{eq:rhou_negative_estimate} to deduce that
  \begin{equation}
  \begin{split}
      \lVert u \rVert_{{\sobo}^{\nu+1,\varphi}(B_t)} 
      &\leq c \lVert \rho u \rVert_{{\sobo}^{\nu+1,\varphi}(B_2)} \\
      &\leq c \left(\lVert \B u \rVert_{\sobo^{-1,\varphi}(B_s)}+\lVert \C^{\ast} u \rVert_{\sobo^{-m-1,\varphi}(B_s)}+\lVert u \rVert_{\sobo^{\nu,\varphi}(B_s)} \right)
      \end{split}
  \end{equation} with $c=c(n ,\nu,\B,\C, \Delta_2(\varphi), \nabla_2(\varphi), s-t)>0$,
  as required.
\end{proof}

\subsection{Endpoint estimates}\label{sec:endpoint_estimates}

We will combine Lemma~\ref{lem:negative_norm_estimate} with suitable endpoint estimates in the fractional and modular scales.
The modular estimate we will need is the following:

\begin{lemma}\label{lem:negative_measure}
    Let $\lambda >0$ and set $\Phi(t) = \frac1{\lambda} E(t)$. Then if $\mu \in \mathcal M(B_1)$, we have $\mu \in \sobo^{-1,\Phi^q}(B_1)$ for all $1 < q < \frac{n}{n-1}$ with the associated estimate
    \begin{equation}
        \lVert \mu \rVert_{\sobo^{-1,\Phi^q}(B_1)} \leq c(n,q) \lVert \mu \rVert_{\lebe^{\Phi}(B_1)}.
    \end{equation}
\end{lemma}

Here the Luxemburg norm of $\mu$ is understood to be
\begin{equation}
    \lVert \mu \rVert_{\lebe^{\Phi}(B_1)} = \inf\bigg\{ s > 0 : \frac1{\lambda}\int_{B_1} E\bigg(\frac{\mu}{s}\bigg) \leq 1\bigg\},
\end{equation}
where we used \eqref{Emu}.

\begin{proof}
    Extending $\mu$ by zero, we view $\mu \in \mathcal M(\mathbb R^n)$ such that $\spt(\mu) \subset B_1$.
    We first claim for $1 < q < \frac{n}{n-1}$ that
    \begin{equation}\label{eq:riesz_luxemburg}
        \lVert I_1\mu \rVert_{\lebe^{\Phi^q}(B_1)} \leq c(n,q) \lVert \mu \rVert_{\lebe^{\Phi}(B_1)}.
    \end{equation}
    This will follow from the modular estimate
    \begin{equation}\label{eq:riesz_estimate}
        \left(\int_{B_1} E(I_1\mu)^q \,\mathrm{d} x\right)^{\frac1q} \leq c(n,q)\int_{B_1} E(\mu),
    \end{equation}
    applied to $\mu/s$ in place of $\mu$ for suitable $s>0$.
    
    To see this, consider first the case when $\mu = f \mathscr L^n$ with $f \in \lebe^1(B_1)$.
    Observe there is $c_1>1$ such that for each $x \in B_1$, we have  $c_1^{-1} \leq \gamma_1(x):=\int_{B_1} \lvert x - y \rvert^{-(n-1)} \,\mathrm{d}y \leq c_1.$
    Hence by Jensen's inequality applied with $\nu_x\coloneqq \gamma_1(x)^{-1}\lvert x - \cdot \rvert^{-(n-1)} \Le^n \mres B_1$, we obtain the pointwise estimate
    \begin{equation}
        \begin{split}
            E(I_1f(x)) 
            &\leq E\left( \gamma_1(x)^{-1}\int_{B_1} \lvert f(y) \rvert\lvert x - y \rvert^{-(n-1)}  \,\mathrm{d}y\right) \\
            &\leq c_1 \int_{B_1} E(f(y)) \lvert x - y \rvert^{-(n-1)} \,\mathrm{d}y = I_1(E(f))(x)
        \end{split}
    \end{equation}
    for all $x \in B_1$. Now applying \cite[Lemma~7.12]{GT01} gives
    \begin{equation}
        \lVert E(I_1(f)) \rVert_{\lebe^q(B_1)} \leq c_1 \lVert I_1(E(f)) \rVert_{\lebe^q(B_1)} \leq c(n,p) c_1 \lVert E(f) \rVert_{\lebe^q(B_1)},
    \end{equation}
    as required.
    For general $\mu$, by mollification we find $\mu_{\eps} = f_{\eps} \mathscr L^n$ such that $\mu_{\eps} \to \mu$ area-strictly (on a slightly larger ball), which also implies $I_1 \mu_{\varepsilon} \weaksto I_1 \mu$ in the sense of measures. Then we can obtain \eqref{eq:riesz_estimate} with the area-strict convergence and the lower semicontinuity of $\int E^q(\cdot)$.

    Now put $\varphi_q(t) = \Phi^q(t)$, which lies in $\Delta_2\cap\nabla_2$, and let $\psi \in \sobo^{1,\varphi_q^{\ast}}_0(B_1)$. Since $\lvert \psi(x) \rvert \leq I_1(\lvert D\psi \rvert)(x)$ for all $x \in B_1$, we have
    \begin{equation}
      \begin{split}
        \bigg\lvert\int_{B_1} \psi(x) \,\dd\mu(x) \bigg\rvert 
        &\leq \int_{B_1} I_1(\lvert D\psi\rvert) \,\dd \lvert\mu\rvert(x) \\
        &\leq \int_{B_1} I_1(\lvert\mu \rvert)(y) \lvert D \psi(y) \rvert \,\dd y\\
        &\leq \lVert I_1\lvert\mu\rvert \rVert_{\lebe^{\varphi_q}(B_1)} \lVert D\psi \rVert_{\lebe^{\varphi_q^{\ast}}(B_1)}\\
        &\leq c(n,q) \lVert \mu \rVert_{\lebe^{\Phi}(B_1)} \lVert \psi \rVert_{\sobo^{1,\varphi_q^{\ast}}(B_1)},
      \end{split}
    \end{equation}
    where we used \eqref{eq:riesz_luxemburg}, 
    noting that $\Delta_2(\varphi_q^{\ast})$ and $\nabla_2(\varphi_q^{\ast})$ depend on $q$ only.
    Hence we deduce that $\mu \in \sobo^{1,\varphi_q^{\ast}}_0(B_1)^{\ast} \cong \sobo^{-1,\varphi_q}(B_1)$ with the associated estimate
    \begin{equation}
        \lVert \mu \rVert_{\sobo^{-1,\Phi^q}(B_1)} \leq c(n,q) \lVert \mu \rVert_{\lebe^{\Phi}(B_1)},
    \end{equation}
    as required.
\end{proof}

In the fractional scales we will need the following:

\begin{proposition}\label{prop:endpoint_fractional}
  Let $1<p<\frac{n}{n-1}$ and let $u\in \sobo^{k-1,p}(B_1)$ be such that $\B u\in \mathcal M(B_1)$ and $\C^{\ast}u \in \sobo^{-m,p}(B_1)$. 
  Then for all $\sigma \in (0,1)$ such that $\frac{\sigma}n > 1-\frac1p$,
  we have $u \in \sobo^{k-\sigma,p}_{\mathrm{loc}}(B_{1})$ and for all $0<t<s<1$ the estimate
  \begin{equation}\label{eq:fractional_korn}
    \lVert u \rVert_{\sobo^{k-\sigma,p}(B_{t})} \leq c \left( \int_{B_s} \lvert \B u\rvert + \lVert \C^{\ast}u\rVert_{\sobo^{-m,p}(B_s)} + \lVert u \rVert_{\sobo^{k-1,p}(B_s)} \right)
  \end{equation} 
  holds, where $c = c(n,\sigma,p,\B,\C,s-t)>0$.
\end{proposition}

We will need some preparatory estimates.
Recall that we have the isomorphism 
\begin{equation}\label{eq:negative_sobolev2}
    I_m \colon \dot{\sobo}^{-m,p}(\mathbb R^n) \xrightarrow{\ \sim\ } \lebe^p(\mathbb R^n)
\end{equation}
as in \eqref{eq:negative_sobolev} with $\varphi(t)=t^p$, valid for all $m \in \mathbb N$ and $1<p<\infty$.
In the endpoint $p=1$ case, for all $\mu \in \mathcal M(\mathbb R^n)$ we have the replacement estimate
\begin{equation}\label{eq:endpoint_riesz}
      \lVert I_1\mu \rVert_{\dot{\operatorname{B}}^{1-\sigma,\frac{n}{n-\sigma}}_{\infty}(\mathbb R^n)} \lesssim \lvert \mu \rvert(\mathbb R^n)
\end{equation}
valid for all $\sigma \in (0,1)$; see \cite[Proposition~8.22]{VSC13} for an elementary proof.

\begin{lemma}\label{lem:Cstar_term}
  Let $1 < p < \infty$, and suppose $u \in \sobo^{k-1,p}(B_2)$ satisfies $\C^{\ast}u \in \sobo^{-m,p}(B_1)$, where $m = l - k>0$.
  Then given any $s \in (0,1)$ and $\rho \in \ccinfty(B_s)$, if we define $\psi$ as
  \begin{equation}
    \hat{\psi}(\xi) = \C^{\ast\dagger}(\xi) \widehat{\C^{\ast}(\rho u)},
  \end{equation} 
  then $\psi \in \sobo^{k,p}(\mathbb R^n)$ and satisfies the estimate
  \begin{equation}
    \lVert \psi \rVert_{\sobo^{k,p}(\mathbb R^n)} \leq c \lVert \rho \rVert_{\sobo^{l,\infty}(B_s)} \left( \lVert \C^{\ast}u\rVert_{\sobo^{-m,p}(B_1)} + \lVert u \rVert_{\sobo^{k-1,p}(B_s)} \right),
  \end{equation} 
  where $c=c(n,m,p,\B,\C)>0$.
\end{lemma}
\begin{proof}
  Since by Lemma~\ref{lem:dagger} we have $\C^{\ast\dagger}(\xi)$ is homogeneous of degree $-l = -k-m$, we can write
  \begin{equation}
    \widehat{D^j\psi}(\xi) = \hat{K}_j(\xi) \lvert \xi\rvert^{j-l} \widehat{\C^{\ast}(\rho u)}, \quad\mbox{where } \hat{K}_j(\xi)= \C^{\ast\dagger}\left( \frac{\xi}{\lvert \xi\rvert} \right) \otimes \left( \frac{\xi}{\lvert \xi\rvert} \right)^{\otimes j} 
  \end{equation} 
  for each $0\leq j \leq k$ and all $\xi \neq 0$.
  Using the H\"ormander--Mikhlin multiplier theorem and \eqref{eq:negative_sobolev}, we can estimate
  \begin{equation}
    \lVert D^j\psi\rVert_{\lebe^p(\mathbb R^n)} \lesssim \lVert I_{m+(k-j)} \C^{\ast}(\rho u)\rVert_{\lebe^p(\mathbb R^n)} \lesssim \lVert \C^{\ast}(\rho u)\rVert_{\dot{\sobo}^{-m-(k-j),p}(\mathbb R^n)}.
  \end{equation} 
  Furthermore, since $0<m+k-j \leq l$ and since we can write $\C^{\ast}$ in the form $\di^lL$ by \eqref{eq:div_expression}, by Lemma~\ref{lem:sobolev_negative_support}\ref{it:negative_sobolev_a} we have
  \begin{equation}
    \lVert \C^{\ast}(\rho u)\rVert_{\dot{\sobo}^{-m-(k-j),p}(\mathbb R^n)} \lesssim \lVert \C^{\ast}(\rho u)\rVert_{{\sobo}^{-m-(k-j),p}(B_2)}.
  \end{equation} 
  We will first consider the case $j = k$. Given any $\phi \in \sobo_0^{m,p^{\prime}}(B_2)$ we have
  \begin{equation}
    \begin{split}
      \langle \C^{\ast}(\rho u),\phi \rangle  
      &= \langle \C^{\ast}u, \rho\;\!\phi\rangle + \sum_{\lvert \alpha\rvert = l} \sum_{0 \neq \beta \leq \alpha} \langle \C_{\alpha}^{\ast}\,\partial^{\beta}\!\rho\,\partial^{\alpha-\beta}u, \phi \rangle \\
      &\leq \lVert \C^{\ast}u \rVert_{\sobo^{-m,p}(B_1)} \lVert \rho\;\!\phi \rVert_{\sobo^{m,p^{\prime}}(B_s)} \\
      &\quad+ c \lVert u \rVert_{\sobo^{k-1,p}(B_s)} \lVert \rho \rVert_{\sobo^{l,\infty}(B_s)}\lVert \phi \rVert_{\sobo^{m,p}(B_s)},
    \end{split}
  \end{equation} 
  where for the commutator terms we integrate by parts to move (at most) $m$ derivatives on $u$ to the test function $\phi$.
  Thus, we infer that
  \begin{equation}
    \lVert \C^{\ast}(\rho u) \rVert_{\sobo^{-m,p}(B_2)} \leq c(n,m,\C) \lVert \rho \rVert_{\sobo^{l,\infty}(B_s)} \left(\lVert \C^{\ast}u \rVert_{\sobo^{-m,p}(B_1)} +  \lVert u \rVert_{\sobo^{k-1,p}(B_s)}\right),
  \end{equation} 
  so by combining the above estimates we obtain
  \begin{equation}\label{eq:Cstar_term_est1}
    \lVert D^k\psi \rVert_{\lebe^p(\mathbb R^n)} \lesssim  \lVert \rho \rVert_{\sobo^{l,\infty}(B_s)} \left(\lVert \C^{\ast}u \rVert_{\sobo^{-m,p}(B_2)} +  \lVert u \rVert_{\sobo^{k-1,p}(B_s)}\right).
  \end{equation} 
  For $0 \leq j < k$, given $\phi \in \sobo^{m+(k-j),p^{\prime}}(B_2)$ we can more simply estimate
  \begin{equation}
    \langle \C^{\ast}(\rho u), \phi \rangle \leq \lVert \rho u \rVert_{\sobo^{k-1,p}(B_s)} \lVert \phi \rVert_{\sobo^{m+1,p^{\prime}}(B_2)},
  \end{equation} 
  thereby establishing that
  \begin{equation}
    \lVert D^j\psi \rVert_{\lebe^p(\mathbb R^n)} \lesssim \lVert \C^{\ast}(\rho u)\rVert_{{\sobo}^{-m-(k-j),p}(B_2)} \lesssim \lVert \rho \rVert_{\sobo^{l,\infty}(B_s)}  \lVert u \rVert_{\sobo^{k-1,p}(B_s)},
  \end{equation} which with \eqref{eq:Cstar_term_est1} gives the desired estimate.
\end{proof}

\begin{proof}[Proof of Proposition {\ref{prop:endpoint_fractional}}]
  Given $0 < t < s <1$, we choose $\rho \in \ccinfty(B_1)$ satisfying $\mathbbm{1}_{B_t} \leq \rho \leq \mathbbm{1}_{B_s}$ and $\lvert \nabla^j\rho\rvert \leq c_j(s-t)^{-j}$ for all $j=1,\cdots,l$.
  Then we can write
  \begin{equation}
      \widehat{\rho u} =  \B^{\dagger}(\xi)\B(\xi)\widehat{\rho u} + \C^{\ast\dagger}(\xi)\C^{\ast}(\xi)\widehat{\rho u} =: \widehat\psi_1 + \widehat\psi_2,
  \end{equation}
  and in particular
  \begin{equation}
    \widehat{{D}^{k-1}(\rho u)} = (\B^{\dagger}(\xi) \otimes \xi^{\otimes (k-1)}) \widehat{\B(\rho u)} + (\C^{\ast\dagger}(\xi) \otimes \xi^{\otimes (k-1)}) \widehat{\C^{\ast}(\rho u)}.
  \end{equation}
  By Lemma~\ref{lem:Cstar_term}, we can estimate the 
  second term as
  \begin{equation}
    \normm{\psi_2}_{\sobo^{k-\sigma,p}(B_s)} \leq \lVert \psi_2 \rVert_{\sobo^{k,p}(B_s)} \leq \frac{c}{(s-t)^{l}} \left( \lVert \C^{\ast}u\rVert_{\sobo^{-m,p}(B_1)} + \lVert u \rVert_{\sobo^{k-1,p}(B_s)} \right).
  \end{equation} 
  For the first term, since Lemma~\ref{lem:dagger} implies that $\B^{\dagger}$ is homogeneous of degree $-k$, we can write
  \begin{equation}
    D^{k-1}\psi_1 = B \ast I_1 (\B(\rho u)), \quad\mbox{where } \widehat B(\xi) = \B^{\dagger}\left( \frac{\xi}{\lvert\xi\rvert} \right)  \otimes \left( \frac{\xi}{\lvert \xi\rvert} \right)^{\otimes (k-1)}. 
  \end{equation} 
  Since $\hat{B}$ is homogeneous of degree zero and smooth in $\R^n \setminus \{0\}$, by multiplier estimates (see \cite[\S 2.3.7]{Triebel73}) combined with the endpoint estimate \eqref{eq:endpoint_riesz} we have
  \begin{equation}
      \lVert D^{k-1}\psi_1\rVert_{\dot{\besov}^{1-\sigma_0,p}_{\infty}(\mathbb R^n)} \lesssim \lVert I_1 (\B(\rho u))\rVert_{\dot{\besov}^{1-\sigma_0,p}_{\infty}(\mathbb R^n)} \lesssim \int_{B_1} \lvert \B (\rho u)\rvert,
  \end{equation}
  where $\sigma_0 = \frac np(p-1)$.
  Hence for $\sigma < \sigma_0$ (which is equivalent to $p < \frac{n}{n-\sigma}$), using the local embedding $\dot{\besov}^{1-\sigma_0,p}_{\infty,\mathrm{loc}} \hookrightarrow \dot{\sobo}^{1-\sigma,p}_{\mathrm{loc}}$ from \cite[\S 3.3.1]{Triebel73} we have
  \begin{equation}
    [D^{k-1}\psi_1 ]_{\sobo^{1-\sigma,p}(B_1)} \leq c\int_{B_1} \lvert \B(\rho u)\rvert.
  \end{equation} 
  Therefore combining the above estimates and noting that $\rho u = u$ on $B_t$, we have
  \begin{equation}
    \begin{split}
    &\ [D^{k-1}u ]_{\sobo^{1-\sigma,p}(B_t)} \leq [D^{k-1}(\rho u) ]_{\sobo^{1-\sigma,p}(B_s)}\\
    \lesssim&\ [D^{k-1}\psi_1 ]_{\sobo^{1-\sigma,p}(B_s)} + \lVert \psi_2 \rVert_{\sobo^{k,p}(B_s)} \\
    \lesssim &\ \int_{B_1} \lvert \B(\rho u)\rvert + \frac{c}{(s-t)^{l}} \left( \lVert \C^{\ast}u\rVert_{\sobo^{-m,p}(B_1)} + \lVert u \rVert_{\sobo^{k-1,p}(B_s)} \right).
  \end{split}
  \end{equation} 
  Since $\rho u $ vanishes on the boundary of $B_t$, we can estimate the lower order terms as
  \begin{equation}
    \lVert u \rVert_{\sobo^{k-1,p}(B_t)} \leq \lVert \rho u \rVert_{\sobo^{k-1,p}(B_s)} \lesssim [ D^{k-1} (\rho u) ]_{\sobo^{1-\sigma,p}(B_s)},
  \end{equation} 
  from which \eqref{eq:fractional_korn} follows.
\end{proof}

\subsection{A general solvability result}

\begin{theorem}\label{thm:pu_replacement}
    Given any real constant coefficient differential operator $\mathscr P$ on $\mathbb R^n$ from $\mathbb R^{\pp}$ to $\mathbb R^{\qq}$, there exists $N \in \mathbb N$  such that the following holds: for any $s \in \mathbb R$, $r \in (0,1)$ and $u \in \mathscr{D}'(B_1,\mathbb R^{\pp})$ such that $\mathscr P u \in \sobo^{s,2}(B_1,\mathbb R^{\qq})$, there exists $v \in \sobo^{s-N,2}(B_{r},\mathbb R^{\pp})$ such that $\mathscr P v = \mathscr P u$ in $B_1$ and the estimate
    \begin{equation}
        \lVert v \rVert_{\sobo^{s-N,2}(B_{r})} \leq c \rVert \mathscr P u \rVert_{\sobo^{s,2}(B_1)}
    \end{equation}
    holds, where $c = c(n,\mathscr P, s,r)>0$.
\end{theorem}

This largely follows from the results in \cite[\S 7.6]{Hormander2000} and \cite[\S 15.2]{Hormander2}, where the key ingredient is a quantitative version of a solvability result due to Ehrenpreis \cite{Ehrenpreis1960}, Malgrange \cite{Malgrange1962}, and Palamadov \cite{Palamodov1970}.
In what follows, the space of entire (analytic) functions $u \colon \mathbb C^n \to \mathbb C^{\qq}$ will be denoted by $\mathrm{A}(\mathbb C^n,\mathbb C^{\qq})$.

\begin{lemma}[{\cite[Theorem~7.6.11]{Hormander2000}}]\label{lem:algebraic_solvability}
    Let $P \in \mathbb C[\zeta]^{\qq \times \pp}$ be a matrix of polynomials in $\mathbb C^n$, then there exists $N \in \mathbb N$ such that the following holds: 
    Let $\phi \colon \mathbb C^n \to \mathbb R$ be globally Lipschitz and plurisubharmonic in that
    \begin{align}
        \label{eq:phi_lipschitz}\lvert D\phi(\zeta) \rvert \leq L \quad&\mbox{for all } \zeta \in \mathbb C^n, \\
        \label{eq:phi_pluri}\sum_{i,j=1}^n w_i \overline{w}_j \partial_{\zeta_i}\partial_{\overline{\zeta}_j}\phi(\zeta) \geq 0 \quad&\mbox{for all } \zeta, w \in \mathbb C^n,
    \end{align}
    Then given any $U \in \mathrm{A}(\mathbb C^n,\mathbb C^{\pp})$, there exists $V \in \mathrm{A}(\mathbb C^n,\mathbb C^{\pp})$ such that
    \begin{equation}
        P(\zeta) V(\zeta) = P(\zeta) U(\zeta) \quad\mbox{for all $\zeta \in \mathbb C^n$},
    \end{equation}
    and we have the estimate
    \begin{equation}
        \int_{\mathbb C^n} \lvert V(\zeta) \rvert^2 \mathrm{e}^{-2\phi(\zeta)}(1+\lvert \zeta\rvert^2)^{-N} \,\dd\zeta \leq c(P,L) \int_{\mathbb C^n} \lvert P(\zeta)U(\zeta) \rvert^2 \mathrm{e}^{-2\phi(\zeta)} \,\dd\zeta
    \end{equation}
    holds, provided the right-hand side is finite.
\end{lemma}

We will apply the above result in conjunction with a quantitative version of the Paley--Wiener--Schwarz theorem, which is essentially contained in \cite[\S 15.2]{Hormander2}.
In what follows, given a compactly supported distribution $u$ in $\mathbb R^n$, we will identify its Fourier transform $\hat u$ with its unique analytic continuation to $\mathbb C^n$.

\begin{lemma}\label{lem:paleywiener_estimate}
    Let $s \in \mathbb R$ and $w \in \sobo^{s,2}(\mathbb R^n,\mathbb R^{\pp})$ with $\spt(w) \subset \overline B_r$.
    Then $\hat{w}$ is entire and satisfies the estimate
    \begin{align}
        \label{eq:pw_upper}\int_{\mathbb C^n} \lvert \hat w(\zeta)\rvert^2 \mathrm{e}^{-2r \lvert \Ima \zeta\rvert}  (1 + \lvert \Rea \zeta \rvert^2)^s ( 1 + \lvert \Ima \zeta \rvert^2)^{-\lvert s \rvert - 2n} \,\dd \zeta &\leq c\lVert w \rVert_{\sobo^{s,2}(\mathbb R^n)},
        \end{align}
    where $c=c(n,s)$.
\end{lemma}
\begin{proof}
    This follows from \cite[Lemma~15.2.2]{Hormander2} with the choice $k(\xi) = (1+\lvert \xi \rvert^2)^{\frac{s}2}$ and $K = B_r$, noting that $H_K(\eta) = \sup_{x \in B_r} \langle x,\eta \rangle = r\lvert \eta \rvert$.
\end{proof}
\begin{lemma}\label{lem:fourier_converse}
    Let $W \in \mathrm{A}(\mathbb C^n,\mathbb C^{\qq})$ be an entire function.
    \begin{enumerate}[label=\rm(\alph*)]
        \item\label{it:FT_recovery_sobolev} If there exists $s \in \mathbb R$ such that
    \begin{equation}
        \int_{\{\zeta \in \mathbb C^n \,:\, \lvert \Ima \zeta \rvert \leq 1\}} \lvert W(\zeta)\rvert^2 (1 + \lvert \Rea \zeta \rvert^2)^{s} \,\dd \zeta < \infty,
    \end{equation}
    then there exists $w \in \sobo^{s,2}(\mathbb R^n,\mathbb C^{\qq})$ such that $\hat w = W$, and the estimate
    \begin{align}
        \label{eq:pw_lower}\int_{\{ \zeta \in \mathbb C^n \,:\, \lvert \Ima \zeta \rvert \leq 1\}} \lvert \hat w(\zeta)\rvert^2 (1 + \lvert\Rea \zeta\rvert^2)^{s} \,\dd \zeta &\geq c^{-1} \lVert w \rVert_{\sobo^{s,2}(\mathbb R^n)}
    \end{align}
    holds, where $c = c(n,s)$.
    \item\label{it:FT_recovery_support} If there exists $r>0$ and $N \in \mathbb Z$ such that
    \begin{equation}
        \int_{\mathbb C^n} \lvert W(\zeta) \rvert^2 \mathrm{e}^{-2r \lvert \Ima \zeta \rvert} (1+\lvert \zeta\rvert)^{-N} < \infty
    \end{equation}
    then there exists a distribution $w \in \mathscr{D}'(\mathbb R^n,\mathbb C^{\qq})$ such that $\hat{w} = W$ and $\spt(w) \subset \overline B_r$.
    \end{enumerate}
\end{lemma}
\begin{proof}
    Let $D_r(\zeta) = \{ z \in \mathbb C^n : \lvert z_i - \zeta_i\rvert < r \ \mbox{for all } 1 \leq i \leq n\}$, then from the analyticity of $W$ we have the mean value inequality
    \begin{equation}\label{eq:mean_value_inequality}
        \sup_{z \in D_{\frac12}(\zeta)} \lvert W(z) \rvert \leq c \bigg(\int_{D_1(\zeta)} \lvert W(z) \rvert^2\,\dd z \bigg)^{\frac12} \quad\mbox{for all $\zeta \in \mathbb C^n$},
    \end{equation}
    using for instance \cite[Theorem~2.2.3]{Hormander2000}, noting the estimate is translation invariant.
    For \ref{it:FT_recovery_sobolev}, using \eqref{eq:mean_value_inequality} we have
    \begin{equation}
        \int_{B_{\frac12}(\xi)} \lvert W(x) \rvert^2 \,\dd x \leq c \int_{D_1(\xi)} \lvert W(z) \rvert^2 \dif z \quad\mbox{for all } \xi\in\mathbb R^n,
    \end{equation}
    where $B_{\frac12}(\xi)$ denotes a ball in $\mathbb R^n$.
    Also there is $c_1=c_1(n,s)>0$ such that for each $\xi \in \mathbb R^n$ we have
    \begin{equation}
        c_1^{-1}(1+\lvert \xi\rvert^2)^s \leq (1+\lvert \Rea \zeta \rvert^2)^s \leq c_1(1+\lvert \xi\rvert^2)^s \quad\mbox{for each $\zeta \in D_1(\xi)$},
    \end{equation}
    from which we infer that
    \begin{equation}
        \int_{B_{\frac12}(\xi)} \lvert W(x) \rvert^2 (1+\lvert x \rvert^2)^s\,\dd x \leq c\,c_1^2 \int_{D_1(\xi)} \lvert W(z) \rvert^2 (1+\lvert \Rea z \rvert^2)^s \,\dd z \quad\mbox{for all } \xi\in\mathbb R^n,
    \end{equation}
    Applying this to a locally finite covering of $\mathbb R$ taking the form $\{B_{1/2}(\xi_i)\}_i$, we obtain the estimate
    \begin{equation}\label{eq:W_real_bound}
        \int_{\mathbb R^n} \lvert W(\xi) \rvert^2 (1 + \lvert \xi \rvert^2)^{s} \,\dd \xi \lesssim \int_{\{\zeta \in \mathbb C^n \,:\, \lvert \Ima \zeta \rvert \leq 1\}} \lvert W(\zeta)\rvert^2 (1 + \lvert \Rea \zeta \rvert^2)^{s} \,\dd \zeta.
    \end{equation}
    Taking the inverse Fourier transform we obtain $w \in \mathscr S'(\mathbb R^n,\mathbb C^{\qq})$ such that $\hat w = W$, and \eqref{eq:W_real_bound} implies that $w \in \sobo^{s,2}(\mathbb R^n,\mathbb C^{\qq})$ with the claimed estimate \eqref{eq:pw_lower}.
    
    For \ref{it:FT_recovery_support}, we only show the claim for $N\in \mathbb N$. If $N \leq0$, then we can simply use that $1 \leq (1+\lvert\zeta\rvert)^{-N}$ to estimate 
    \[ \int_{\mathbb C^n} \abss{W(\zeta)}^2 \mathrm{e}^{-2r\abss{\Ima\zeta}}\dif \zeta \leq \int_{\mathbb C^n} \abss{W(\zeta)}^2 \mathrm{e}^{-2r\abss{\Ima\zeta}}(1+\abss{\zeta})^{-N}\dif \zeta <\infty.\]
    \par If $N >0$ the argument is similar to that that of \ref{it:FT_recovery_sobolev}; there exists $c_2 = c_2(n,N,r)>1$ such that for each $\zeta \in \mathbb C^n$ we have
    \begin{equation}
        c_2^{-1} \mathrm{e}^{r \lvert \Ima \zeta \rvert} (1+\lvert\zeta\rvert)^{\frac{N}2} \leq \mathrm{e}^{r \lvert \Ima z \rvert} (1+\lvert z\rvert)^{\frac{N}2} \leq c_2 \mathrm{e}^{r \lvert \Ima \zeta \rvert} (1+\lvert\zeta\rvert)^{\frac{N}2} \quad\mbox{for each $z \in D_1(\zeta)$,}
    \end{equation}
    so combining with \eqref{eq:mean_value_inequality} gives
    \begin{equation}
        \sup_{z \in D_{\frac12}(\zeta)} \bigg(\lvert W(z) \rvert \mathrm{e}^{-r \lvert \Ima z \rvert} (1+\lvert z\rvert)^{-\frac{N}2}  \bigg) \leq c \, c_2^2 \bigg(\int_{D_1(\zeta)} \lvert W \rvert^2 \mathrm{e}^{-2r \lvert \Ima z \rvert} (1+\lvert z\rvert)^{-N} \,\dd z\bigg)^{\frac12}.
    \end{equation}
    By taking a suitable covering of $\mathbb C^n$ of the form $\{D_{\frac12}(\zeta_i)\}_{i \in \mathbb N}$ and using the assumption in \ref{it:FT_recovery_support}, we infer that
    \begin{equation}\label{eq:FT_recovery_integral}
        \lvert W(\zeta) \rvert \lesssim \mathrm{e}^{r \lvert \Ima \zeta \rvert} (1+\lvert\zeta\rvert)^{\frac{N}{2}} < \infty.
    \end{equation}
    By the Paley--Wiener--Schwarz theorem (see e.g.\ \cite[Theorem 7.3.1]{Hormander1}), we have $W$ is the Fourier transform of a distribution supported in $B_r$.
\end{proof}

\begin{lemma}\label{lem:pluri_modification}
    Given $s \in \mathbb R$ and any $0 < r < R \leq 1$, consider
    \begin{equation}
        \tilde\phi(\zeta) = r\lvert \Ima \zeta \rvert - s \log(1+ \lvert \Rea \zeta\rvert^2)^{\frac12} + (\lvert s \rvert +2n) \log(1 + \lvert \Ima \zeta \rvert^2)^{\frac12}.
    \end{equation}
    Then there exists $c=c(n,s,R-r)>0$ and a globally Lipschitz and plurisubharmonic function $\phi$ satisfying the estimates
    \begin{align}
        \label{eq:psi_lower}\tilde\phi(\zeta) &\leq \phi(\zeta)+ c \\
        \label{eq:psi_support}
        \phi(\zeta) &\leq R\lvert \Ima\zeta \rvert - s \log(1+ \lvert \Rea \zeta \rvert^2)^{\frac12} + c.
    \end{align}
    for all $\zeta \in \mathbb C^n$.
\end{lemma}

\begin{remark}\label{rmk:pw_abbreviation} 
    The function $\tilde\phi$ arises from Lemma~\ref{lem:paleywiener_estimate}, where if $w \in \sobo^{s,2}(\mathbb R^n,\mathbb R^{\qq})$ satisfies $\spt(w)\subset \overline{B}_r$, then $\int_{\mathbb C^n} \lvert \hat{w}(\zeta)\rvert^2 \mathrm{e}^{-2\lvert\tilde{\phi}(\zeta)\rvert} \,\dd\zeta < \infty$. Moreover, Lemma~\ref{lem:fourier_converse} gives a suitable converse.
    The construction of Lemma~\ref{lem:pluri_modification} allows us to modify $\tilde\phi$ to be plurisubharmonic, at the cost of increasing the support, and this way we can apply Lemma~\ref{lem:algebraic_solvability}. 
\end{remark}

\begin{proof}
    To simplify notation, for $\zeta \in \mathbb C^n$ we will write $\xi = \Rea \zeta$ and $\eta = \Ima \zeta$.
    Let $\rho$ be a standard mollifier on $\mathbb R^n$, and given $t \geq 0$ to be determined we will define
    \begin{equation}\label{eq:psit}
        \psi_t(\zeta) = - s\int_{\mathbb R^n} \rho(y) \log(1 + \lvert \xi + (t^2 + \lvert \eta\rvert^2)^{\frac12}y \rvert) \,\dd y + t^{-\frac12} (t^2 + \lvert \eta\rvert^2)^{\frac12} - (t^2 + \lvert \eta\rvert^2)^{\frac14}.
    \end{equation}
    By \cite[Lemma 15.2.3]{Hormander2}, there is $t_0=t_0(n,s)\geq 1$ such that $\psi_t$ is plurisubharmonic for all $t \geq t_0$. 
    To estimate $\psi_t$, we use \cite[(15.2.7), (15.2.8)]{Hormander2} for the first term in \eqref{eq:psit} along with the bound
    \begin{equation}
        t^{-\frac12} (t^2 + \lvert \eta\rvert^2)^{\frac12} + (t^2 + \lvert \eta\rvert^2)^{\frac14} \leq 2t^{-\frac14} (t^2 + \lvert \eta\rvert^2)^{\frac12},
    \end{equation}
    for the latter two terms noting that $t_0 \geq 1$ to estimate
    \begin{align}
        \lvert -s\log(1+\lvert \xi\rvert^2)^{\frac12} -\psi_t(\zeta) \rvert \leq c_1 \log (t^2 + \lvert \eta \rvert^2)^{\frac12} + 2t^{-\frac14} (t^2 + \lvert \eta\rvert^2)^{\frac12}, 
        \label{eq:psit_perturb} \\
        \lvert D \psi_t(\zeta) \rvert \leq c_2 (t^2+\lvert\eta\rvert^2)^{-\frac12} \log (t^2 + \lvert \eta \rvert^2)^{\frac12} +2 t^{-\frac14}(t^2+\lvert \eta \rvert^2)^{-\frac12} \lvert \eta \rvert,\label{eq:psit_lipschitz}
    \end{align}
    where $c_1, c_2>0$ depend on $n$ and $s$ only. 
    Now we define
    \begin{equation}
        \phi(\zeta) = \frac{r+R}2 \lvert \eta \rvert + \psi_t(\zeta),
    \end{equation}
    which is globally Lipschitz by \eqref{eq:psit_lipschitz} and plurisubharmonic since $\psi_t$ is.
    Moreover, the corresponding Lipschitz constant $L$ depends on $n$ and $s$ only.
    To verify the claimed estimates, using \eqref{eq:psit_perturb} we can bound
    \begin{equation}\label{eq:phi_modification_c_2}
    \begin{split}
        &\ \tilde{\phi}(\zeta) - \phi(\zeta) \\ 
        \leq &\ -\frac{R-r}2 \lvert \eta \rvert + (c_1 + \lvert s \rvert + 2n) \log(t^2+\lvert\eta\rvert^2)^{\frac12} + 2t^{-\frac14} (t^2+\lvert\eta\rvert^2)^{\frac12}  \leq c_2,
    \end{split}
    \end{equation}
    where we choose $t\geq t_0$ large enough such that $\frac{R-r}4 \geq 2 t^{-\frac14}$, which allows us to find $c_2 = c_2(n,s,R-r)$, establishing \eqref{eq:psi_lower}.
    
    For \eqref{eq:psi_support}, we use \eqref{eq:psit_perturb} to obtain an analogous estimate as \eqref{eq:phi_modification_c_2}, namely
    \begin{equation}
    \begin{split}
        &\phi(\zeta) - \left( R \lvert \eta \rvert - s \log(1 + \lvert \xi\rvert^2)^{\frac12}\right)\\
        &\quad\leq - \frac{R-r}2 \lvert \eta \rvert + c_1 \log(t^2+\lvert \eta\rvert^2)^{\frac12} + 2 t^{-\frac14} (t^2+\lvert\eta\rvert^2)^{\frac12} \leq c_3,
    \end{split}
    \end{equation}
    where $c_3 = c_3(n,s,R-r)$ can be found since we chose $t$ so that $\frac{R-r}4 \geq 2t^{-\frac14}$.
\end{proof}

\begin{proof}[Proof of {Theorem~\ref{thm:pu_replacement}}]
    We argue as in \cite[Theorem~7.6.13]{Hormander2000}.
    Set $f := \mathscr{P}u$ and $X = \{ \mathscr P^{\ast}w \colon w \in \CC_{\rmc}^{\infty}(B_{r},\mathbb R^{\qq})\}$, and define a linear functional $T \colon X \to \mathbb R$ by sending
    \begin{equation}
        T(\mathscr P^{\ast} w) = \langle f, w \rangle,
    \end{equation}
    where $\langle\cdot,\cdot\rangle$ is the usual distributional pairing. Observe that this is well-defined, since if $\mathscr{P}^{\ast}w=0$, then since $f = \mathscr{P}u$ and $w$ is compactly supported in $B_1$, we have
    \begin{equation}\label{eq:T_well_defined}
        T(\mathscr P^{\ast}w) = \langle \mathscr P u, w\rangle = \langle u, \mathscr P^{\ast}w\rangle = 0.
    \end{equation}

    We will now show that $T$ is bounded with respect to the norm on $\sobo^{-s+N,2}(B_{r},\mathbb R^{\pp})$, where $N$ is as in Lemma~\ref{lem:algebraic_solvability} applied with $\mathscr P^{\ast}(\zeta)$.
    To do this, let $w \in \CC^{\infty}_{\rmc}(B_1,\mathbb R^{\qq})$ such that $\spt(w) \subset \overline B_r,$ and put $R = \frac{1+r}2$. We will show there exists $w_1 \in \sobo^{s-N,2}(\mathbb R^n,\mathbb R^{\qq})$ satisfying $\mathscr{P}^{\ast}w_1 = \mathscr{P}^{\ast}w$ in $B_R$, $\spt(w_1) \subset B_R$ and
    \begin{equation}
        \lVert w_1 \rVert_{\sobo^{s-N,2}(\mathbb R^n)} \leq c \lVert \mathscr P^{\ast}w\rVert_{\sobo^{s,2}(\mathbb R^n)}.
    \end{equation} 

    To establish the existence of $w_1$, we follow the strategy outlined in Remark~\ref{rmk:pw_abbreviation}.
    Let $\phi$ be the plurisubharmonic function obtained from Lemma~\ref{lem:pluri_modification} with the same $r$, with $R=\frac{1+r}2$ and with $-s+N$ in place of $s$, and fix
    $w \in \CC^{\infty}_{\rmc}(B_{r},\mathbb R^{\qq})$.
    Then by 
    applying Lemma~\ref{lem:paleywiener_estimate} and \eqref{eq:psi_lower} we can estimate
    \begin{equation}
        \int_{\mathbb C^n} \lvert \mathscr{P}^{\ast}(\zeta)\hat w(\zeta)\rvert^2 \mathrm{e}^{-2\phi(\zeta)} \,\dd \zeta \leq c \lVert \mathscr P^{\ast}w \rVert_{\sobo^{-s+N,2}(\mathbb R^n)}.
    \end{equation}
    Now applying Lemma~\ref{lem:algebraic_solvability} with $\mathscr P^{\ast}(\zeta)$ we can find an entire function $W_1 \colon \mathbb C^n \to \mathbb C^{\qq}$ such that $\mathscr P^{\ast}(\zeta)\hat w(\zeta) = \mathscr P^{\ast}(\zeta) W_1(\zeta)$ and the estimate
        \begin{equation}
        \int_{\mathbb C^n} \lvert W_1(\zeta) \rvert^2 \mathrm{e}^{-2\phi(\zeta)}(1+\lvert \zeta\rvert^2)^{-N} \,\dd\zeta \leq c \int_{\mathbb C^n} \lvert \mathscr P^{\ast}(\zeta)\hat w(\zeta) \rvert^2 \mathrm{e}^{-2\phi(\zeta)} \,\dd\zeta
    \end{equation}
    holds.
    Then using that $\phi$ satisfies \eqref{eq:psi_support} (with $-s+N$ in place of $s$) we obtain 
    \begin{equation}\label{eq:W1_pluri_estimate}
    \begin{split}
        &\int_{\mathbb C^n} \lvert W_1(\zeta) \rvert^2 \mathrm{e}^{-2R \lvert \Ima \zeta \rvert}(1+\lvert \Rea \zeta\rvert^2)^{-s}(1+\lvert \Ima \zeta \rvert^2)^{-N} \,\dd\zeta\\
        &\leq c  \int_{\mathbb C^n} \lvert W_1(\zeta) \rvert^2 \mathrm{e}^{-2\phi(\zeta)}(1+\lvert \zeta\rvert^2)^{-N} \,\dd\zeta 
        \leq c \lVert \mathscr P^{\ast}w\rVert_{\sobo^{-s+N,2}(\mathbb R^n)},
        \end{split}
    \end{equation}
    where we used the rudimentary estimate $(1+\lvert\zeta\rvert)^{-N} \geq (1+\lvert \Rea \zeta\rvert)^{-N} (1+\lvert\Ima\zeta\rvert)^{-N}$.
    Now by restricting the integral in \eqref{eq:W1_pluri_estimate} to $\lvert\Ima\zeta\rvert \leq 1$ and by applying Lemma~\ref{lem:fourier_converse}\ref{it:FT_recovery_sobolev}, there exists $w_1 \in \sobo^{-s,2}(\mathbb R^n,\mathbb C^{\qq})$ such that $\hat w_1 = W_1$ and we have the bound
    \begin{equation*}
    \begin{split}
        \lVert w_1 \rVert_{\sobo^{-s,2}(\mathbb R^n)} 
        \leq c  \int_{\{\zeta \in \mathbb C^n \,:\, \lvert \Ima \zeta \rvert \leq 1\}} \lvert W_1(\zeta) \rvert^2 (1+\lvert \Rea \zeta\rvert^2)^{-s} \,\dd\zeta
        \leq c \lVert \mathscr P^{\ast}w\rVert_{\sobo^{-s+N,2}(\mathbb R^n)}.
        \end{split}
    \end{equation*}
    Moreover using \eqref{eq:W1_pluri_estimate} with Lemma~\ref{lem:fourier_converse}\ref{it:FT_recovery_support} and noting that $R=\frac{1+r}2 <1$, we obtain that $\spt(w_1) \subset \overline  B_R \subset B_1$.
    Since $\mathscr{P}$ has real coefficients, we have $\mathscr{P}^{\ast}(\Rea w_1) = \Rea \mathscr{P}^{\ast}w_1 = \mathscr{P}^{\ast}w$, so by replacing $w_1$ by $\Rea w_1$, we can assume that $w_1$ is real valued.

    Since $w$ and $w_1$ are compactly supported in $\overline B_R$, the mollification $(w-w_1) \ast \rho_{\eps}$ lies in $\CC^{\infty}_{\rmc}(B_1,\mathbb C^{\qq})$ for all $\eps < 1-R$. Thus by \eqref{eq:T_well_defined} we have
    \begin{equation}
        \langle f, w - w_1 \rangle = \lim_{\eps \to 0 } \langle f, (w - w_1) \ast \rho_{\eps}\rangle = 0,
    \end{equation}
    so we can estimate
    \begin{equation}
        \lvert \langle f, w \rangle \vert = \lvert \langle f, w_1\rangle \rvert \leq \lVert f \rVert_{\sobo^{s,2}(B_1)} \lVert w_1 \rVert_{\sobo^{-s,2}(B_1)}.
    \end{equation}
    Note that since $w_1$ is compactly supported in $B_1$, we have $\lVert w_1 \rVert_{\sobo^{-s,2}(B_1)} \leq \lVert w_1 \rVert_{\sobo^{-s,2}(\mathbb R^n)}$.
    Hence by chaining the above estimates, we obtain
    \begin{equation}
        \lvert \langle f, w \rangle \vert  \leq \lVert f \rVert_{\sobo^{s,2}(B_1)} \lVert w_1 \rVert_{\sobo^{-s,2}(\mathbb R^n)} \leq c \lVert f \rVert_{\sobo^{s,2}(B_1)}\lVert \mathscr P^{\ast}w\rVert_{\sobo^{-s+N,2}(\mathbb R^n)},
    \end{equation}
    that is,
    \begin{equation}\label{eq:T_bounded}
        \lVert T(\mathscr P^{\ast}w)\rVert \leq c \lVert f \rVert_{\sobo^{s,2}(B_1)}\lVert \mathscr P^{\ast}w\rVert_{\sobo^{-s+N,2}(\mathbb R^n)} \quad\mbox{for all } \mathscr P^{\ast}w \in X.
    \end{equation}
    We will view $X \subset \sobo^{-s+N,2}(\mathbb R^n,\mathbb R^{\pp})$ equipped with the norm $\lVert \cdot \rVert_{\sobo^{-s+N,2}(\mathbb R^n)}$.
    Then by \eqref{eq:T_bounded}, we can extend $T$ to a linear functional $\tilde{T}$ on $\sobo^{-s+N,2}(\mathbb R^n,\mathbb R^{\pp})$ by extending by density to $\overline{X}$ and setting $\tilde{T} = 0$ on the orthogonal complement $X^{\perp}$.
    Using that $\sobo^{-s+N,2}(\mathbb R^n,\mathbb R^{\pp})^{\ast} \cong \sobo^{s-N,2}(\mathbb R^n,\mathbb R^{\pp})$ holds with respect to the distributional pairing, there exists $v \in \sobo^{s-N,2}(\mathbb R^n,\mathbb R^{\pp})$ such that $\tilde{T} = \langle v,\cdot\rangle$.

    Therefore
    \begin{equation}
        \langle f, w \rangle =\tilde T(\mathscr{P}^{\ast}w) = \langle v, \mathscr{P}^{\ast}w \rangle = \langle \mathscr{P}v,w \rangle \quad\mbox{for all $w \in \CC^{\infty}_{\rmc}(B_{r},\mathbb R^{\qq})$},
    \end{equation}
    which precisely asserts that $\mathscr Pv = f = \mathscr P u$ in $B_r$, and $v$ moreover satisfies the estimate
    \begin{equation}
        \lVert v \rVert_{\sobo^{s-N,2}(\mathbb R^n)} = \lVert T \rVert \leq c \lVert \mathscr P u \rVert_{\sobo^{s,2}(\mathbb R^n)}.
    \end{equation}
    Thus the result follows by restricting $v$ to $B_r$.
    Finally, since $v$ is obtained through the mapping
    \begin{equation}
        \sobo^{s,2}(\mathbb R^n,\mathbb R^{\pp}) \xrightarrow{\mathscr P u \,\mapsto\, T} X^{\ast} \xrightarrow{T \,\mapsto\, \tilde T} (\overline{X} \oplus X^{\perp})^{\ast} = \sobo^{-s+N,2}(\mathbb R^n,\mathbb R^{\pp})^{\ast} \xrightarrow{\simeq} \sobo^{s-N,2}(\mathbb R^n,\mathbb R^{\pp}),
    \end{equation}
    the association $u \mapsto v$ is linear.
\end{proof}

\subsection{Proof of Theorem~\texorpdfstring{\ref{thm:B_poincare2}}{3.1}}
We can now combine the ingredients from the previous sections to conclude. We will need the following lemma, which will be a consequence of Jensen's inequality.

\begin{lemma}\label{lem:jensen_trick}
    Suppose $\lambda>0$, and define $\Phi = \frac1{\lambda} E$. Then for any $f\in \lebe^1(B_1)$ it holds that
    \begin{equation}\label{eq:Phi_jensen}
        \frac{c_J^{-1}}{E^{-1}(\lambda)} \lVert f \rVert_{\lebe^1(B_1)} \leq \lVert f \rVert_{\lebe^{\Phi}(B_1)} \leq \frac{c_J}{E^{-1}(\lambda)} \lVert f \rVert_{\lebe^2(B_1)},
    \end{equation}
    where $c_J\geq 1$ only depends on $n$. Moreover the first inequality holds if $f$ is a measure.
\end{lemma}

\begin{proof}
    By Jensen's inequality applied with $E(\cdot)$, for $s>0$ we have
    \begin{equation}
        \frac1{\lambda}E\bigg(\fint_{B_1} \lvert f \rvert/s \,\dd x\bigg) \leq \frac1{\lambda} \fint_{B_1} E(f/s)\,\dd x,
    \end{equation}
    and choosing $s = \lVert f \rVert_{\lebe^{\Phi}(B_1)}$ ensures the right-hand side is bounded by a constant, which implies that $ E(\lVert f \rVert_{\lebe^1(B_1)}/\lVert f \rVert_{\lebe^{\Phi}(B_1)} ) \lesssim \lambda$. Rearranging this gives the first inequality, which extends to measures by a density argument.
    For the second inequality we apply Jensen to $E(\sqrt{t}) = \sqrt{t+1}-1$, which is concave, to estimate
    \begin{equation}
        \fint_{B_1} E(f/s)\,\dd x \lesssim E\bigg(\sqrt{\fint_{B_1} \lvert f \rvert^2/s^2\,\dd x }\bigg) = E(\lVert f \rVert_{\lebe^2(B_1)}/s).
    \end{equation}
    Choosing $s$ so that $E(\lVert f \rVert_{\lebe^2(B_1)}/s) = \lambda$, we infer that $\lVert f \rVert_{\lebe^{\Phi}(B_1)} \lesssim s = \lVert f \rVert_{\lebe^2(B_1)} / E^{-1}(\lambda)$, establishing the result.
\end{proof}

\begin{proof}[Proof of Theorem {\ref{thm:B_poincare2}}]
    By rescaling, we can assume that $R=1$. Set $\Phi = \frac1{\lambda} E$ with $\lambda>0$ to be determined.
    Since $\B u \in \mathcal M(B_1)$, we have $\B u \in \sobo^{-n,2}(B_1)$ by Sobolev embedding with the estimate
    \begin{equation}\label{eq:Bu_negative_est}
        \lVert \B u \rVert_{\sobo^{-n,2}(B_1)} \leq c \lVert \B u \rVert_{\mathcal M(B_1)} \leq cE^{-1}(\lambda) \lVert \B u \rVert_{\lebe^{\Phi}(B_1)},
    \end{equation}
    where we used \eqref{eq:Phi_jensen} to obtain the last inequality.
    By applying Theorem~\ref{thm:pu_replacement} with $\mathscr P = (\mathscr B,\mathscr C^{\ast})$, $\theta_1 = \frac{1+\theta}2 \in (0,1)$ in place of $\theta$ and $s=-n$, there is $N \in \mathbb N$ and $v \in \sobo^{-n-N,2}(B_{\theta_1})$ satisfying $\mathscr B v = \mathscr B u$ and $\mathscr C^{\ast}v=0$ in $B_{\theta_1}$, and the estimate
    \begin{equation}\label{eq:v_negative_estimate}
        \lVert v \rVert_{\sobo^{-n-N,2}(B_{\theta_1})} \leq c\lVert \B u \rVert_{\sobo^{-n,2}(B_1)}
    \end{equation}
    holds. We will set $\nu := -(n +N + 1) \in \mathbb Z$,
    then since $\nu+1<0$ and $v \in \sobo^{\nu+1,2}(B_{\theta_1})$, for each $\lvert\alpha\rvert \leq -\nu-1$ we can find $v_{\alpha} \in \lebe^2(B_{\theta_1})$ such that 
    \begin{equation}
        v = \sum_{\lvert \alpha\rvert \leq -\nu-1} D^{\alpha}v_{\alpha}, \ \ \mbox{with}\ \  \lVert v \rVert_{\sobo^{\nu+1,2}(B_{\theta_1})} \leq c \sum_{\lvert \alpha\rvert \leq -\nu-1} \lVert v_{\alpha}\rVert_{\lebe^2(B_{\theta_1})}.
    \end{equation}
    Now using the second estimate in \eqref{eq:Phi_jensen} and applying Lemma~\ref{lem:negative_measure} to each $v_{\alpha}$, we have
    \begin{equation}
        \lVert v_{\alpha}\rVert_{\sobo^{-1,\Phi^q}(B_{\theta_1})} \leq c\lVert v_{\alpha}\rVert_{\lebe^{\Phi}(B_{\theta_1})} \leq \frac{c}{E^{-1}(\lambda)} \lVert v_{\alpha}\rVert_{\lebe^2(B_{\theta_1})}
    \end{equation}
    for each $1 < q < \frac{n}{n-1}$.
    Hence it follows that
    \begin{equation}\label{eq:v_negative_Jensen}
        \lVert v \rVert_{\sobo^{\nu,\Phi^q}(B_{\theta_1})} \leq c\sum_{\lvert \alpha \rvert \leq -\nu-1} \lVert v_{\alpha}\rVert_{\sobo^{-1,\Phi^q}(B_{\theta_1})} \leq \frac{c}{E^{-1}(\lambda)} \lVert v\rVert_{\sobo^{\nu+1,2}(B_{\theta_1})}.
    \end{equation}
    Combining the above estimates we infer that
    \begin{equation}
        \lVert v \rVert_{\sobo^{\nu,\Phi^q}(B_{\theta_1})}\leq c \lVert \B u \rVert_{\lebe^{\Phi}(B_1)},
    \end{equation}
    where the constant $c$ is independent of $\lambda$.
    By applying Lemma~\ref{lem:negative_measure} with $\mu = \B u$, we also have that
    \begin{equation}
         \lVert \B u \rVert_{\sobo^{-1,\Phi^q}(B_1)} \leq c\lVert \B u \rVert_{\lebe^{\Phi}(B_1)}.
    \end{equation}
    Thus by applying Lemma~\ref{lem:negative_norm_estimate} with $\varphi(t) = \Phi^q(t)$, we have $v \in \sobo^{k-1,\Phi^q}(B_{\theta})$ with the associated estimate
    \begin{equation}
        \lVert v \rVert_{\sobo^{k-1,\Phi^q}(B_{\theta})} \leq c \big(\lVert \B v \rVert_{\lebe^{\Phi}(B_1)} +  \lVert v \rVert_{\sobo^{\nu,\Phi^q}(B_{\theta_1})}\big) \leq c \lVert \B u \rVert_{\lebe^{\Phi}(B_1)}.
    \end{equation}
    Now choose $\lambda = \int_{B_1} E(\B u)$, then we have $\normm{\B u}_{\lebe^{\Phi}(B_1)}=1$ and thus $\normm{v}_{\sobo^{k-1,\Phi^q}(B_{\theta})}\leq c$. By the definitions of $\Phi$ and the Orlicz norm, and Proposition \ref{prop:Eprop}\ref{it:Emulti}, we arrive at the desired the modular estimate
    \begin{equation}
        \sum_{j=0}^{k-1} \left(\int_{B_{\theta}} E(D^jv)^q \,\dd x\right)^{\frac1q} \leq c \int_{B_1} E(\B u).
    \end{equation}
    
    For the fractional estimate, let $s \in (0,1)$ and $1 < p < \frac{n}{n-s}$. Then by \eqref{eq:Bu_negative_est}, \eqref{eq:v_negative_estimate} we have
    \begin{equation}
        \lVert v \rVert_{\sobo^{\nu,p}(B_{\theta_1})} \leq c \lVert \B u \rVert_{\sobo^{-n,2}(B_1)} \leq c \lVert \B u \rVert_{\mathcal M(B_1)},
    \end{equation}
    noting that $p \leq 2$.
    Let $\theta_2 =\frac{\theta+\theta_1}2$, so we have $\theta < \theta_2<\theta_1$.
    Then by applying Lemma~\ref{lem:negative_norm_estimate} with $\varphi(t) = t^p$ along with the embedding $\mathcal M(B_1) \hookrightarrow \sobo^{-1,p}(B_1)$, we have $v \in \sobo^{k-1,p}(B_{\theta_2})$ with the associated estimate
    \begin{equation}
        \lVert v \rVert_{\sobo^{k-1,p}(B_{\theta_2})} \leq c (\lVert \B v \rVert_{\sobo^{-1,p}(B_1)} + \lVert v \rVert_{\sobo^{\nu,p}(B_{\theta_1})}) \leq c\lVert \B u \rVert_{\mathcal M(B_1)}.
    \end{equation}
    Then by applying Proposition~\ref{prop:endpoint_fractional} we have $v \in \sobo^{k-s,p}(B_{1/2})$ and the estimate
    \begin{equation}
        \lVert v \rVert_{\sobo^{k-1,p}(B_{\theta})} \leq c( \lVert \B u \rVert_{\mathcal M(B_1)} + \lVert v \rVert_{\sobo^{k-1,p}(B_{\theta_2})}) \leq c\lVert \B u \rVert_{\mathcal M(B_1)},
    \end{equation}
    from which the result follows by taking $\pi = u-v$.
\end{proof}

\begin{proof}[Proof of Theorem \ref{thm:p_korn} (sketch)]
    As in the proof of Theorem \ref{thm:B_poincare2}, we can assume $R=1$. For the modular estimate, we will write $E_p(t) = \lvert V_p(t)\rvert^2$ and put $\Phi(t) = \frac1{\lambda} E_p(t)$ for $\lambda>0$ to be determined.
    We will also set $p_0 = \min\{2,p\}$ and $p_1 = \max\{2,p\}$.
    Then since $\Phi(t^{1/p_0})$ is equivalent to a convex function and $\Phi(t^{1/p_1})$ to a concave function, arguing as in Lemma~\ref{lem:jensen_trick} we have
    \begin{equation}\label{eq:Vp_jensen}
        \frac{c^{-1}}{E_p^{-1}(\lambda)} \lVert f \rVert_{\lebe^{p_0}(B_1)} \leq \lVert f \rVert_{\lebe^{\Phi}(B_1)} \leq \frac{c}{E_p^{-1}(\lambda)} \lVert f \rVert_{\lebe^{p_1}(B_1)}
    \end{equation}
    for all $f \in \lebe^{p_1}(B_1)$.
    Since $\B u \in \lebe^p(B_1)$, by Sobolev embedding \eqref{eq:Vp_jensen}${}_1$ we estimate
    \begin{equation}
        \lVert \B u \rVert_{\sobo^{-n,2}(B_1)} \leq c \lVert \B u \rVert_{\lebe^{p_0}(B_1)} \leq c E_p^{-1}(\lambda) \lVert \B u \rVert_{\lebe^{\Phi}(B_1)}.
    \end{equation}
    Setting $\theta_1=\frac{1+\theta}2$, we can use Theorem~\ref{thm:pu_replacement} there is $N \in \mathbb N$ and $v \in \sobo^{-n-N,2}(B_{\theta_1})$ with $\B v = \B u$ and $\C^{\ast}v=0$ in $B_{\theta_1}$ such that 
    \begin{equation}\lVert v \rVert_{\sobo^{-n-N,2}(B_{\theta_1})} \leq c \rVert \B u \rVert_{\sobo^{-n,2}(B_1)}.\end{equation}
    We will also choose $m_1 \in \mathbb N$ such that $2n > p_0(n-2m_1)$, which ensures that $\sobo^{m_1,2}(B_{\theta_1}) \hookrightarrow \lebe^{p_1}(B_{\theta_1})$, and set $\nu = -(n+N+m_1)$.
    Arguing as in the proof of \eqref{eq:v_negative_Jensen}, and using \eqref{eq:Vp_jensen} and the above estimates give
    \begin{equation}
        E_p^{-1}(\lambda)\lVert v \rVert_{\sobo^{-n-N-m_1,\Phi}(B_{\theta_1})} \leq c \lVert v \rVert_{\sobo^{-n-N,2}(B_{\theta_1})} \leq c E_p^{-1}(\lambda) \lVert \B u \rVert_{\lebe^{\Phi}(B_1)}.
    \end{equation}
    Setting $\theta_2 = \frac{\theta+\theta_1}2$ we can then apply Lemma~\ref{lem:negative_norm_estimate} with this $\Phi$ to infer that
    \begin{equation}
        \lVert v \rVert_{\sobo^{k-1,\Phi}(B_{\theta_2})} \leq c (\lVert v \rVert_{\sobo^{-n-N-1,\Phi}(B_{\theta_1})} + \lVert \B u \rVert_{\lebe^{\Phi}(B_{\theta_1})}) \leq c\lVert \B u \rVert_{\lebe^{\Phi}(B_1)},
    \end{equation}
    where the constant $c$ is independent of $\lambda$.
    Taking $\lambda = \int_{B_1} E_p(\B u) \,\dd x$ we obtain the modular estimate
    \begin{equation}
        \sum_{j=0}^{k-1} \int_{B_{\theta_2}} \lvert V_p(D^jv) \rvert^2 \,\dd x \leq c \int_{B_1} \lvert V_p(\B u)\rvert^2 \,\dd x
    \end{equation}
    and the final derivative can be estimated by applying \cite[Proposition 5.1]{LiRaita24}. Therefore \eqref{eq:p_modular_poincare} following by taking $\pi = u-v$.
    The norm estimate follows by taking $\Phi(t) = t^p$ instead.
\end{proof}

\section{Partial regularity}\label{sec:partialreg}

In this section we will establish the main partial regularity theorem.  
Throughout this section, $\B$ will be a constant rank operator of order $k$ with potential operator $\C$ of order $l>k$. 
By replacing $\C$ with $\Delta^q\C$ for sufficiently large $q$ if necessary, we will moreover assume that $l > k$.
The local minimisers we consider will additionally satisfy $\C^{\ast}u=0$ in $\Omega$; in Section~\ref{subsec:reduction} we will show how to locally reduce to this case, where the proof of Theorem~\ref{thm:main} will be presented.

We will establish the following $\eps$-regularity theorem:

\begin{theorem}\label{thm:partial_reg}
    Let $f$ satisfy \textnormal{\ref{it:fgrowth}}, \textnormal{\ref{it:fcontinuity}}, \textnormal{\ref{it:fBqcstrong}} and \textnormal{\ref{it:frecession}}, and let $u \in \sobo^{k-1,1}(\Omega)$ with $\B u\in \mathcal{M}(\Omega)$ and $\C^{\ast}u=0$ be a local minimiser of \eqref{eq:P_Omega}.
    Then if $\Omega' \Subset \Omega$, $M>0$ and $\alpha \in (0,1)$, there is $\eps>0$ and $R_1>0$ such that if
    $x \in \Omega'$ and $0<R<R_1$ such that
    \begin{equation}
        \lvert (\B u)_{B_R(x_0)}\rvert \leq M, \quad \fint_{B_R(x)} E(\B u - (\B u)_{B_R(x)}) < \eps,
    \end{equation}
    we have $u \in \CC^{k,\alpha}(\overline B_{R/2}(x))$.
\end{theorem}

The proof of Theorem~\ref{thm:partial_reg} will involve establishing decay estimates for the \emph{excess energy}, which for fixed $\alpha \in (0,1)$ is defined as 
    \begin{equation} \label{eq:AHexcess}
    \begin{split}
        \EE(u,x_0,R)&\coloneqq  
        R^{2\alpha}+\fint_{B_R(x_0)} E(\B^a u-(\B u)_{x_0,R}) \,\dd x + \fint_{B_R(x_0)} \lvert \B^su\rvert.
        \end{split}
    \end{equation}

\subsection{Caccioppoli-type inequality}
\par The following Caccioppoli-type inequality is an analogue of \cite{Evans86}.
We remark that we do not use the constant rank property in the below proof, so the result holds for any homogeneous and constant-coefficient $\B$ under the corresponding quasiconvexity assumption \ref{it:fBqcstrong}.

\begin{proposition}\label{prop:Caccioppoli}
     Suppose that the integrand $f$ satisfies the  assumptions \textnormal{\ref{it:fgrowth}}, \textnormal{\ref{it:fcontinuity}}, \textnormal{\ref{it:fBqcstrong}} and \textnormal{\ref{it:frecession}}, $u \in \sobo^{k-1,1}(\Omega,U)$ such that $\B u\in \mathcal{M}(\Omega,V)$ is a local minimizer of \eqref{eq:P_Omega}, and the map $a\col \Omega \to V$ is a polynomial of order $k$ with $\abss{\B a}\leq M$ for some $M>0$. 
     Then for any $\tau \in (0,1)$ and $B_R(x_0) \Subset \Omega$, we have the estimate
        \begin{multline}\label{eq:BCaccioppoli}
            \int_{B_{\frac{R}2}(x_0)}E(\B(u-a))   \leq c \sum_{i=0}^{k-1} \int_{B_R(x_0)} E\left(\frac{D^{i}(u-a)}{R^{k-i}}\right) \dif x\\ 
            +CR\sum_{i=0}^{k-1} \int_{B_R(x_0)} \frac{\abss{D^i(u-a)}}{R^{k-i}}\dif x + CR \int_{B_R(x_0)} \abss{\B (u-a)},
        \end{multline} 
        where $c=c(n,V,\B,M,L,\ell,L_{x_0,R,M+1})>0$.
\end{proposition}

\begin{proof}
    Fix a ball $B_R \coloneqq B_R(x_0)\Subset \Omega$ and $s,t>0$ such that $\frac{R}{2}<s<t<R$. Take a cut-off function $\rho \in \CC_{\rmc}^{\infty}(B_t)$ with \[ \mathbbm{1}_{B_s} \leq \rho \leq \mathbbm{1}_{B_t}, \quad\abss{D^i\rho}\leq c_i(t-s)^{-i} \ \ \mbox{for all } i=1,\dots,k.\] 
    Given $a$ as in the assumption, let $\tilde{f} = f_{\B a}$ be the shifted integrand as in \eqref{eq:shifted_int}. 
    We will also put $\tilde{u}:=u-a$, $\varphi:= \rho \tilde{u}$ and $\psi := (1-\rho)\tilde{u} = \tilde{u}-\varphi$. 
    By the strong quasiconvexity estimate \eqref{eq:pshiftedqc} from Lemma~\ref{lem:shifted_estimates} applied with $\varphi$ and noting that $\varphi \equiv \tilde u$ on $B_s$, we have
    \begin{align*}
        \int_{B_s} E(\B \tilde{u}) &\leq \int_{B_t} E(\B \varphi) \lesssim_{M,\ell} \int_{B_t} \tilde{f}(\B \varphi) \\
        &=\int_{B_t} \tilde{f}(\B \tilde{u}) + \int_{B_t} (\tilde{f}(\B \tilde{u} -\B \psi)-\tilde{f}(\B \tilde{u})),
    \end{align*} 
    where we used that $\varphi = \tilde u - \psi$ in the last line.
    By minimality of $u$, we can estimate the first term as
    \begin{equation*}\begin{split}
        \int_{B_t} \tilde{f}(\B \tilde{u}) &= \int_{B_t}f(x,\B u) + \int_{B_t} (\tilde{f}(\B \tilde{u}) -f(x,\B u)) \\
        &\leq \int_{B_t}f(x,\B u-\B \varphi) + \int_{B_t} (\tilde{f}(\B \tilde{u}) -f(x,\B u))\\
        &= \int_{B_t} \tilde f(\B \psi) + \int_{B_t} (\tilde f(\B \tilde u) - f(x,\B u)) + \int_{B_t} (f(x,\B\psi + \B a) - \tilde f(\B \psi)),
    \end{split}\end{equation*}
   since $\psi + a = u - \varphi$.

    Using the definition of $\tilde f$, and noting that $\tilde{u}-\psi=\varphi$ is compactly supported in $B_t$, we can write
    \begin{equation*}
        \int_{B_t} (\tilde f(\B \tilde u) - \tilde f(\B \psi)) = \int_{B_t} (f(x_0,\B u)-f(x_0,\B \psi+\B a)) + \partial_zf(x_0,\B a) \cdot\underbrace{\int_{B_t} \B(\tilde u-\psi)}_{=0}.
    \end{equation*}
    Combining the above estimates, we obtain
    \begin{align}
        \int_{B_s} E(\B \tilde{u})
        &\lesssim \left[\int_{B_t} \tilde{f}(\B \psi) + \int_{B_t}(f(x_0,\B u) -f(x_0,\B \psi + \B a))\right. \label{eq:CaccioppoliSupmid1}\notag\\
        &\quad\left.+\int_{B_t}(f(x,\B \psi + \B a) -f(x,\B u)) +\int_{B_t} (\tilde{f}(\B \tilde{u}-\B \psi) -\tilde{f}(\B \tilde{u})\right] \\
        &=:c\:\!(I +I\!I+I\!I\!I+I\!V).\notag
    \end{align} 
    For the first term, we will use \eqref{eq:pshiftedup} from Lemma~\ref{lem:shifted_estimates} and that $\psi \equiv 0$ on $B_s$ to estimate
    \begin{equation} \label{eq:CaccioppoliSupmid2}
        I \leq c \int_{B_t\setminus B_s}E(\B \psi) \leq c \int_{B_t\setminus B_s}E(\B \tilde{u}) +c\sum_{i=0}^{k-1}\int_{B_t}E\brac{\frac{D^{i}\tilde{u})}{(t-s)^{k-i}}} \dif x. 
    \end{equation} 
    For the next two terms, we claim that
    \begin{equation}\label{eq:CaccioppoliSupmid3}
        I\!I+I\!I\!I \leq c\:\!t\int_{B_t} \abss{\B \tilde{u}} \dif x + c\:\!t\sum_{i=0}^{k-1}\int_{B_t} \frac{\abss{D^{i}\tilde{u}}}{(t-s)^{k-i}}\dif x.
    \end{equation}
    We will first prove this assuming $\B u$ has no singular part, and extend the result by density.
    By mollification, we can find a sequence $\{u_m\} \subset \CC^{\infty}(\overline B_t)$ such that $\B u_m \to \B u$ area-strictly in $B_t$ and $u_m \to u$ strongly in $\sobo^{k-1,1}(B_t)$.
    Then using \ref{it:fcontinuity} we estimate
 
    \begin{align*}
    \begin{split}
    &\ \int_{B_t}(f(x_0,\B u_m) -f(x_0,\B \psi_m + \B a))\dif x+\int_{B_t}(f(x,\B \psi_m + \B a) -f(x,\B u_m)) \dif x\\
    =&\ \int_{B_t} \int_0^1(\partial_z f(x_0,\B u_m-\tau \B \varphi_m) -\partial_z f(x,\B u_m-\tau \B \varphi_m))\cdot \B \varphi_m\dif \tau \dif x \\
    \leq&\ c\int_{B_t}\int_0^1 \abss{x-x_0}\abss{\B \varphi_m}\dif x \leq c\:\!t\int_{B_t} \abss{\B \varphi_m}\dif x \\
    \leq&\ c\:\!t\int_{B_t} \abss{\B (u_m-a)} \dif x + c\:\!t\sum_{i=0}^{k-1}\int_{B_t} \frac{\abss{D^{i}(u_m-a)}}{(t-s)^{k-i}}\dif x,
    \end{split}
    \end{align*} 
    where $\varphi_m = \rho (u_m-a)$ and $\psi_m = (1-\rho)(u_m-a)$.
    The passage to the limit follows by applying Lemma~\ref{lem:strict_cont}, which relies on \ref{it:frecession}.
    
 For $I\!V$, notice that the integrand vanishes on $B_s$. 
We consider the absolutely continuous and singular parts separately, and use \eqref{eq:pshiftedup} to obtain
    \begin{equation}\label{eq:CaccioppoliSupmid4} \begin{split}
    I\!V &\leq c\int_{B_t\setminus B_s} E(\B^a\tilde u) + E(\B \psi) \dif x + c\int_{B_t\setminus B_s} \abss{\B^s \tilde{u}}\\
    &\leq c \int_{B_t\setminus B_s} E(\B \tilde{u})  + c\sum_{i=0}^{k-1}\int_{B_t}E\brac{\frac{D^{i}\tilde{u}}{(t-s)^{k-i}}}\dif x. 
    \end{split}
    \end{equation}
    \par Then \eqref{eq:CaccioppoliSupmid1}-\eqref{eq:CaccioppoliSupmid4} together imply 
    \begin{align*}
        \int_{B_s} E(\B \tilde{u}) &\leq c_1\int_{B_t\setminus B_s}E(\B \tilde{u}) + c\sum_{i=0}^{k-1}\int_{B_t}E\brac{\frac{D^{i}\tilde{u}}{(t-s)^{k-i}}}\dif x \\
        &\quad+ c\:\!t\brac{\int_{B_t} \abss{\B \tilde{u}}+\sum_{i=0}^{k-1}\int_{B_t} \frac{\abss{D^{i}\tilde{u}}}{(t-s)^{k-i}}\dif x }, 
    \end{align*} and we add $c_1\int_{B_s} E(\B \tilde{u})$ to both sides to get
    \begin{multline*}
        \int_{B_s} E(\B \tilde{u}) \leq \frac{c_1}{c_1+1}\int_{B_t}E(\B \tilde{u})  \\
        + \frac{c_1}{c_1+1}\sum_{i=0}^{k-1}\int_{B_t}E\brac{\frac{D^{i}\tilde{u}}{(t-s)^{k-i}}}\dif x +\frac{c_1}{c_1+1}t\brac{\int_{B_t} \abss{\B \tilde{u}} +\sum_{i=0}^{k-1}\int_{B_t} \frac{\abss{D^{i}\tilde{u}}}{(t-s)^{k-i}}\dif x } .
    \end{multline*} Now the desired inequality \eqref{eq:BCaccioppoli} follows by iteration; more precisely we can use \cite[Lemma 2.11]{LiRaita24} with 
    \begin{align*}
    &\Phi(r) = \int_{B_s}E(\B\tilde{u}), \quad B = cR\int_{B_R}\abss{\B \tilde{u}}, \\
    &\Psi_{k-i}(t) = \int_{B_R}E\brac{\frac{D^{i}\tilde{u}}{t^{k-i}}}\dif x + R\int_{B_t} \frac{\abss{D^{i}\tilde{u}}}{t^{k-i}}\dif x,
    \end{align*}
    which concludes the proof.
\end{proof}

\subsection{\texorpdfstring{$A$}{A}-harmonic approximation}\label{subsec:AHapproximation}
    \par We will adapt the $A$-harmonic approximation lemma, which first appeared in the context of geometric problems in \cite{Simon1983,Simon1996}, and was adapted to the variational setting in \cite{DuzaarGrotowski2000,DGG00,DuSt02,DGK05}.
    In the constant rank setting this result is new, and relies on Theorem~\ref{thm:main_chris'_inequ} along with an $\lebe^2$-based solvability result.

\begin{lemma}\label{lem:AHapproximation2}
    Let $\eps>0$ and $0 < \lambda \leq \Lambda$, then there exists $K = K(n,\B, \C,\lambda,\Lambda)>0$ and $\delta = \delta(n,\B,\C,\lambda,\Lambda,\eps)>0$ for which the following holds:
    Suppose that $A\col V\times V \to \mathbb R$ is a symmetric bilinear operator with constant coefficients and is uniformly elliptic in the wave cone $\Lambda_{\A}$, in that
    \begin{equation}\label{eq:A_elliptic}
        \lambda \lvert v \rvert^2 \leq A[v,v] \leq \Lambda \lvert v\rvert^2 \quad\mbox{for all } v \in \Lambda_{\A},
    \end{equation}
    for some $\Lambda \geq \lambda > 0$.
    Let $x_0 \in \mathbb R^n$ and $r>0$.
    If $w \in \sobo^{k-1,1}(B_r(x_0))$ satisfies 
    \begin{gather}
        \B w\in \mathcal{M}(B_r(x_0)), \quad \C^{\ast}w =0, \label{eq:AHmeasure}\\
        \fint_{B_r(x_0)} E(\B w) \leq \gamma^2 \leq 1,\label{eq:AHenergy}\\
        \fint_{B_r(x_0)} A[\B w,\B\varphi] \leq \gamma \delta \sup_{B_r(x_0)} \abss{\B \varphi}, \quad \mbox{for any }\varphi \in \CC_{\rmc}^{\infty}(B_r(x_0)),\label{eq:AHalmost_harmonic}
    \end{gather} then there exists a function $h\col B_{r/2}(x_0)\to U$ which is $A$-harmonic in that
        \begin{align}
            &\int_{B_{r/2}(x_0)}A[\B h,\B \varphi]\dif x =0, \quad \mbox{for any }\varphi \in \CC_{\rmc}^{\infty}(B_{r/2}(x_0)), \quad \C^{\ast}h=0,\label{eq:Aharmonic}
        \end{align}
        and satisfies the estimates
        \begin{align}
            &\fint_{B_{r/2}(x_0)}E(\B h) \leq K \quad \mbox{and}\quad \sum_{i=0}^{k-1}\fint_{B_{r/2}(x_0)}E\brac{\frac{D^i(w-\gamma h)}{r^{k-i}}}\dif x \leq \gamma^2 \varepsilon. \label{eq:AHapproximation}
        \end{align}
\end{lemma}

We will isolate one result on solvability of elliptic systems, which will be used in the proof.

\begin{lemma}\label{lem:l2_linear_solvability}
    Suppose $r>0$ and $0< \lambda \leq \Lambda$, and let $A, \tilde A$ be uniformly $\Lambda_{\A}$-elliptic coefficient fields in the sense that \eqref{eq:A_elliptic} holds.
    If $h \in \sobo^{k-1,2}(B_r)$ with $\B h \in \lebe^2(B_r)$ is $A$-harmonic in that
    \begin{equation}
        \int_{B_r} A[\B h, \B\varphi] \,\dd x = 0 \quad\mbox{for all $\varphi \in \CC^{\infty}_{\rmc}(B_r)$}, \quad \C^{\ast}h=0,
    \end{equation}
    then there exists $\tilde h \in \sobo^{k-1,2}(B_r)$ which is $\tilde A$-harmonic in that
    \begin{equation}
        \int_{B_r} \tilde A[\B \tilde h, \B\varphi] \,\dd x = 0 \quad\mbox{for all $\varphi \in \CC^{\infty}_{\rmc}(B_r)$}, \quad \C^{\ast}\tilde h=0,
    \end{equation}
    and we have the estimate
    \begin{align}\label{eq:solvability_estimate}
        \lVert \tilde h -h \rVert_{\sobo^{k,2}(B_r)} \leq c\lvert \tilde A - A \rvert \lVert \B h \rVert_{\lebe^2(B_r)},
    \end{align}
    where $c = c(n,\B,\C,\lambda,\Lambda)>0$.
\end{lemma}
\begin{proof}
    Consider the space
    \begin{equation}
        X = \overline{\{\B \varphi : \varphi \in \ccinfty(B_r)\}}^{\,\lebe^2(B_r)} \subset \lebe^2(B_r).
    \end{equation}
    We equip $X$ with the inner product $\langle V, W \rangle_{\tilde A} = \int_{B_r} \tilde A[V,W]$, which by Parseval's theorem and the $\Lambda_{\A}$-ellipticity condition \eqref{eq:A_elliptic} is equivalent to the $\lebe^2$ inner product.
    Then by the Riesz representation theorem, we can find $V \in X$ such that
    \begin{equation}\label{eq:V_riesz_def}
        \int_{B_r} \tilde A[V,\Phi] \,\dif x = -\int_{B_r} \tilde A[\B h, \Phi] \,\dif x \quad\mbox{for all $\Phi \in X$}.
    \end{equation}
    Extending $V$ by zero to $\mathbb R^n$, we define $v \in \mathcal S'(\mathbb R^n)$ to satisfy
    \begin{equation}
           \hat v(\xi) = \B^{\dagger}(\xi)\hat V(\xi) \quad\mbox{for all $\xi \in \mathbb R^n$},
    \end{equation}
    so in particular $\B v = V$.
    Since $V \in \lebe^2(\mathbb R^n)$, by Fourier multiplier estimates we have
    \begin{equation}\label{eq;V_riesz_potential}
        \lVert v \rVert_{\sobo^{k,2}(B_r)} = \lVert \check{\B}^{\dagger} \ast \B v \rVert_{\sobo^{k,2}(\R^n)} \leq c \lVert \B v \rVert_{\lebe^2(\R^n)} =c \lVert \B v \rVert_{\lebe^2(B_r)}.
    \end{equation}
    Now using the $A$-harmonicity of $h$, \eqref{eq:V_riesz_def} and that $\B v = V \in X$, we have
    \begin{equation}
    \begin{split}
        \int_{B_{r}} \lvert \B v \rvert^2 \,\dd x 
        &\lesssim \int_{B_{r}} \tilde A[V,V] \,\dd x \\
        &= -\int_{B_{r}} \tilde A[\B h,V] \,\dd x + \int_{B_{r}} A[\B h, V] \,\dd x\\
        &= \int_{B_{r}} (A - \tilde A)[\B h,\B v] \,\dd x \\
        &\lesssim \lvert A - \tilde A \rvert \left( \int_{B_{r}} \lvert \B h \rvert^2 \right)^{\frac12} \left( \int_{B_{r}} \lvert \B v \rvert^2 \,\dd x \right)^{\frac12},
    \end{split}
    \end{equation}
    so combining with \eqref{eq;V_riesz_potential} we infer that
    \begin{equation}
        \lVert v \rVert_{\sobo^{k,2}(B_r)} \leq c \lvert A -\tilde A \rvert \lVert \B h \rVert_{\lebe^2(B_r)}.
    \end{equation}
    Therefore the result follows by taking $\tilde h = h+v$.
\end{proof}
\begin{proof}[Proof of {Lemma~\ref{lem:AHapproximation2}}]
    By means of the following reduction, it suffices to consider the case of the unit ball $B_r(x_0) = B_1(0)$: we will rescale $w$ by considering $w_r(y) = w(x_0+ry)/r^k$.
    To verify the scaling for $\B w$, choose $s \in (0,r)$ such that $\lvert \B w \rvert(\partial B_s(x_0)) = 0$.
    Then the mollification $w^{\eta} = w \ast \rho_{\eta}$ satisfies $(w^{\eta})_{\rho} = (w_{\rho})^{\eta/\rho}$, and we have $\B(w^{\eta}(x_0+r\,\cdot)) \to \B(w(x_0+r\,\cdot)$ area-strictly in $B_{s/r}$, verifying that
    \begin{equation}
        \B(w(x_0+r\,\cdot)/r^k) = (\B w)(x_0+r\,\cdot) \quad\mbox{in }B_{s/r}.
    \end{equation}
    Since this holds for almost all $s \in (0,r)$, we infer the correct scaling.
         
    We proceed by contradiction, so let $K>0$ to be determined and suppose the result fails. Then there exists $\eps>0$ and for each $m \in \mathbb N$ there exists $w_m \in \sobo^{k-1,1}(B_1)$, coefficient fields $A_m$ satisfying \eqref{eq:A_elliptic} and $\gamma_m \in (0,1]$ satisfying
        \begin{gather}
        \B w_m\in \mathcal{M}(B_1), \quad \C^{\ast}w_m =0, \label{eq:AHmeasure2}\\
        \int_{B_1} E(\B w_m) \leq \gamma_m^2 \leq 1,\label{eq:AHenergy2}\\
        \int_{B_1} A_m[\B w_m,\B\varphi] \leq \frac{\gamma_m}{m} \sup_{B_1} \abss{\B \varphi} \quad \mbox{for any }\varphi \in \CC_{\rmc}^{\infty}(B_1),\label{eq:AHalmost_harmonic2}
    \end{gather}
    and for any $h$ that is $A_m$-harmonic on $B_{1/2}$ in that \eqref{eq:Aharmonic} holds and satisfies the bound $\fint_{B_{1/2}} E(\B h) \leq K$, we have
    \begin{equation}\label{eq:AHenergy_contra}
        \sum_{i=0}^{k-1}\int_{B_{1/2}} \frac1{\gamma_m^2}E(D^i(w_m-\gamma_m h))\dif x >  \varepsilon.
    \end{equation}
    \par By Lemma~\ref{lem:Eint_est} we can estimate
    \begin{equation}
        \int_{B_1} \lvert \B w_m \rvert \leq \lvert B_1\rvert \bigg(3 \fint_{B_1} E(\B w_m)\bigg)^{\frac12} \leq c(n)\gamma_m.
    \end{equation}
    We will apply Theorem~\ref{thm:B_poincare2}, which gives $\pi_m \in \sobo^{k-1,1}(B_{3/4})$ such that $(\B,\C^{\ast})\pi_m =0$ and the estimates
    \begin{align}
        \lVert w_m - \pi_m \rVert_{\sobo^{k-s,p}(B_{3/4})} &\leq c \int_{B_1}\lvert \B w_m\rvert \leq c\;\! \gamma_m, \label{eq:ah_fractional}\\
        \sum_{i=0}^{k-1} \brac{\int_{B_{3/4}} E(D^i(w_m-\pi_m))^p \,\dd x}^{\frac1p} &\leq c \int_{B_1} E(\B w_m),\label{eq:ah_modular}
    \end{align}
    hold for $s \in (0,1)$ and $1 < p < \frac{n}{n-s}$ and all $m \in \mathbb N$.
    We then set
    \begin{equation}
        v_m = \frac{w_m-\pi_m}{\gamma_m},
    \end{equation}
    which by \eqref{eq:ah_fractional} satisfies $\lVert v_m \rVert_{\sobo^{k-s,p}(B_{3/4})}$.
    Then by passing to a subsequence, we can find a limit map $h$ in $B_{3/4}$ satisfying
    \begin{equation}\label{eq:vm_convergence}
    \begin{cases}
        v_m \to h &\mbox{in }\sobo^{k-1,p}(B_{3/4})\\
        D^i v_m \to D^i h & \mbox{a.e. in } B_{3/4} \mbox{ for all } i=0,\dots,k-1\\
        \B v_m \weaksto \B h & \mbox{in } \mathcal{M}(B_{3/4}).
    \end{cases}
    \end{equation} 
    Moreover, by passing to a further subsequence if necessary, we can assume that $A_m \to A$, where the limit field also satisfies \eqref{eq:A_elliptic}.
    Passing to the distributional limit in $\mbox{\eqref{eq:AHmeasure2}}_2$ and \eqref{eq:AHalmost_harmonic}, we have $\C^{\ast}h=0$ in $B_{3/4}$ and
    \begin{equation}
        \int_{B_{3/4}} A[\B h, \B \varphi] =0, \quad \mbox{for any }\varphi \in \CC_{\rmc}^{\infty}(B_{3/4}).
    \end{equation}

    \par By the weak${}^{\ast}$ convergence of $\B w_m$ and convexity of $E(\cdot)$, we have
    \begin{gather}
        \int_{B_{3/4}} \lvert \B h \rvert \leq \liminf_{m \to \infty} \int_{B_{3/4}} \lvert \B v_m \rvert \leq c,\\
        \int_{B_{3/4}}E(\B h) \leq \liminf_{m\to \infty} \int_{B_{3/4}} E(\B v_m)\leq \liminf_{m\to \infty} \int_{B_1} \gamma_m^{-2} E(\B v_m)\leq 1. 
    \end{gather}
    Also using the $A$-harmonicity of $h$, \cite[Theorem 5.2]{LiRaita24}, \eqref{eq:ah_fractional} and \eqref{eq:vm_convergence}, we have $h \in \CC^{\infty}(B_{7/8})$ and the following estimate
    \begin{equation}\label{eq:AH_limit_l2}
    \begin{split}
         \normm{\B h}_{\lebe^2(B_{1/2})} \leq c \int_{B_{3/4}} \lvert D^{k-1} h \rvert \dif x 
        =\ \lim_{m\to \infty} c \int_{B_{3/4}} \lvert D^{k-1} v_m \rvert \dif x\leq c(n,\B,\C,\lambda,\Lambda).
    \end{split}
    \end{equation}
    We wish to modify $h$ to be $A_m$-harmonic, for which we apply Lemma~\ref{lem:l2_linear_solvability} on $B_{1/2}$ to each $h = h_m$ and $\tilde A = A_m$ which gives maps $\tilde h_m \in \sobo^{k-1,2}(B_{1/2})$ that are $A_m$-harmonic and satisfies the estimates
    \begin{equation}\label{eq:vm_strong_conv}
        \lVert \tilde h_m - h \rVert_{\sobo^{k,2}(B_{1/2})} \lesssim c \abss{A_m-A} \to 0
    \end{equation}
    as $m \to \infty$,
    where we used \eqref{eq:AH_limit_l2} and that $A_m \to A$.
 
    Using \eqref{eq:AH_limit_l2} and \eqref{eq:vm_strong_conv} with Proposition \ref{prop:Eprop}\ref{it:Eest} and \ref{it:Esum}, we can estimate
    \begin{equation}
        \fint_{B_{1/2}} E(\B \tilde h_m) \,\dd x \leq  c \fint_{B_{1/2}} \lvert \B (\tilde h_m -h)\rvert^2 + \lvert \B h \rvert \,\dd x \leq c_1,
    \end{equation}
    and we choose $K = c_1$.
    Therefore we can apply \eqref{eq:AHenergy_contra} with the $A_m$-harmonic function $ \tilde h_m + \pi_m/\gamma_m$ to obtain
    \begin{equation}\label{eq:AH_twocontradict2}
        \sum_{i=0}^{k-1} I_{i,m} := \sum_{i=0}^{k-1} \int_{B_{1/2}} \frac1{\gamma_m^2} E(D^i(w_m-\gamma_m\tilde h_m - \pi_m)) \,\dd x \geq \eps.
    \end{equation}
    We will show that $I_{i,m} \to 0$ as $m \to \infty$ for each $0 \leq i \leq k-1$, which will give the desired contradiction.
    For each such $i$, using Proposition~\ref{prop:Eprop}\ref{it:Esum} we can estimate
    \begin{equation}
        I_{i,m} \lesssim \int_{B_{1/2}} \frac1{\gamma_m^2} E(D^i(w_m-\gamma_m h - \pi_m)) \,\dd x + \int_{B_{1/2}} \frac1{\gamma_m^2}E(\gamma_m D^i(h-\tilde h_m)) \,\dd x.
    \end{equation}
    Since $\tilde h_m \to h$ strongly in $\sobo^{k,2}(B_{1/2})$ by \eqref{eq:vm_strong_conv}, combined with Proposition~\ref{prop:Eprop}\ref{it:Eest} we have
    \begin{equation}
        \limsup_{m \to \infty} \int_{B_{1/2}} \frac1{\gamma_m^2}E(\gamma_m D^i(h-\tilde h_m)) \,\dd x \leq c \int_{B_{1/2}} \lvert D^i(h-\tilde h_m)\rvert^2 \,\dd x  =0.
    \end{equation}
    For the first term, using \eqref{eq:ah_modular} we have
    \begin{equation}\label{eq:em_bound_1}
        \sum_{i=0}^{k-1} \bigg( \int_{B_{3/4}} \frac1{\gamma_m^{2p}}E( D^{i}(w_m-\pi_m))^p \,\dd x \bigg)^{\frac1p} \quad\leq \frac{c}{\gamma_m^2}\int_{B_1} E( \B w_m )  \leq c
    \end{equation}
    for $1 < p < \frac{n}{n-1}$.
    Also since each $D^ih$ is bounded on $B_{1/2}$ for $i=0,\ldots,k-1$, using Proposition~\ref{prop:Eprop}\ref{it:Eest} we have
    \begin{equation}\label{eq:em_bound_2}
        \frac1{\gamma_m^{2p}} \int_{B_{1/2}} E(\gamma_m D^ih)^p \,\dd x \leq \int_{B_{1/2}} \lvert D^ih\rvert^p \,\dd x \leq c.
    \end{equation}
    Therefore setting
    \begin{equation}
        e_m := \frac1{\gamma_m^2} E(D^i(w_m-\gamma_m h-\pi_m)) \lesssim \frac1{\gamma_m^2} E(D^i(w_m-\pi_m)) + \frac1{\gamma_m^2} E(\gamma_m D^ih),
    \end{equation}
    which we have estimated using Proposition~\ref{prop:Eprop}\ref{it:Esum}, by \eqref{eq:em_bound_1} and \eqref{eq:em_bound_2} we have $e_m$ is uniformly bounded in $\lebe^p(B_{1/2})$,
    and hence is uniformly integrable on $B_{1/2}$.
    Moreover by \eqref{eq:vm_convergence}$_2$, we have $D^i(w_m-\gamma_m h-\pi_m) = \gamma_m D^i(v_m-h) \to 0$ a.e.\,in $B_{1/2}$, and so by Proposition~\ref{prop:Eprop}\ref{it:Eest} we have
    \begin{equation}
    \begin{split}
         \limsup_{m \to \infty} e_m
        \leq &\ c \lim_{m \to \infty} \frac{\lvert D^i(w_m-\gamma h-\pi_m)\rvert^2}{\gamma_m^2} = 0 \quad\mbox{a.e.\,in } B_{1/2}.
    \end{split}
    \end{equation}
    Therefore, by Vitali's convergence theorem we have
    \begin{equation}
        \lim_{m \to \infty} \int_{B_{1/2}} \frac1{\gamma_m^2} E(D^i(w_m-\gamma_mh-\pi_m)) \,\dd x = 0,
    \end{equation}
    so we that infer each $I_{i,m} \to 0$ as $m \to \infty$. This contradicts \eqref{eq:AH_twocontradict2}, thereby completing the proof.
\end{proof}

\subsection{Almost \texorpdfstring{$A$}{A}-harmonicity}\label{subsec:almost_aharmonic}

The following lemma asserts that minimizers $u$ of \eqref{eq:P_Omega} are locally \emph{almost $A$-harmonic} in the sense that \eqref{eq:AHalmost_harmonic} holds for a suitable $A$, allowing us to apply Lemma~\ref{lem:AHapproximation2}.
It turns out that we will only use the extremality of $u$, namely that $u$ satisfies the Euler-Lagrange system
\begin{equation}\label{eq:EL_system}
    \int_{\Omega} \partial_zf(x,\B^au) \cdot \B \varphi \,\dd x = 0 \quad\mbox{for all }\varphi\in\ccinfty(\Omega).
\end{equation}
This follows from local minimality, which implies the criticality of the mapping 
\begin{equation}t \mapsto \int_{\Omega} f(x,\B u + t \B\varphi) = \int_{\Omega} f(x,\B^au+t\B\varphi) \,\dd x + \int_{\Omega} f^{\infty}(x,\B^su) \quad\mbox{at } t= 0.\end{equation} 
Indeed differentiating at $t=0$ we infer \eqref{eq:EL_system}, noting that, since $\B^s\varphi=0$, the singular part is constant along such variations.

\begin{lemma}\label{lem:AHalmost_har}
    Let $f$ satisfy \textnormal{\ref{it:fgrowth}}, \textnormal{\ref{it:fcontinuity}} and \textnormal{\ref{it:fBqcstrong}}, let $u \in \sobo^{k-1,1}(\Omega)$ be a local minimiser of \eqref{eq:P_Omega}, and fix $M>0$ and $\alpha\in(0,1)$.
    Then for any ball $B_R(x_0)\Subset \Omega$ with $\abss{(\B u)_{x_0,R}}<M$, setting $A := \partial_z^2f(x_0,(\B u)_{x_0,R})$, we have $u$ satisfies
        \begin{equation}\label{eq:AHmin_almost_har}
            \fint_{B_R(x_0)} A[\B u,\B\varphi] \leq c_1(\EE^{\frac{1}{2\alpha}}(R)+ \EE(R)+\omega^{\frac{1}{2}}(c\EE^{\frac{1}{2}}(R)) \EE^{\frac{1}{2}}(R)) \sup_{B_R(x_0)} \abss{\B \varphi} 
        \end{equation} for any $\varphi \in \CC_{\rmc}^{\infty}(B_R(x_0))$, where $\mathscr{E}(R) := \mathscr{E}(u,x_0,R)$ is the excess from \eqref{eq:AHexcess}, $\omega := \omega_{x_0,R,M+1}$ is as in \eqref{eq:loc_cont} and $c_1=c_1(n,V,L,L_{x_0,R,M+1})>0$.
\end{lemma}

\begin{proof}
    Writing $B_R := B_R(x_0)$, let $\varphi \in \CC_{\rmc}^{\infty}(B_R(x_0))$. Then we can write Euler-Lagrange system \eqref{eq:EL_system} as
    \begin{multline}
        0=\fint_{B_R} \partial_z f(x,\B^a u)\cdot \B\varphi \dif x = \fint_{B_R} (\partial_z f(x,\B^a u)-\partial_z f(x_0,\B^a u))\cdot \B\varphi \dif x \\ +\fint_{B_R} (\partial_z f(x_0,\B^a u)-\partial_z f(x_0,(\B u)_R))\cdot \B\varphi \dif x =: \fint_{B_R} (A_1+A_2) \dif x,
    \end{multline}
    noting that $\fint_{B_R} \B\varphi \,\dd x=0$.
    By definition of the excess we have $\mathscr{E}(R) \geq R^{2\alpha}$, and combined with the
    Lipschitz continuity of $\partial_z f(\cdot,z)$ as in \ref{it:fcontinuity} we can estimate
    \begin{equation}\label{eq:Aharm_A1term}
        \abs{\fint_{B_R} A_1 \dif x }\leq LR \,\sup_{B_R}\,\abss{\B \varphi} \leq L \EE^{\frac1{2\alpha}}(R) \,\sup_{B_R}\,\abss{\B \varphi}.
    \end{equation} 
    \par Now consider
    \[ \fint_{B_R}A[\B u,\B \varphi]  = \fint_{B_R} A[\B^a u-(\B u)_R,\B\varphi]\dif x + \fint_{B_R} A[\B^s u,\B\varphi], \]
    recalling that $A = \partial_z^2f(x_0,(\B u)_{R})$.
    Define the sets
    \[ E^+:= \{x\in B_R\col \abss{\B^a u-(\B u)_R}\geq 1\}, \quad E^-:= \{x\in B_R\col \abss{\B^a u-(\B u)_R}<1\},
    \] then we have
    \begin{equation*}\begin{split} 
    \fint_{B_R}A[\B u,\B \varphi] -\fint_{B_R} A_2\dif x &= \frac{1}{\abss{B_R}} \brac{\int_{E^+}+\int_{E^-}} (A[\B^a u-(\B u)_R,\B \varphi] - A_2 )\dif x  \notag\\
    & \hspace{3cm}+ \fint_{B_R} A [\B^s u,\B\varphi] =: J_1+J_2+J_3. 
    \end{split}\end{equation*} 
    We will abbreviate $\tilde L := L_{x_0,R,M+1}$ and $\omega := \omega_{x_0,R,M+1}$ from \eqref{eq:loc_bound} and \eqref{eq:loc_cont} in what follows.
    Since $\lvert A \rvert \leq \tilde L$ we can estimate the third term as
    \[ \abs{J_3} \leq \tilde L\,\sup_{B_R}\,\abss{\B \varphi} \fint_{B_R}\abss{\B^s u} \leq \tilde L \,\sup_{B_R}\,\abss{\B \varphi} \EE(R). \] 
    \par To estimate $J_1$, we need to control the size of $E^+$ using Proposition~\ref{prop:Eprop}\ref{it:Eest} as
    \[ \abss{E^+} \leq \int_{E^+}\abss{\B^a u-(\B u)_R} \dif x \leq c\int_{B_R} E(\B^a u-(\B u)_R)\dif x, \] which together with the boundedness of $\partial_z f(x,z)$ from \eqref{eq:fz_bounded} implies
    \begin{align}
    \begin{split}
        \abs{J_1} &\leq \frac{1}{\abss{B_R}} \sup_{B_R}\,\abss{\B \varphi} \brac{ \tilde L\int_{E^+} \abss{\B^a u-(\B u)_R}\dif x +c(n,V,L) \abss{E^+}} \\
        &\leq c(n,V,L,\tilde L)\sup_{B_R}\,\abss{\B \varphi} \fint_{B_R}E(\B^a u-(\B u)_R) \dif x \leq c \sup_{B_R}\,\lvert\B \varphi\rvert\,\EE(R).
    \end{split}
    \end{align}
    \par For $J_2$, we will use the fundamental theorem of calculus, \eqref{eq:loc_cont} and Proposition \ref{prop:Eprop}\ref{it:Eest} to estimate
    \begin{align*} 
    &\abs{J_2} \leq \frac{\tilde L}{\abss{B_R}} \int_{E^-} \omega(\abss{\B^a u-(\B u)_R})\abss{\B^a u-(\B u)_R}\abss{\B\varphi} \dif x \\
    &\leq \frac{\tilde L\sqrt{\omega(1)}}{\abss{B_R}} \sup_{B_R}\,\abss{\B \varphi} \brac{\int_{E^-}\omega(c\sqrt{E(\B^a u-(\B u)_R)})\dif x}^{\frac{1}{2}}\brac{\int_{E^-} \abss{\B^a u-(\B u)_R}^2 \dif x}^{\frac{1}{2}} \notag\\
    &\leq c \sup_{B_R}\,\abss{\B \varphi} \,\omega^{\frac{1}{2}}\brac{c\brac{\fint_{B_R} E(\B^a u-(\B u)_R)\dif x}^{1/2}}\brac{\fint_{B_R} E(\B^a u-(\B u)_R)\dif x}^{1/2} \notag\\
    &\leq c\sup_{B_R}\,\abss{\B \varphi} \,\omega^{\frac{1}{2}}(c\EE^{\frac{1}{2}}(R)) \EE^{\frac{1}{2}}(R), \notag
    \end{align*} where the penultimate line follows from Jensen's inequality, and $c=c(n,V,\tilde L)>0$.
    \par Combining the above, we arrive at the desired estimate
    \begin{align*}
        \fint_{B_R} A[\B u,\B \varphi] &= -\fint_{B_R} A_1 \,\dd x +
        \fint_{B_R} (A[\B u,\B\varphi] - A_2)\\
        &\leq c(\EE^{\frac{1}{2\alpha}}(R)+ \EE(R)+\omega^{\frac{1}{2}}(c\EE^{\frac{1}{2}}(R)) \EE^{\frac{1}{2}}(R)) \sup_{B_R}\abss{\B \varphi}. \qedhere
    \end{align*}
\end{proof}

\subsection{Excess decay estimate}

Combining the previous results, we obtain the following decay estimate for the excess energy. 

\begin{proposition}\label{prop:EDEAH}
    Let $f$ satisfy \textnormal{\ref{it:fgrowth}}, \textnormal{\ref{it:fcontinuity}}, \textnormal{\ref{it:fBqcstrong}} and \textnormal{\ref{it:frecession}}, let $u \in \sobo^{k-1,1}(\Omega)$ be a local minimiser of \eqref{eq:P_Omega} satisfying $\C^{\ast}u=0$, and fix $B_{2R_0}(x_0) \Subset \Omega$ and $M>0$.
    Then for each $\tau \in (0,\frac1{16})$, there exists $\eps_{\mathrm{exc}} \in (0,1)$ such that the following holds:
    Suppose that $x \in B_{R_0}(x_0)$ and $0 < R < R_0$ such that
    \begin{equation}\label{eq:excess_smallness}
        \lvert (\B u)_{B_R(x)} \rvert \leq M, \quad \EE(u,x,R) < \eps_{\mathrm{exc}}, 
    \end{equation}
     then it holds that
    \begin{equation}\label{eq:excess_decay}
        \EE(u,x,\tau R) \leq 2\tau^{2\alpha}\EE(u,x,R).
    \end{equation}
\end{proposition}
In what follows, we will abbreviate $\EE(u,x,R) = \EE(R)$ when $u$ and $x$ are clear from context.
  
\begin{proof}
    Fix $\eps_{\mathrm{exc}}, \eps_{\mathrm{har}}\in(0,1)$ to be determined.
    By \eqref{eq:loc_bound}, \eqref{eq:ellipticity} we know there exists $0 < \lambda_M < \Lambda_{x_0,R,M}$ such that for all $x \in B_{R_0}(x_0)$ and $z \in V$ with $\lvert z \rvert \leq M$,
    \begin{equation}\label{eq:D2f_elliptic}
        \lambda_{M} \lvert v \rvert^2 \leq \partial^2_zf(x,z)[v,v] \leq \Lambda_{x_0,R_0,M} \lvert v \rvert^2 \quad\mbox{for all } v \in \Lambda_{\A}.
    \end{equation}
    Then corresponding to $\eps = \eps_{\mathrm{har}}$, there is $\delta_{\mathrm{har}}>0$ such that Lemma~\ref{lem:AHapproximation2} holds with $\lambda=\lambda_M$ and $\Lambda=\Lambda_{x_0,R,M}$.

    Fix $x \in B_{R_0}(x_0)$ and $0<R<R_0$ for which \eqref{eq:excess_smallness} holds.
    Suppressing the $x$-dependence, by applying Lemma~\ref{lem:AHalmost_har} we have
    \begin{equation}
        \fint_{B_R} A_{x,R}[\B u,\B\varphi] \leq c_1 (\EE^{\frac{1}{2\alpha}}(R) + \EE(R)+\omega^{\frac{1}{2}}(c\EE^{\frac{1}{2}}(R)) \EE^{\frac{1}{2}}(R)) \sup_{B_R} \abss{\B \varphi} 
    \end{equation} for any $\varphi \in \CC_{\rmc}^{\infty}(B_R)$, where $A_{x,R} = \partial_z^2 f(x,(\B u)_R)$, $c_1 = c_1(n,V,L,L_{x_0,2R_0,M+1})$, and $\omega = \omega_{x_0,2R_0,M+1}$ is as in \eqref{eq:loc_cont}.
    We then choose $\eps_{\mathrm{exc}}$ so that
    \begin{equation}
        c_1 ( \eps_{\mathrm{exc}}^{\frac{1-\alpha}{2\alpha}} + \eps_{\mathrm{exc}}^{\frac12}+  \omega^{\frac12}(c\;\! \eps_{\mathrm{exc}}^{\frac12})) < \delta_{\mathrm{har}}.
    \end{equation}
    Using Lemma~\ref{lem:B_poly}, let $a$ be a homogeneous polynomial of order $k$ such that $\B a = (\B u)_{R}$, 
   and note that $\C^{\ast}a=0$ since the order of $\C^{\ast}$ is strictly greater than $k$. 

   Then by applying Lemma~\ref{lem:AHapproximation2} with the choice $\eps=\eps_{\mathrm{har}}$ and $\gamma^2 := \mathscr{E}(R)$, we obtain $h \in \CC^{\infty}(B_{R/2})$ which is $A_{x,R}$-harmonic and satisfies
    \begin{equation}\label{eq:exc_aharmonic}
        \fint_{B_{R/2}} E(\B h) \leq K, \quad \sum_{i=1}^{k-1} \fint_{B_{R/2}} E\brac{\frac{D^i(u-a-\gamma h)}{R^{k-i}}} \dif x \leq \gamma^2 \eps_{\mathrm{har}},
    \end{equation} where $K=K(n,\B, \C,\lambda,\Lambda)>0$ is as in the statement. 
    \par We will collect several estimates for $h$. 
    Applying Theorem~\ref{thm:B_poincare2} to $h$, we can find $\pi \in \ker(\B,\C^{\ast})$ on $B_{R/2}$ such that
    \begin{align}\label{eq:harmonic_poincare}
        \frac1R\fint_{B_{R/4}} \abss{D^{k-1} (h-\pi)}\dif x &\leq c \fint_{B_{R/2}} \abss{\B h}\dif x,
    \end{align}
    which combined with Lemma~\ref{lem:Eint_est} and \eqref{eq:exc_aharmonic} gives
    \begin{equation}\label{eq:harmonic_poincare2}
        \frac1R\fint_{B_{R/4}} \abss{D^{k-1} (h-\pi)}\dif x \leq c \sqrt{K(K+1)}.
    \end{equation}
    It is clear that $h-\pi$ is $A_{x,R}$-harmonic, so applying the elliptic estimates from \cite[Theorem 5.2]{LiRaita24} we have
    \begin{equation}\label{eq:exc_elliptic_est}
        R\sup_{B_{2\tau R}} \lvert D^{k}(h-\pi) \rvert + R^2\sup_{B_{2\tau R}} \lvert D^{k+1}(h-\pi) \rvert \leq c\fint_{B_{R/4}} \lvert D^{k-1}(h-\pi)\rvert \,\dd x,
    \end{equation}
    noting that $2\tau R < R/8$.
    Now fix a polynomial $b$ of order $k$ such that\[ (D^i b)_{2\tau R} = \gamma(D^i (h-\pi))_{2\tau R}, \quad \mbox{for all } i=0,\dots, k, \]
    then using \eqref{eq:exc_elliptic_est} and \eqref{eq:harmonic_poincare2} we can estimate
    \begin{equation*}
        \begin{split}
            \lvert \B(a+b) \rvert &\leq \lvert (\B u)_{B_{R}(x)} \rvert + \gamma\lvert (\B (h-\pi))_{B_{2\tau R}(x)} \rvert \\
            &\leq \lvert (\B u)_{B_R(x)} \rvert + c\;\! \gamma \sup_{B_{2\tau R}} \abss{D^k(h-\pi)}\\
            &\leq M + c\;\!\gamma \frac{1}{R} \fint_{B_{R/4}} \abss{D^{k-1}(h-\pi)}\dif x\\
            & \leq M + c\;\!\varepsilon_{\mathrm{exc}} \sqrt{K(K+1)}.
        \end{split}
    \end{equation*}
    By shrinking $\eps_{\mathrm{exc}}$ as necessary, we can ensure that $\lvert \B(a+b)\rvert \leq M+1$.
    Now observe that, since $\B(u + \gamma\pi) = \B u$ in $B_{R/4}$, we have $u+\gamma\pi$ is also a local minimiser of $\mathcal F$ in $B_{R/4}$.
   We will now apply the Caccioppoli-type inequality (Proposition~\ref{prop:Caccioppoli}) to $u+\gamma\pi$ on $B_{\tau R}$ using the polynomial $a+b$ to obtain
       \begin{equation}\label{eq:main_excess_split}\begin{split}
        &\fint_{B_{\tau R}} E(\B u-(\B u)_{\tau R}) 
        \leq 4\fint_{B_{\tau R}} E(\B (u+\gamma\pi-a-b))  \\
        &\quad\leq c \sum_{i=0}^{k-1} \fint_{B_{2\tau R}} E\brac{\frac{D^i(\tilde{u}+\gamma\pi-b)}{(2\tau R)^{k-i}}}\dif x + c\tau R \sum_{i=0}^{k-1}\fint_{B_{2\tau R}} \frac{\abss{D^i(\tilde{u}+\gamma\pi-b)}}{(2\tau R)^{k-i}}\dif x \\
        &\qquad+ c\tau R \fint_{B_{2\tau R}} \abss{\B \tilde{u}-\B b} =: I+I\!I+I\!I\!I,
    \end{split}\end{equation}
    where we abbreviate $\tilde u = u-a$, and used the almost minimality of $(\B u)_{\tau R}$ in the first line (see \cite[Lemma~2.8]{GmeKri_BV}).
    To estimate $I$, we split the integrand using Proposition~\ref{prop:Eprop}\ref{it:Esum} as
    \begin{equation}\begin{split}
        I &\leq c\sum_{i=0}^{k-1}\fint_{B_{2\tau R}} E\brac{\frac{D^i(\tilde{u}-\gamma h)}{(2\tau R)^{k-i}}}\dif x + c\sum_{i=0}^{k-1}\fint_{B_{2\tau R}} E\brac{\frac{D^i(\gamma (h-\pi)-a)}{(2\tau R)^{k-i}}}\dif x\\
        &\coloneqq I_1 + \sum_{i=0}^{k-1} J_i. 
    \end{split}\end{equation} 
    Using Proposition~\ref{prop:Eprop}\ref{it:Emulti} and $\mbox{\eqref{eq:exc_aharmonic}}_2$, we can estimate the first term as
    \begin{equation}
        I_1 \leq c \sum_{i=0}^{k-1} \frac1{\tau^{2(k-i)+n}} \fint_{B_{R/2}} E\bigg(\frac{D^i(\tilde u -\gamma h)}{R^{k-i}}\bigg) \,\dd x \leq \frac{c \;\!\eps_{\mathrm{har}}}{\tau^{2k+n}} \EE(R).
    \end{equation}

    For the second term, combining \eqref{eq:exc_elliptic_est}, the fact that $E(z) \lesssim \lvert z \rvert^2$ and the Poincar\'e inequality, for each $i = 0,\ldots,k-1$ we can estimate
    \begin{equation} 
    \begin{split}
        J_i &\leq c\fint_{B_{2\tau R}} \frac{\lvert D^i(\gamma (h-\pi) - b)\rvert^2}{(\tau R)^{2(k-i)}} \,\dif x 
        \leq c \fint_{B_{2\tau R}} \lvert D^k(\gamma (h-\pi) - b)\rvert^2 \,\dd x \\
        &\quad\leq c\;\! \tau^2 \gamma^2 \sup_{B_{2\tau R}}\,\lvert D^{k+1}(h-\pi)\rvert^2 
        \leq c\;\! \tau^2 \gamma^2 \bigg(\fint_{B_{R/4}} \frac{\lvert D^{k-1} (h-\pi)\rvert}R \,\dd x\bigg)^2.
    \end{split}
    \end{equation}
    Now using \eqref{eq:harmonic_poincare}, $\mathrm{\eqref{eq:exc_aharmonic}}_1$ and Lemma~\ref{lem:Eint_est}, we deduce that
    \begin{equation}
        J_i \leq c\;\! \tau^2 \gamma^2 \bigg( \fint_{B_{R/2}} \lvert \B h \rvert \,\dd x\bigg)^{2} \leq c\;\! \tau^2 \gamma^2  \fint_{B_{R/2}} E(\B h) \,\dd x  \leq c \tau^2 \EE(R).
    \end{equation}
    Hence combining the above estimates gives
    \[ I \leq I_1 + \sum_{i=0}^{k-1} J_i \leq c_2 \left(\frac{\eps_{\mathrm{har}}}{\tau^{2k+n}} + \tau^2 \right)\EE(R).\]
    Since $\eps_{\mathrm{har}} < 1$, by choosing $\eps_{\mathrm{exc}} > \EE(R)$ small enough so that $c_2 ( \tau^{-2k-n} + \tau^2) \eps_{\mathrm{exc}} \leq 1$, we can use Lemma~\ref{lem:Eint_est} to bound
     \[
     \fint_{B_{2\tau R}} \frac{\abss{D^i(\tilde{u}+\gamma\pi-b)}}{(2\tau R)^{k-i}}\dif x \leq c \brac{\fint_{B_{2\tau R}} E\brac{\frac{D^i(\tilde{u}+\gamma\pi-b)}{(2\tau R)^{k-i}}}\dif x }^{1/2}
    \] for each $i = 0,\ldots,k-1$.
    This implies that
    \[I\!I \leq c\tau R I^{1/2} \leq c\tau R \big(\frac{\sqrt{\eps_{\mathrm{har}}}}{\tau^{k+n/2}}+c\tau\big)\sqrt{\EE(R)} \leq c \bigg( \frac{\sqrt{\eps_{\mathrm{har}}}}{\tau^{k+n/2}} + \tau^2\bigg) \EE(R),\]
        recalling that $R \leq \EE(R)^{\frac1{2\alpha}} \leq \sqrt{\EE(R)}$.
    For $I\!I\!I$, noting that $\EE(R) < \eps_{\mathrm{exc}} \leq 1$ and using again Lemma~\ref{lem:Eint_est} and that $R \leq \EE(R)^{\frac1{2\alpha}}$, we have
    \begin{align}
    \begin{split}
        I\!I\!I &\leq c \tau R \fint_{B_{2\tau R}} \abss{\B \tilde{u}} + c\tau R\fint_{B_{2\tau R}} \gamma\abss{\B h}\,\dd x \\
        &\leq c\tau R \tau^{-n}\brac{ \EE(R)^{\frac12} + \gamma \brac{\fint_{B_{R/2}} E(\B h)\dif x}^{1/2}} \\
        &\leq c \tau^{1-n} R\gamma \leq c \tau^{1-n} \EE^{\frac1{2\alpha}+\frac12},
    \end{split}
    \end{align}
    where the penultimate inequality follows from \eqref{eq:exc_aharmonic}$_1$.
    Therefore combining our estimates from \eqref{eq:main_excess_split} onwards, we obtain that
    \begin{equation}
        \EE(\tau R) - (\tau R)^{2\alpha} \leq c_3 \left( \frac{\eps_{\mathrm{har}}}{\tau^{2k+n}} + \frac{\sqrt{\eps_{\mathrm{har}}}}{\tau^{k+n/2}} + \tau^2 + \tau^{1-n} \eps_{\mathrm{exc}}^{\frac1{2\alpha}-\frac{1}{2}} \right) \EE(R).
    \end{equation}
    By choosing $\eps_{\mathrm{har}}$ and $\eps_{\mathrm{exc}}$ to be sufficiently small, the claimed estimate \eqref{eq:excess_decay} follows.
\end{proof}

\subsection{Iteration and H\"{o}lder regularity}\label{subsec:iteration}

Once the decay estimate \eqref{eq:excess_decay} is established, a routine iteration argument establishes the H\"older continuity of $\B u$. We will briefly record the details.

\begin{proof}[Proof of Theorem {\ref{thm:partial_reg}}]
    Let $B_{2R_0}(x_0) \Subset \Omega$, $M>0$, $\alpha \in (0,1)$ and $\tau \in (0,1)$ to be determined.
    Then letting $\beta \in (\alpha,1)$ by applying Proposition~\ref{prop:EDEAH} using $M+1$ and $\beta$ in place of $M$ and $\alpha$, there is $\eps_0 \in (0,1)$ such that if $x\in B_{R_0}(x_0))$ and $0<R<R_0$ such that
    \begin{equation}
        \lvert (\B u)_{B_R(x)} \rvert \leq M+1, \quad \EE(u,x,R) = \fint_{B_R(x)} E(\B u - (\B u)_{B_R(x)}) + R^{2\beta} < \eps_0,
    \end{equation}
    implies that
    \begin{equation}
        \EE(u,x,\tau R) \leq 2\tau^{2\beta}\EE(u,x,R).
    \end{equation}

    We now choose $\tau \in (0,1)$ such that $2 \tau^{2(\beta-\alpha)} <1$.
    Now letting $0<\eps_1 <\min\{2^{-n-3}\eps_0, 2^{-n}/3\}$ to be determined and putting $R_1 = (\eps_0/2)^{\frac1{2\beta}}$, suppose that $x \in B_{R_0}(x_0)$ and $0<R<R_1$ such that
    \begin{equation}
        \lvert (\B u)_{B_R(x)} \rvert \leq M, \quad \fint_{B_R(x)} E(\B u - (\B u)_{B_R(x)}) < \eps_1.
    \end{equation}
    Then for all $y \in B_{R/2}(x)$, we have $\lvert (\B u)_{B_{R/2}(y)}\rvert \leq M+ 2^n \sqrt{3\varepsilon_1}\leq M+1$ and $\EE(u,y,R/2) < 2^{n+2}\eps_1 + R^{2\beta} \leq \eps_0$ by our choice of $\eps_1$ and $R_1$, so $\EE(u,y,\tau R/2) \leq \tau^{2\alpha} \EE(u,y,R/2)$ by our choice of $\tau$.
    Arguing similarly as in \cite[Proposition~4.8]{GmeKri_BV} can apply this inductively to infer that
    \begin{equation}
      \begin{split}
        \EE(u,y,\tau^kR/2) &\leq \tau^{2k\alpha} \EE(u,y,R/2), \\
        \lvert (\B u)_{B_{\tau^k R/2}(y)}\rvert &\leq \lvert (\B u)_{B_{R}(y)}\rvert +  \frac{\sqrt{3\eps_1}}{\tau^n(1-\tau^{\alpha})} \leq M+1,
      \end{split}
    \end{equation}
    by choosing $\eps_1$ to be sufficiently small.
    Therefore if $0<r<R/2$, we can find $j \in \mathbb N$ such that $\tau^{j-1} \leq r/R \leq \tau^j$, so it follows that
    \begin{equation}
        \EE(u,y,r) \leq \frac{4}{\tau^n} \EE(u,y,\tau^j R/2) \leq \frac{4}{\tau^n} \tau^{2\alpha j} \EE(u,y,R/2) \leq \frac{2^{n+4}}{\tau^{n+2\alpha}} \left(\frac{r}{R}\right)^{2\alpha} \EE(u,x,R). 
    \end{equation}
    Since this holds for all $y \in B_{R/2}(x)$ and $r \in (0,R/2)$, it follows that $\B u$ is $\alpha$-H\"older continuous on $B_{R/2}(x)$.
    Since $\C^{\ast}u=0$ in addition, using Fourier multiplier estimates (see \cite[Theorem~7.9.6]{Hormander1} it follows that $u$ is $\CC^{k,\alpha}$ in $B_{R/2}(x)$.
\end{proof}

\section{Reduction and proof of Theorems \texorpdfstring{\ref{thm:main_A-free}}{A} and \texorpdfstring{\ref{thm:main}}{B}}\label{sec:proof_of_main_thm_A-free}

\par In this section, we show how to conclude the proof of the main partial regularity theorems. 
For Theorem~\ref{thm:main_A-free}, this will involve a reduction from the $\A$-free formulation to the potential formulation with $\B$; more precisely, we will show that any local minimizer $v$ of \eqref{eq:P_Omega_Afree} can be locally represented by $S(x)+\B u$, where $S$ is smooth and $\B u$ is a local minimizer of the functional
    \begin{equation}\label{eq:reduced_problem}
    \widetilde{\mathcal F}(u) = \int f(x,S(x) + \B u).
    \end{equation}
This will be carried out in Section~\ref{subsec:reduction}, where we will also establish Theorem~\ref{thm:main} by a similar reduction argument. More precisely, we will locally modify minimizers $u$ of \eqref{eq:P_Omega} to be $\C^{\ast}$-free in addition, where $\C$ is a potential operator for $\B$, to which Theorem~\ref{thm:partial_reg} applies.

Unless we impose additional growth conditions on the integrand $f$ however, we cannot directly apply the results from Section~\ref{sec:partialreg} to the modified functional \eqref{eq:reduced_problem}. Instead, we can adapt the proof of said theorem 
by keeping track of the additional term $S$, which is carried out in Section \ref{sec:reduction_regularity}.

\subsection{Reduction to the potential formulation}\label{subsec:reduction}

We will assume we are in the setting of Theorem~\ref{thm:main_A-free}; that is, $\A$ is a homogeneous differential operator satisfying \eqref{eq:CR}, $\B$ is a potential operator for $\A$ in that $\ker \A(\xi) = \rmim \B(\xi)$ for all $\xi \in \mathbb R^n\setminus\{0\}$, and $\C$ is a potential operator for $\B$.

\begin{proposition}\label{prop:reduction}
    Suppose that $f\colon \Omega\times V\to\R$ satisfies \ref{it:fgrowth}.
    Let $v\in \mathcal M(\Omega,V)$ such that $\A v=0$ is a local minimizer of \eqref{eq:P_Omega_Afree}, meaning that \eqref{eq:local_min_A} holds. 
    Then if $\omega \Subset\Omega$ is an open set, there exists $u\in \sobo^{k-1,1}(\Omega,U)$ such that $\B u\in\mathcal M(\omega,V)$, $\C^{\ast}u=0$ in $\omega$, and $S \in \CC^\infty(\bar\omega,V)$ such that
    \begin{equation}
        v = \B u + S \quad\mbox{ in }\omega.
    \end{equation}
    Moreover $u$ satisfies
    \begin{align}\label{eq:local_min_B}
    \int_\omega f(x,S(x)+\B u(x))\leq \int_\omega \tilde f(x,S(x)+\B( u+\phi)(x))\quad\text{for }\phi\in \hold_c^\infty(\omega,U)
    \end{align}
\end{proposition}

\begin{proof}
    Consider a cut-off function $\rho\in \CC_{\rmc}^\infty(\Omega)$ such that $0 \leq \rho \leq 1$ in $\Omega$ and $\rho=1$ in an open neighborhood of $\overline\omega$. We write
    $$
    \widehat{\rho v}(\xi)=\B(\xi)\B^\dagger(\xi)\widehat{\rho v}(\xi)+\A^*(\xi)\A^{*\dagger}(\xi)\widehat{\rho v}(\xi),
    $$
    which follows by the definition of the Moore--Penrose inverse and the relation between $\A$ and $\B$. Define
    $$
   \hat u(\xi)\coloneqq \B^\dagger(\xi)\widehat{\rho v}(\xi)
   \quad\text{and}\quad   
   \hat w(\xi)\coloneqq\A^{*\dagger}(\xi)\widehat{\rho v}(\xi),
    $$
    so that $\rho v=\B u+\A^* w$. Write $S\coloneqq\A^*w$, so $\B^* S=0$ in $\R^n$. Since $\A \B u=0$ on $\R^n$ and $v=\B u + S$ in $\{x \in \Omega\col \rho(x)=1\}$, we also have $\A S=0$, so that $S$ solves the elliptic system $(\A,\B^{\ast})S=0$ in an open neighborhood of $\omega$. In particular, $S\in \CC^\infty(\bar\omega,V)$ and hence $\B u \in \mathcal M(\omega,V)$. Notice that $\ima \B^{\dagger}(\xi) = \ima \B^{\ast}(\xi)$ for any $\xi \neq 0$, and thus $\C^{\ast}u=0$. Using Lemma~\ref{lem:dagger} we can also write
    \begin{equation}
        D^{k-1}u = \mathcal K \ast (\rho v), \ \ \mbox{where} \ \ \hat{\mathcal K}(\xi) = \B^\dagger\left(\frac{\xi}{|\xi|}\right)\frac{\widehat{\rho v}(\xi)}{|\xi|}\otimes \left(\dfrac{\xi}{|\xi|}\right)^{\otimes{(k-1)}},
    \end{equation}
    which satisfies $\lvert \mathcal K(x) \rvert \leq c \lvert x \rvert^{1-n}$ on $\mathbb R^n\setminus\{0\}$, and hence is locally integrable. Thus since $\Omega$ is bounded, we can estimate
    \begin{equation}
        \lVert D^{k-1}u \rVert_{\lebe^1(\Omega)} =\lVert \mathcal K \ast (\rho v) \rVert_{\lebe^1(\Omega)}   \lesssim \lVert \rho v \rVert_{\mathcal M(\Omega)} \leq \lVert v \rVert_{\mathcal M(\Omega)},
    \end{equation}
     which also implies that $u\in \sobo^{k-1,1}(\Omega,U)$.
    Finally, the minimality property \eqref{eq:local_min_B} follows from the decomposition.
\end{proof}

A similar argument allows us to reduce to the $\C^{\ast}$-free setting in the context of $\B$-gradients, which allows us to establish Theorem~\ref{thm:main}.

\begin{lemma}\label{lem:Cast_reduction}
    Let $u \in \mathscr{D}'(\Omega,U)$ such that $\B u \in \mathcal M(\Omega,V)$.
    Then if $\omega \Subset \Omega$ is an open set, there exists $\tilde u \in \sobo^{k-1,1}(\omega,U)$ such that $\B \tilde u = \B u$ and $\C^{\ast}\tilde u = 0$.
\end{lemma}

\begin{proof}
    The argument is similar to that of Proposition~\ref{prop:reduction}; take a cutoff $\rho \in \ccinfty(\Omega)$ such that $0 \leq \rho \leq 1$ in $\Omega$ and $\rho =1$ in a neighborhood of $\overline{\omega}$, and define $\tilde u$ by
    \begin{equation}
        \hat{\tilde u}(\xi) = \B^{\dagger}(\xi) \B(\xi)\widehat{\rho u}(\xi).
    \end{equation}
    Then it follows that $\B \tilde u = \B(\rho u) = \B u$ in a neighborhood of $\overline\omega$, and that $\C^{\ast}\tilde u =0$.
    Then the $\sobo^{k-1,1}$-regularity follows from ellipticity of the pair $(\B,\C^{\ast})$; we can for instance apply Lemmas~\ref{lem:negative_norm_estimate} and \ref{lem:negative_measure}.
\end{proof}

\begin{proof}[Proof of Theorem~\ref{thm:main}]
    Given $\omega \Subset \Omega$, by applying Lemma~\ref{lem:Cast_reduction} we can find $\tilde u \in \sobo^{k-1,1}(\omega,U)$ such that $\B \tilde u = \B u$ and $\C^{\ast}\tilde u =0$.
    Then since $\tilde u$ is a local minimiser to the problem \eqref{eq:P_Omega} in $\omega$, from the $\eps$-regularity result Theorem~\ref{thm:partial_reg}, we can find $\omega_0 \subset \omega$ such that $\mathscr{L}^n(\omega \setminus \omega_0)=0$ and $\tilde u \in \CC^{k,\alpha}_{\locc}(\omega_0,U)$ for each $\alpha \in (0,1)$.
    Since $\B u = \B \tilde u$, we have $\B u$ is also partially $C^{0,\alpha}$ in $\omega$.
    Finally since $\omega \Subset\Omega$ was arbitrary, the claimed partial regularity result follows.
\end{proof}

\begin{proposition}\label{prop:reduction_integrand}
    Suppose $f \colon \omega \times V \to \mathbb R$ satisfies \ref{it:fgrowth}--\ref{it:frecession}, and let $S \in \sobo^{1,\infty}(\omega,V)$.
    Then if we define the modified integrand
    \begin{equation}
        \tilde f(x,z) = f(x,S(x)+z) \quad\mbox{for all } x \in \omega, z \in V,
    \end{equation}
    we have $\tilde f$ satisfies \ref{it:fgrowth}, \ref{it:fqcstrong} and \ref{it:frecession}.
    Furtheremore, $\tilde f$ satisfies \ref{it:fcontinuity} if we additionally assume that
    \begin{equation}\label{eq:H5}
        \lvert\partial_z^2f(x,z)\rvert\leq L_0 \quad\mbox{for all }(x,z) \in\omega \times V.
    \end{equation}
\end{proposition}

\begin{proof}
    We check the required assumptions:
    \begin{align*}
        |\tilde f(x,z)|=| f(x,S(x)+z)|\leq L(1+|S(x)+z|)\leq L(1+\|S\|_{\lebe^\infty(\omega)})(1+|z|)
    \end{align*}
    for all $(x,z)\in\omega\times V$, establishing the linear growth of $\tilde f$. The smoothness assumptions are all obvious, and for the Lipschitz bound we estimate
    \begin{align*}
        |\partial_z\tilde f(x,z)-\partial_z\tilde f(y,z)|&=
        |\partial_z f(x,S(x)+z)-\partial_z f(y,S(y)+z)|\\
        &\leq |\partial_z f(x,S(x)+z)-\partial_z f(y,S(x)+z)|\\
        &\quad+|\partial_z f(y,S(x)+z)-\partial_z f(y,S(y)+z)| \\
        &\leq (L + L_0 \lVert S \rVert_{\sobo^{1,\infty}(\omega)}) \lvert x - y \rvert,
    \end{align*}
    where we used \ref{it:fgrowth} and \eqref{eq:H5}.
    The quasiconvexity is easy to transport: write $F=f-\ell E$ and $\tilde F=\tilde f-\ell E$, so that $F$ is $\A$-quasiconvex. Consider a cube $Q\supset \omega$. Then any vector field $v=z+B\phi$ for $\phi\in \CC_{\rmc}^\infty(\omega,U)$ has a  $Q$-periodic extension which is $\A$-free. Therefore for fixed $x_0\in\omega$
    \begin{align*}
    \tilde F(x_0,z)&= F(x_0,S(x_0)+z)= F\left(x_0,\fint_Q S(x_0)+ v\dif x\right)\\
    &\leq \fint_Q F(x_0,S(x_0)+v)\dif x=\fint_Q \tilde F(x_0,v)\dif x=\fint_Q \tilde F(x_0,z+\B\phi)\dif x,
    \end{align*}
    which establishes the appropriate quasiconvexity of $\tilde f$.
    Finally, for \ref{it:frecession}, we simply observe that $f^{\infty}(x,z) = {\tilde f}^{\infty}(x,z)$.
\end{proof}

Therefore if we additionally assume that $\partial_z^2f$ is bounded, by Proposition~\ref{prop:reduction} we obtain a decomposition $v = \B u + S$ in $\omega$, so setting $\tilde f(x,z) = f(x,S(x)+z)$ we have $u$ locally minimizes the integrand $w \mapsto \int_{\omega} \tilde f(x,\B w)$. Applying Theorem~\ref{thm:main}, noting that $\tilde f$ satisfies the necessary assumptions by Proposition~\ref{prop:reduction_integrand}, we infer that $u$ and hence $v$ is partially regular in $\omega$, establishing Theorem~\ref{thm:main_A-free} under this added assumption \eqref{eq:H5}.

In the absence of the second derivative bound, we cannot directly apply Theorem~\ref{thm:main}. Nevertheless, we will show in the next section that partial regularity still holds.

\subsection{Partial regularity: proof of Theorem \texorpdfstring{\ref{thm:main_A-free}}{A}}\label{sec:reduction_regularity}
We will sketch how to establish partial regularity if our functional takes the form
\begin{equation}\label{eq:tildeE}
    \widetilde{\mathcal F}(u) = \int_{\Omega} f(x,S(x) + \B u),
\end{equation}
arising from Proposition~\ref{prop:reduction}.

\begin{theorem}\label{thm:A_regularity}
    Let $\B$ be a differential operator satisfying \eqref{eq:CR}, suppose $f \colon \Omega \times V \to \mathbb R$ satisfies \ref{it:fgrowth}--\ref{it:frecession}, and let $S \in \sobo^{1,\infty}(\Omega,V)$.
    Then for all $M>0$ and $\alpha \in (0,\frac12)$, there exists $\eps>0$ such that the following holds; if $u \in \sobo^{k-1,1}(\Omega,U)$ is a local minimizer of \eqref{eq:tildeE} satisfying $\C^{\ast}u=0$, then for any $B_R(x_0) \Subset \Omega$ for which
    \begin{equation}
        \lvert (\B u)_{B_R(x_0)}\rvert \leq M, \quad \fint_{B_R(x_0)} E(\B u - (\B u)_{B_R(x_0)}) < \eps,
    \end{equation}
    we have $u \in \CC^{k,\alpha}(\overline B_{R/2}(x_0),U)$.
\end{theorem}

The proof is similar to that of Theorem~\ref{thm:partial_reg}, except that we keep track of the additional $S(x)$ term in the integrand, whose presence also prevents us from reducing to a setting where \eqref{eq:SC} holds. This will follow by establishing analogues of the Caccioppoli inequality (Proposition~\ref{prop:Caccioppoli}) and the almost $A$-harmonicity estimate (Lemma~\ref{lem:AHalmost_har}); from here the argument is entirely analogous to obtain an excess decay estimate, where the $A$-harmonic approximation lemma requires no modification, except that we apply it with a different choice of $A$.

\begin{proposition}\label{prop:Caccioppoli_S}
     Suppose that the integrand $f$ satisfies the  assumptions \textnormal{\ref{it:fgrowth}}, \textnormal{\ref{it:fcontinuity}}, \textnormal{\ref{it:fBqcstrong}} and \textnormal{\ref{it:frecession}}, $S \in \sobo^{1,\infty}(\Omega,V)$, $u \in \sobo^{k-1,1}(\Omega,U)$ such that $\B u\in \mathcal{M}(\Omega,V)$ is a local minimizer of \eqref{eq:tildeE}, and the map $a\col \Omega \to V$ is a polynomial of order $k$.
     Then for any $M>0$, $\tau \in (0,1)$ and $B_R(x_0) \Subset \Omega$ such that $\lvert S(x_0) \rvert + \lvert \B a \rvert \leq M$, we have the estimate
    \begin{equation}
    \begin{split}
        \int_{B_{R/2}(x_0)} E(\B (u-a)) 
        &\leq c \sum_{i=0}^{k-1} \int_{B_R(x_0)} E\bigg(\frac{D^{i}(u-a)}{R^{k-i}}\bigg) \,\dd x \\
        &\quad +c R \int_{B_R(x_0)} \left[ \lvert \B (u-a) \rvert + \sum_{i=0}^{k-1} \frac{\lvert D^{i}(u-a)\rvert}{R^{k-i}}\right] \,\dd x\\
        &\quad + c R^{n+1} \lVert S \rVert_{\lebe^{\infty}(B_R(x_0))}.
    \end{split}
    \end{equation} 
    where $c=c(n,V,\B,M,L,\ell,L_{x_0,R,M+1})>0$.
\end{proposition}

\begin{proof}
    We will proceed as in the proof of Proposition~\ref{prop:Caccioppoli} and only point out the necessary modifications. Let $\rho$ be a cutoff such that $\mathbbm{1}_{B_s} \leq \rho \leq \mathbbm{1}_{B_t}$ and put $\tilde u = u-a$, $\varphi = \rho \tilde u$ and $\psi = (1-\rho)\tilde u$.
    Also by letting $z_0 \coloneqq S(x_0)+\B a$, we consider the shifted functional $\tilde f = f_{x_0,z_0}$ from \eqref{eq:shifted_int}.
    By the strong quasiconvexity estimate applied with $\varphi$ and local minimiality of $u$ as in \eqref{eq:CaccioppoliSupmid1}, we can estimate
    \begin{equation}
        \begin{split}
        \int_{B_s} E(\B \varphi) 
        &\lesssim \int_{B_t} \tilde f(\B \psi) + \int_{B_t} (\tilde f(\B \tilde u) - f(x,S+\B u)) \\
        &\quad+ \int_{B_t} (f(x,S + \B \psi + \B a)-\tilde f(\B \psi)) + \int_{B_t} (\tilde f(\B u - \B \psi)-\tilde f(\B \tilde u)) \\
        &=: I + \widetilde{I\!I} + \widetilde{I\!I\!I} + I\!V,
        \end{split}
    \end{equation}
    noting we have grouped the terms differently compared to \eqref{eq:CaccioppoliSupmid1}.
    The main difference lies in estimating $\widetilde{I\!I} + \widetilde{I\!I\!I}$; using \eqref{eq:fz_bounded} we can estimate
    \begin{equation}
    \begin{split}
        \widetilde{I\!I} &= \int_{B_t} (\tilde f(\B \tilde u) - f(x,S(x_0)+\B u)) + \int_{B_t} (f(x,S(x)+\B u)- f(x,S(x_0)+\B u))\\
        &\leq \int_{B_t} (\tilde f(\B \tilde u) - f(x,S(x_0)+\B u)) + cL \lVert DS \rVert_{\lebe^{\infty}(B_t)}\:\! t^{n+1},
        \end{split}
    \end{equation}
    and similarly
    \begin{equation}
        \widetilde{I\!I\!I} \leq  \int_{B_t} (f(x,S(x_0)+\B \psi+\B a) - \tilde f(\B \psi) ) + cL \lVert DS \rVert_{\lebe^{\infty}(B_t)} \:\!t^{n+1}.
    \end{equation}
    Hence arguing exactly as in \eqref{eq:CaccioppoliSupmid3} we have
    \begin{equation}
        \widetilde{I\!I} + \widetilde{I\!I\!I} \leq c\:\!t\left[\int_{B_t} \abss{\B \tilde{u}}  + \sum_{i=1}^{k}\int_{B_t} \frac{\abss{D^{k-i}\tilde{u}}}{(t-s)^{i}}\dif x + t^n\lVert D S \rVert_{\lebe^{\infty}(B_t)}\right].
    \end{equation}
    The terms $I$ and $I\!V$ are then estimated as in \eqref{eq:CaccioppoliSupmid2} and \eqref{eq:CaccioppoliSupmid4} respectively, which combines to give
    \begin{equation}
      \begin{split}
        \int_{B_s} E(\B \tilde u) 
        &\leq c_1 \int_{B_t \setminus B_s} E(\B \tilde u) + c \sum_{i=0}^{k-1} \int_{B_t} E\bigg(\frac{D^{i}\tilde u}{(t-s)^{k-i}}\bigg) \,\dd x \\
        &\quad +c t \left[ \int_{B_t} \lvert \B \tilde u \rvert + \sum_{i=0}^{k-1} \int_{B_t} \frac{\lvert D^{i}\tilde u\rvert}{(t-s)^{k-i}} \,\dd x + t^n\lVert DS \rVert_{\lebe^{\infty}(B_t)}\right],
        \end{split}
    \end{equation}
    from which the desired inequality follows by the hole filling trick and iteration.
\end{proof}

\begin{lemma}\label{lem:AHalmost_har_S}
    Let $f$ satisfy \textnormal{\ref{it:fgrowth}}, \textnormal{\ref{it:fcontinuity}} and \textnormal{\ref{it:fBqcstrong}}, $S \in \sobo^{k,\infty}(\Omega)$, $u \in \sobo^{k-1,1}(\Omega)$ be a local minimiser of \eqref{eq:tildeE}, and fix $M>0$ and $\alpha\in(0,1)$.
    Then for any ball $B_R(x_0)\Subset \Omega$ with $\lvert S(x_0)\rvert+\abss{(\B u)_{x_0,R}}<M$, setting $A := \partial_z^2f(x_0,S(x_0)+(\B u)_{x_0,R})$, we have $u$ satisfies
    \begin{equation}
        \fint_{B_R(x_0)} A[\B u,\B \varphi] \leq c_1(\EE^{\frac{1}{2\alpha}}(R)+ \EE(R)+\omega^{\frac{1}{2}}(c_1\EE^{\frac{1}{2}}(R)) \EE^{\frac{1}{2}}(R)) \sup_{B_R(x_0)}\abss{\B \varphi},
    \end{equation}
    for any $\varphi \in \CC_{\rmc}^{\infty}(B_R(x_0))$, where $\mathscr{E}(R) := \mathscr{E}(u,x_0,R)$ is the excess from \eqref{eq:AHexcess}, $\omega := \omega_{x_0,R,M+1}$ and $c_1 = c_1(n,V,L,L_{x_0,R,M+1},\lVert S \rVert_{\sobo^{1,\infty}(B_R)})>0$.
\end{lemma} 

We note the estimate is identical to that in Lemma~\ref{lem:AHalmost_har}, except the associated constants now depend on $\lVert S \rVert_{\sobo^{1,\infty}(B_R)}$.

\begin{proof}
    Writing $B_R =B_R(x_0)$, let $z_0 : = S(x_0) + (\B u)_{R}$ and $A = \partial_z^2f(x_0,z_0)$.
    The Euler-Lagrange system reads as
    \begin{equation}\label{eq:EL_system_S}
        \fint_{B_R} \partial_zf(x,S(x)+\B^{\rma}u) \cdot \B \varphi \,\dd x = 0 \quad\mbox{for all }\varphi \in \CC^{\infty}_{\rmc}(B_R).
    \end{equation}
    Set $Z(x) = S(x)+\B^{\rma}u(x)$.
    We will perform an analogous argument as in the proof of Lemma~\ref{lem:AHalmost_har}.
    We consider the sets
    \begin{equation}
        E^+ := \{x \in B_R : \lvert \B^a u - (\B u)_R\rvert \geq 1\}, \quad E^- := \{x \in B_R : \lvert \B^a u - (\B u)_R\rvert < 1\},
    \end{equation}
    and given any $\varphi \in \CC^{\infty}_{\rmc}(B_R)$ we will consider the splitting
    \begin{equation*}
        \begin{split}
            &\fint_{B_R} A[\B u,\B \varphi] \,\dd x 
            = \fint_{B_R} (\partial_zf(x,Z)-\partial_zf(x_0,Z))\cdot \B \varphi \,\dd x\\
            &\quad+ \frac1{\lvert B_R\rvert} \int_{E^+} ( A[\B^{\rma}u - (\B u)_{R},\B\varphi] + (\partial_zf(x_0,Z)-\partial_zf(x_0,S(x_0)+(\B u)_R))\cdot \B \varphi)\\            
            &\quad+ \frac1{\lvert B_R\rvert} \int_{E^-} ( A[\B^{\rma}u - (\B u)_{R},\B\varphi] + (\partial_zf(x_0,Z)-\partial_zf(x_0,S(x_0)+(\B u)_R))\cdot \B \varphi)\\
            &\quad+\fint_{B_R} A[\B^{\rms}u,\B\varphi] =: J_0 + J_1 + J_2 + J_3.
        \end{split}
    \end{equation*}
    We will estimate the terms separately; for this let $M > \lvert z_0 \rvert$ and abbreviate $\tilde L = L_{x_0,R,M=1}$, $\omega = \omega_{x_0,R,M+1}$, using the notation from \eqref{eq:loc_bound}, \eqref{eq:loc_cont}.
    Then using \ref{it:fcontinuity} we have
    \begin{equation}
        \lvert J_0 \rvert \leq L \mathscr{E}^{\frac1{2\alpha}}(R) \sup_{B_R} \,\lvert\B\varphi\rvert.
    \end{equation}
    By \eqref{eq:fz_bounded} we have
    \begin{equation}
        \lvert J_1 \rvert \leq c (\int_{E^+} \lvert \B^{\rma}u - (\B^{\rma}u)_R \vert \dif x + \lvert E^+\rvert) \leq c \mathscr{E}(R) \sup_{B_R}\, \lvert\B \varphi\rvert,
    \end{equation}
    which is obtained in a similar way to Lemma \ref{lem:AHalmost_har}.
    Since $\lvert A \rvert \leq \tilde L$ we have
    \begin{equation}
        \lvert J_3 \rvert \leq \tilde L \mathscr{E}(R) \,\sup_{B_R} \lvert \B \varphi\rvert
    \end{equation}
    For the remaining term, using Lemma~\ref{lem:shifted_continuity}, the concavity of $\omega$ and the fact $\abss{Z-z_0} \leq \normm{DS}_{\lebe^{\infty}(B_R)} R + \abss{\B^a u - (\B u)_R}$, we obtain
    \begin{equation}
        \begin{split}
        \lvert J_2 \rvert 
        &\leq c \int_{E^-} \omega(\lvert Z - z_0 \rvert) \lvert Z-z_0\rvert \lvert \B \varphi \rvert \,\dd x \\
        &\leq c \sup_{B_R}\abss{\B \varphi} \,\omega^{\frac12}(c(R+\mathscr{E}^{\frac12}(R)))(R+\mathscr{E}^{\frac12}(R)).
        \end{split}
    \end{equation}
    If $R \leq 1$, then we have $R \leq R^{\alpha} \leq \mathscr{E}^{\frac12}(R)$.
    Hence by combining the above estimates we obtain
    \begin{equation}
        \fint_{B_R} A[\B u,\B \varphi] \leq c(\EE^{\frac{1}{2\alpha}}(R)+ \EE(R)+\omega^{\frac{1}{2}}(c\EE^{\frac{1}{2}}(R)) \EE^{\frac{1}{2}}(R)) \sup_{B_R}\abss{\B \varphi},
    \end{equation}
    as required.
\end{proof}

\begin{proof}[Proof of Theorem~\ref{thm:A_regularity}]
  By arguing as in Proposition~\ref{prop:EDEAH}, for each $M>0$, $\alpha \in (0,\frac12)$ and $\tau \in (0,1)$ we can find $\eps_{\mathrm{exc}}\in(0,1)$ such that if $B_{2R_0}(x_0) \Subset \Omega$ we have
  \begin{equation}
      \lVert S \rVert_{\sobo^{1,\infty}(B_{2R_0}(x_0))} + \lvert (\B u)_{B_R(x)}\rvert \leq M, \quad\EE(u,x,R) < \eps_{\mathrm{exc}}
  \end{equation}
  for each $x \in B_{R_0}(x_0)$ and $0<R<R_0$, then we obtain the estimate
  \begin{equation}\label{eq:excess_decay_one}
      \EE(u,x,\tau R) - (\tau R)^{2\alpha} \leq \tau^{2\alpha} \EE(u,x,R) + c \lVert S \rVert_{\sobo^{1,\infty}(B_{2R}(x_0)} \tau R.
  \end{equation}
  Indeed we apply Lemma~\ref{lem:AHapproximation2} with $A_{x,R} = \partial_z^2f(x,S(x)+(\B u)_{x,R})$, using Lemma~\ref{lem:AHalmost_har_S} to verify the hypothesis \eqref{eq:AHalmost_harmonic}.
  We then chain this with the Caccioppoli inequality from Proposition~\ref{prop:Caccioppoli_S}, however we obtain an extra term of the form $c \lVert S \rVert_{\sobo^{1,\infty}(B_{2R}(x_0))} \tau R$ in the estimate \eqref{eq:main_excess_split}.
  Keeping this extra term we arrive at \eqref{eq:excess_decay_one}.

  Now since $\alpha < \frac12$, for $R \leq R_0$ we have $R \leq R_0^{1-2\alpha}R^{2\alpha}$, so shrinking $R_0$ as necessary we obtain the estimate
  \begin{equation}
      \EE(u,x,\tau R) - (\tau R)^{2\alpha} \leq 2\tau^{2\alpha} \EE(u,x,R).
  \end{equation}
  Then by iteration as in Section \ref{subsec:iteration} we infer that $u \in \CC^{k,\alpha}(B_{R/2}(x))$ for any $\alpha \in (0,1/2)$.
\end{proof}

Finally we conclude with:
\begin{proof}[Proof of Theorem~\ref{thm:main_A-free}]
    Let $\omega \Subset \Omega$ and apply Proposition~\ref{prop:reduction} to decompose $v = \B u +S$ in $\omega$, where $u \in \sobo^{k-1,1}(\omega,U)$, $\B u \in \mathcal M(\omega,V)$, $\C^{\ast}u=0$ in $\omega$ and $S \in \CC^{\infty}(\overline\omega,V)$.
    Furthermore $u$ locally minimizes the functional
    \begin{equation}
        \tilde \F(u) = \int_{\omega} f(x,S(x)+\B u),
    \end{equation}
    in that \eqref{eq:local_min_B} holds.
    Therefore by the $\eps$-regularity theorem of Theorem~\ref{thm:A_regularity}, there exists $\omega_0 \subset \omega$ such that $\mathscr L^n(\omega\setminus\omega_0) = 0$ and that $u \in \CC^{k,\alpha}(\omega_0)$ for all $\alpha \in (0,1/2)$.
    Since $\omega \Subset \Omega$ was arbitrary, the claimed partial regularity statement follows.
\end{proof}

\begin{remark}
    In the case of $p$-growth with $p>1$, by a similar argument one can also remove the second derivative bound $\lvert \partial_z^2f(x,z)\rvert \lesssim 1 + \lvert z \rvert^{p-2}$ imposed in \cite[Theorem A]{LiRaita24}.
\end{remark}

\printbibliography

@article{AF87,
	author = {Acerbi, E. and Fusco, N.},
	title = {A regularity theorem for minimizers of quasiconvex integrals},
	journal = {Arch. Rational Mech. Anal.},
	number = {3},
	volume = {99},
	pages = {261-281},
	year = {1987},
	doi = {10.1007/BF00284509}}

@article{AF89,
	title = {Regularity for minimizers of nonquadratic functionals: the case $1<p<2$},
	author = {Acerbi, E. and Fusco, N.},
	journal = {J. Math. Anal. Appl.},
	volume = {140},
	number = {1},
	pages = {115-135},
	year = {1989},
	doi = {10.1016/0022-247X(89)90098-X}}

@article{Alberti1993,
  title = {Rank One Property for Derivatives of Functions with Bounded Variation},
  author = {Alberti, Giovanni},
  date = {1993-01},
  journaltitle = {Proceedings of the Royal Society of Edinburgh Section A: Mathematics},
  volume = {123},
  number = {2},
  pages = {239--274},
  issn = {1473-7124, 0308-2105},
  doi = {10.1017/S030821050002566X}
}

@book{BerghLofstrom1976,
  title = {Interpolation Spaces},
  author = {Bergh, J{\"o}ran and L{\"o}fstr{\"o}m, J{\"o}rgen},
  year = {1976},
  series = {Grundlehren Der Mathematischen Wissenschaften},
  volume = {223},
  publisher = {Springer Berlin Heidelberg},
  address = {Berlin, Heidelberg},
  doi = {10.1007/978-3-642-66451-9},
  isbn = {978-3-642-66453-3}}

@book{Bild03,
	author = {Bildhauer, M.},
	title = {Convex variational problems. Linear, nearly linear and anisotropic growth conditions},
	publisher = {Springer-Verlag},
	address = {Berlin},
	series = {Lecture notes in Mathematics},
	volume = {1818},
	year = {2003},
	doi = {10.1007/b12308}}

@article{BS13,
	author = {Beck, L. and Schmidt, T.},
	title = {On the Dirichlet problem for variational integrals in {$BV$}},
	journal = {J. Reine Angew. Math.},
	volume = {674},
	pages = {113-194},
	year = {2013},
	doi = {10.1515/CRELLE.2011.188}}

@article{ChenKris17,
	title = {On coercive variational integrals},
	author = {Chen, C.-Y. and Kristensen, J.},
	journal = {Nonlinear analysis},
	volume = {153},
	pages = {213-229},
	year = {2017},
	doi = {10.1016/j.na.2016.09.011}}

@article{CFM98,
	author = {Carozza, M. and Fusco, N. and Mingione, G.},
	title = {Partial regularity of minimizers of quasiconvex integrals with subquadratic growth},
	journal = {Ann. Mat. Pura Appl. (4)},
	volume = {175},
	pages = {141-164},
	year = {1998},
	doi = {10.1007/BF01783679}}

@book{DHHR2011,
  title = {Lebesgue and {{Sobolev Spaces}} with {{Variable Exponents}}},
  author = {Diening, Lars and Harjulehto, Petteri and Hästö, Peter and Ruzicka, Michael},
  date = {2011},
  series = {Lecture {{Notes}} in {{Mathematics}}},
  volume = {2017},
  publisher = {Springer},
  location = {Berlin, Heidelberg},
  doi = {10.1007/978-3-642-18363-8},
  url = {https://link.springer.com/10.1007/978-3-642-18363-8},
  urldate = {2025-11-28},
  isbn = {978-3-642-18362-1},
  langid = {english}
}

@article{DuzaarGrotowski2000,
  title = {Optimal Interior Partial Regularity for Nonlinear Elliptic Systems: The Method of {{A-harmonic}} Approximation},
  shorttitle = {Optimal Interior Partial Regularity for Nonlinear Elliptic Systems},
  author = {Duzaar, Frank and Grotowski, Joseph F.},
  date = {2000-11-01},
  journaltitle = {manuscripta mathematica},
  shortjournal = {manuscripta math.},
  volume = {103},
  number = {3},
  pages = {267--298},
  issn = {1432-1785},
  doi = {10.1007/s002290070007},
  url = {https://doi.org/10.1007/s002290070007}
}

@article{DGG00,
	author = {Duzaar, F. and Gastel, A. and Grotowski, J. F.},
	title = {Partial regularity for almost minimizers of quasi-convex integrals},
	journal = {SIAM J. Math. Anal.},
	number = {3},
	volume = {32},
	pages = {665-687},
	year = {2000},
	doi = {10.1137/S0036141099374536}}

@article{DGK05,
	author = {Duzaar, F. and Grotowski, J. F. and Kronz, M.},
	title = {Regularity of almost minimizers of quasi-convex variational integrals with subquadratic growth},
	journal = {Ann. Mat. Pura Appl. (4)},
	number = {4},
	volume = {184},
	pages = {421-448},
	year = {2005},
	doi = {10.1007/s10231-004-0117-5}}

@article{Cruz-UribeDiening2019,
  title = {Sharp $\mathcal A$--Harmonic Approximation},
  author = {Cruz-Uribe, David and Diening, Lars},
  date = {2019-01-25},
  journaltitle = {Applicable Analysis},
  volume = {98},
  number = {1--2},
  pages = {374--380},
  publisher = {Taylor \& Francis},
  issn = {0003-6811},
  doi = {10.1080/00036811.2017.1422729},
  url = {https://doi.org/10.1080/00036811.2017.1422729},
  urldate = {2022-05-16}
}

@article{DM04,
	author = {Duzaar, F. and Mingione, G.},
	title = {Regularity for degenerate elliptic problems},
	journal = {Ann. Inst. H. Poincar{\'e} Anal. Non Lin{\'e}aire},
	volume = {21},
	number = {5},
	pages = {735-766},
	year = {2004},
	doi = {10.1016/J.ANIHPC.2003.09.003}}

@article{DuSt02,
	author = {Duzaar, F. and Steffen, K.},
	title = {Optimal interior and boundary regularity for almost minimizers to elliptic variational integrals},
	journal = {J. Reine Angew. Math.},
	volume = {546},
	pages = {73-138},
	year = {2002},
	doi = {10.1515/crll.2002.046}}

@online{EitlerLewintan2025,
  title = {On $\mathrm{BV}^{\mathbb A}$-{{Minimisers}} in Two {{Dimensions}}},
  author = {Eitler, Ferdinand and Lewintan, Peter},
  date = {2025-08-14},
  eprint = {2508.10508},
  eprinttype = {arXiv},
  eprintclass = {math}
}

@article{Evans86,
	author = {Evans, L. C.},
	title = {Quasiconvexity and partial regularity in the calculus of variations},
	journal = {Arch. Rational Mech. Anal.},
	number = {3},
	volume = {95},
	pages = {227-252},
	year = {1986},
	doi = {10.1007/BF00251360}}

@misc{Federico2023,
      title={Partial regularity for $BV^\mathcal{B}$ minimizers}, 
      author={Federico Franceschini},
      year={2023},
      eprint={2310.20002},
      archivePrefix={arXiv},
      primaryClass={math.AP},
%      url={https://arxiv.org/abs/2310.20002}, 
}

@article{FonMul99,
	author = {Fonseca, I. and M\"{u}ller, S.},
	title = {${\mathcal{A}}$-quasiconvexity, lower semicontinuity, and {Y}oung measures},
	journal = {SIAM J. Math. Anal.},
	number = {6},
	volume = {30},
	pages = {1355-1390},
	year = {1999},
	doi = {10.1137/S0036141098339885}}

@article{GiaMod86,
	title = {Partial regularity of minimizers of quasiconvex integrals},
	author = {Giaquinta, M. and Modica, G.},
	journal = {Ann. Inst. H. Poincar\'{e} Anal. Non Lin\'{e}aire},
	volume = {3},
	number = {3},
	pages = {185-208},
	year = {1986}}

@article{GK19,
	author = {Gmeineder, F. and Kristensen, J.},
	title = {Sobolev regularity for convex functionals on {$BD$}},
	journal = {Calc. Var. Partial Differential Equations},
	volume = {58},
	number = {56},
	year = {2019},
	doi = {10.1007/s00526-019-1491-6}}

@article{Gme20,
	title = {The regularity of minima for the {Dirichlet} problem on {BD}},
	author = {Gmeineder, F.},
	journal = {Arch. Rational Mech. Anal.},
	volume = {237},
	pages = {1099-1171},
	year = {2020},
	doi = {10.1007/s00205-020-01507-5}}

@article{Gme21,
	title = {Partial regularity for symmetric quasiconvex functionals on {BD}},
	author = {Gmeineder, F.},
	journal = {J. Math. Pures Appl.},
	volume = {145},
	pages = {83-129},
	year = {2021},
	doi = {10.1016/j.matpur.2020.09.005}}

@article{GMS79,
	author = {Giaquinta, M. and Modica, G. and Sou$\check{c}$ek, J.},
	title = {Functionals with linear growth in the calculus of variations. {I, II}},
	journal = {Comment. Math. Univ. Carolin.},
	number = {1},
	volume = {20},
	pages = {143-156,157-172},
	year = {1979},
	Zbl = {0409.49006/0409.49007}}

@article{GR22,
	author = {Guerra, A. and Rai\cb{t}\u{a}, B.},
	title = {Quasiconvexity, null {L}agrangians, and {H}ardy space integrability under constant rank constraints},
	journal = {Arch. Ration. Mech. Anal.},
	volume = {245},
	number = {1},
	pages = {279-320},
	year = {2022},
	doi = {10.1007/s00205-022-01775-3}}

@book{GT01,
	author = {Gilbarg, D. and Trudinger, N. S.},
	title = {Elliptic partial differential equations of second order},
	series = {Classics in Mathematics},
	publisher = {Springer-Verlag, Berlin},
	year = {2001},
	note = {Reprint of the 1998 edition}}

@book{Hormander2000,
  title = {An introduction to complex analysis in several variables},
  author = {Hörmander, Lars},
  date = {2000},
  series = {North-Holland mathematical library},
  edition = {3},
  number = {7},
  publisher = {North-Holland},
  location = {Amsterdam},
  isbn = {978-0-444-88446-6},
  langid = {ger eng},
  pagetotal = {254}
}

@article{KriMin07,
	author = {Kristensen, J. and Mingione, G.},
	title = {The singular set of {Lipschitzian} minima of multiple integrals},
	journal = {Arch. Ration. Mech. Anal.},
	number = {2},
	volume = {184},
	pages = {341-369},
	year = {2007},
	doi = {10.1007/s00205-006-0036-2}}

@article{KriRin10a,
	author = {Kristensen, J. and Rindler, F.},
	title = {Relaxation of signed integral functionals in {$BV$}},
	journal = {Calc. Var. Partial Differential Equations},
	number = {1-2},
	volume = {37},
	pages = {29-62},
	year = {2010},
	doi = {10.1007/s00526-009-0250-5}}

@misc{LiRaita24,
    author ={Li, Z. and Rai\cb{t}\u{a}, B.},
    title = {Partial regularity and higher integrability for $\mathscr{A}$-quasiconvex variational problems},
      year={2024},
      eprint={2412.10363v2},
      archivePrefix={arXiv},
      primaryClass={math.AP},
%      url={https://arxiv.org/abs/2412.10363}
}

@article{Marcellini1996,
  title = {Everywhere Regularity for a Class of Elliptic Systems without Growth Conditions},
  author = {Marcellini, Paolo},
  date = {1996},
  journaltitle = {Annali della Scuola Normale Superiore di Pisa - Classe di Scienze},
  volume = {23},
  number = {1},
  pages = {1--25},
  publisher = {Scuola normale superiore},
  issn = {0391-173X},
  url = {https://eudml.org/doc/urn:eudml:doc:84224},
  urldate = {2022-08-04},
  langid = {english}
}

@article{Morrey52,
	author = {Morrey, C. B., Jr.},
	title = {Quasi-convexity and the lower semicontinuity of multiple integrals},
	journal = {Pacific J. Math},
	volume = {2},
	pages = {25-53},
	year = {1952},
	doi = {10.2140/pjm.1952.2.25}}

@article{Murat81,
	title={Compacit\'{e} par compensation: condition n\'{e}cessaire et suffisante de continuit\'{e} faible sous une hypoth\`{e}se de rang constant},
	author={Murat, F.},
	journal={Ann. Scuola Norm. Sup. Pisa CI. Sci. (4)},
	language = {French},
	volume={8},
	number = {1},
	pages={69-102},
	year={1981},
	zbl = {0437.35004}}

@article{Raita19,
	author = {Rai\cb{t}\u{a}, B.},
	title = {Potentials for ${\mathscr{A}}$-quasiconvexity},
	journal = {Calc. Var. PDE},
	volume = {58},
	number = {105},
	year = {2019},
	doi = {10.1007/s00526-019-1544-x}}

@book{Simon1983,
  title = {Lectures on Geometric Measure Theory},
  author = {Simon, Leon},
  date = {1983},
  pages = {272},
  publisher = {Centre for Mathematical Analysis, Australian National University},
  url = {https://catalogue.nla.gov.au/Record/2755868},
  isbn = {0-86784-429-9}
}

@book{Simon1996,
  title = {Theorems on Regularity and Singularity of Energy Minimizing Maps},
  author = {Simon, Leon},
  date = {1996},
  publisher = {Birkhäuser Basel},
  location = {Basel},
  doi = {10.1007/978-3-0348-9193-6},
  url = {http://link.springer.com/10.1007/978-3-0348-9193-6},
  urldate = {2021-11-09},
  isbn = {978-3-7643-5397-1},
  langid = {english}
}

@online{Stephan2024,
  title = {Partial Regularity for $\mathbb A$-Quasiconvex Functionals with {{Orlicz}} Growth},
  author = {Stephan, Paul},
  date = {2024-12-12},
  eprint = {2412.09478},
  eprinttype = {arXiv},
  eprintclass = {math}
}

@incollection{Tartar83,
	title={The compensated compactness method applied to systems of conservation laws},
	author={Tartar, L.},
	booktitle={Systems of nonlinear partial differential equations},
	pages={263-285},
	year={1983},
	publisher={Springer},
	address = {Dordrecht}}

@article{Triebel73,
	author = {Triebel, H.},
	title = {Spaces of distributions of Besov type on Euclidean $n$-space. {D}uality, interpolation},
	journal = {Ark. Mat.},
	volume = {11},
	pages = {13-64},
	year = {1973},
	doi = {10.1007/BF02388506}}

@article{VSC13,
	author = {Van Schaftingen, J.},
	title = {Limiting {S}obolev inequalities for vector fields and canceling linear differential operators},
	journal = {J. Eur. Math. Soc.},
	volume = {15},
	number = {3},
	pages = {877-921},
	year = {2013},
	doi = {10.4171/JEMS/380}}

@article{Dacorogna_JEMS,
  title={Calculus of variations with differential forms.},
  author={Bandyopadhyay, Saugata and Dacorogna, Bernard and Sil, Swarnendu},
  journal={Journal of the European Mathematical Society (EMS Publishing)},
  volume={17},
  number={4},
  year={2015}
}

@article{ARDPR,
  title={Lower semicontinuity and relaxation of linear-growth integral functionals under PDE constraints},
  author={Arroyo-Rabasa, Adolfo and De Philippis, Guido and Rindler, Filip},
  journal={Advances in calculus of variations},
  volume={13},
  number={3},
  pages={219--255},
  year={2020},
  publisher={De Gruyter}
}

@article{KriRai22,
  title={Oscillation and concentration in sequences of PDE constrained measures},
  author={Kristensen, Jan and Rai\cb{t}{\u{a}}, Bogdan},
  journal={Archive for Rational Mechanics and Analysis},
  volume={246},
  number={2},
  pages={823--875},
  year={2022},
  publisher={Springer}
}

@article{GmeKri_BV,
  title={Partial regularity for BV minimizers},
  author={Gmeineder, Franz and Kristensen, Jan},
  journal={Archive for Rational Mechanics and Analysis},
  volume={232},
  number={3},
  pages={1429--1473},
  year={2019},
  publisher={Springer}
}

@article{GmeinederKristensen2024,
  title = {Quasiconvex {{Functionals}} of (p, q)-{{Growth}} and the {{Partial Regularity}} of {{Relaxed Minimizers}}},
  author = {Gmeineder, Franz and Kristensen, Jan},
  date = {2024-10},
  journaltitle = {Archive for Rational Mechanics and Analysis},
  shortjournal = {Arch Rational Mech Anal},
  volume = {248},
  number = {5},
  pages = {80},
  issn = {0003-9527, 1432-0673},
  doi = {10.1007/s00205-024-02013-8},
  url = {https://link.springer.com/10.1007/s00205-024-02013-8},
  urldate = {2026-03-23}
}

@article{CG22,
  title={A-quasiconvexity and partial regularity},
  author={Conti, Sergio and Gmeineder, Franz},
  journal={Calculus of Variations and Partial Differential Equations},
  volume={61},
  number={6},
  pages={215},
  year={2022},
  publisher={Springer}
}

@article{BK22,
  title={Partial Regularity for ${\mathbb {A}}$-quasiconvex Functionals},
  author={B{\"a}rlin, Matthias and Ke{\ss}ler, Konrad},
  journal={arXiv preprint arXiv:2203.00153},
  year={2022}
}

@article{Schiffer24,
  title={$\mathscr{A}$-free truncation and higher integrability of minimisers},
  author={Schiffer, Stefan},
  journal={Calculus of Variations and Partial Differential Equations},
  volume={64},
  number={1},
  pages={1--17},
  year={2025},
  publisher={Springer}
}

@article{Raita_pot_new,
  title={A simple construction of potential operators for compensated compactness},
  author={Rai\cb{t}{\u{a}}, Bogdan},
  journal={The Quarterly Journal of Mathematics},
  year={2024},
  publisher={Oxford University Press UK}
}

@article{KK16,
  title={On rank one convex functions that are homogeneous of degree one},
  author={Kirchheim, Bernd and Kristensen, Jan},
  journal={Archive for rational mechanics and analysis},
  volume={221},
  pages={527--558},
  year={2016},
  publisher={Springer}
}

@article{SIL3,
  title={Calculus of variations: a differential form approach},
  author={Sil, Swarnendu},
  journal={Advances in Calculus of Variations},
  volume={12},
  number={1},
  pages={57--84},
  year={2019},
  publisher={De Gruyter}
}

@article{Schmidt2008,
  title = {Regularity of Minimizers of $W^{1,p}$-Quasiconvex Variational Integrals with \$(p,q)\$-Growth},
  author = {Schmidt, Thomas},
  date = {2008-05-01},
  journaltitle = {Calculus of Variations and Partial Differential Equations},
  shortjournal = {Calc. Var.},
  volume = {32},
  number = {1},
  pages = {1--24},
  issn = {1432-0835},
  doi = {10.1007/s00526-007-0126-5},
  url = {https://doi.org/10.1007/s00526-007-0126-5},
  urldate = {2021-11-05}
}

@article{schmidt2009regularity,
  title={Regularity of relaxed minimizers of quasiconvex variational integrals with (p, q)-growth},
  author={Schmidt, Thomas},
  journal={Archive for rational mechanics and analysis},
  volume={193},
  number={2},
  pages={311--337},
  year={2009},
  publisher={Springer-Verlag, Berlin/Heidelberg}
}

@article{kuusi2016partial,
  title={Partial regularity and potentials},
  author={Kuusi, Tuomo and Mingione, Giuseppe},
  journal={Journal de l’{\'E}cole polytechnique—Math{\'e}matiques},
  volume={3},
  pages={309--363},
  year={2016}
}

@article{marcellini1989regularity,
  title={Regularity of minimizers of integrals of the calculus of variations with non standard growth conditions},
  author={Marcellini, Paolo},
  journal={Archive for Rational Mechanics and Analysis},
  volume={105},
  pages={267--284},
  year={1989},
  publisher={Citeseer}
}

@article{defilippis2022quasiconvexity,
  title={Quasiconvexity and partial regularity via nonlinear potentials},
  author={De Filippis, Cristiana},
  journal={Journal de Math{\'e}matiques Pures et Appliqu{\'e}es},
  volume={163},
  pages={11--82},
  year={2022},
  publisher={Elsevier}
}

@article{defilippis2023singular,
  title={Singular multiple integrals and nonlinear potentials},
  author={De Filippis, Cristiana and Stroffolini, Bianca},
  journal={Journal of Functional Analysis},
  volume={285},
  number={2},
  pages={109952},
  year={2023},
  publisher={Elsevier}
}

@article{dephilippis2016structure,
  title={On the structure of $\mathscr{A}$-free measures and applications},
  author={De Philippis, Guido and Rindler, Filip},
  journal={Annals of Mathematics},
  pages={1017--1039},
  year={2016},
  publisher={JSTOR}
}

@article{guerra2022compensated,
  title={Compensated compactness: continuity in optimal weak topologies},
  author={Guerra, Andr{\'e} and Rai{\cb{t}}{\u{a}}, Bogdan and Schrecker, Matthew RI},
  journal={Journal of Functional Analysis},
  volume={283},
  number={7},
  pages={109596},
  year={2022},
  publisher={Elsevier}
}

@article{kazaniecki2017anisotropic,
  title={Anisotropic ornstein noninequalities},
  author={Kazaniecki, Krystian and Stolyarov, Dmitriy M and Wojciechowski, Micha{\l}},
  journal={Analysis \& PDE},
  volume={10},
  number={2},
  pages={351--366},
  year={2017},
  publisher={Mathematical Sciences Publishers}
}

@article{van2022injective,
  title={Injective Ellipticity, Cancelling Operators, and Endpoint},
  author={Van Schaftingen, Jean},
  journal={Geometric and Analytic Aspects of Functional Variational Principles: Cetraro, Italy},    
year={2022},
  pages={259},
  publisher={Springer Nature}
}

@article{hamburger1992regularity,
  title={Regularity of differential forms minimizing degenerate elliptic functionals*},
  author={Hamburger, C},
  journal={J. Reine Angew. Math.},
  volume={431},
  pages={7--64},
  year={1992}
}

@article{uhlenbeck1977regularity,
  title={Regularity for a class of non-linear elliptic systems},
  author={Uhlenbeck, Karen},
  journal={Acta Mathematica},
  volume={138},
  pages={219--240},
  year={1977},
  publisher={Institut Mittag-Leffler}
}

@article{beck2013regularity,
  title={Regularity results for differential forms solving degenerate elliptic systems},
  author={Beck, Lisa and Stroffolini, Bianca},
  journal={Calculus of Variations and Partial Differential Equations},
  volume={46},
  number={3},
  pages={769--808},
  year={2013},
  publisher={Springer}
}

@book{Hormander1,
  title = {The Analysis of Linear Partial Differential Operators {I}},
  author = {H{\"o}rmander, Lars},
  year = {2003},
  series = {Classics in {{Mathematics}}},
  publisher = {Springer Berlin Heidelberg},
  address = {Berlin, Heidelberg},
  doi = {10.1007/978-3-642-61497-2},
  urldate = {2023-07-17},
  isbn = {978-3-540-00662-6}
}

@book{Hormander2,
  title = {The {{Analysis}} of {{Linear Partial Differential Operators II}}},
  author = {Hörmander, Lars},
  date = {2005},
  series = {Classics in {{Mathematics}}},
  publisher = {Springer Berlin Heidelberg},
  location = {Berlin, Heidelberg},
  doi = {10.1007/b138375},
  url = {http://link.springer.com/10.1007/b138375},
  urldate = {2023-07-17},
  isbn = {978-3-540-22516-4}
}

@incollection{Malgrange1962,
  title = {Sur Les Systèmes Différentiels à Coefficients Constants},
  booktitle = {Les {{Équations}} Aux {{Dérivées Partielles}} ({{Paris}}, 1962)},
  author = {Malgrange, Bernard},
  date = {1963},
  series = {Colloq. {{Internat}}. {{CNRS}}, No. 117},
  pages = {113--122},
  publisher = {CNRS, Paris}
}

@article{breit2020trace,
  title={On the trace operator for functions of bounded A-variation},
  author={Breit, Dominic and Diening, Lars and Gmeineder, Franz},
  journal={Analysis \& PDE},
  volume={13},
  number={2},
  pages={559--594},
  year={2020},
  publisher={Mathematical Sciences Publishers}
}

@article{gmeineder2019embeddings,
  title={Embeddings for A-weakly differentiable functions on domains},
  author={Gmeineder, Franz and Rai{\c{t}}{\u{a}}, Bogdan},
  journal={Journal of Functional Analysis},
  volume={277},
  number={12},
  pages={108278},
  year={2019},
  publisher={Elsevier}
}

@article{gmeineder2021limiting,
  title={On limiting trace inequalities for vectorial differential operators},
  author={Gmeineder, Franz and Rai{\c{t}}{\u{a}}, Bogdan and Van Schaftingen, Jean},
  journal={Indiana University Mathematics Journal},
  volume={70},
  number={5},
  pages={2133--2176},
  year={2021},
  publisher={JSTOR}
}

@article{gmeineder2024boundary,
  title={Boundary ellipticity and limiting L1-estimates on halfspaces},
  author={Gmeineder, Franz and Rai{\c{t}}{\u{a}}, Bogdan and Van Schaftingen, Jean},
  journal={Advances in Mathematics},
  volume={439},
  pages={109490},
  year={2024},
  publisher={Elsevier}
}

@article{raictua2019critical,
  title={Critical $L^p$-differentiability of $BV^A$-maps and canceling operators},
  author={Rai{\cb{t}}{\u{a}}, Bogdan},
  journal={Transactions of the American Mathematical Society},
  volume={372},
  number={10},
  pages={7297--7326},
  year={2019}
}

@article{raita2020continuity,
  title={Continuity and canceling operators of order n on R n},
  author={Raiț{\u{a}}, Bogdan and Skorobogatova, Anna},
  journal={Calculus of Variations and Partial Differential Equations},
  volume={59},
  number={2},
  pages={85},
  year={2020},
  publisher={Springer}
}

@article{diening2012partial,
  title={Partial regularity for minimizers of quasi-convex functionals with general growth},
  author={Diening, Lars and Lengeler, Daniel and Stroffolini, Bianca and Verde, Anna},
  journal={SIAM Journal on Mathematical Analysis},
  volume={44},
  number={5},
  pages={3594--3616},
  year={2012},
  publisher={SIAM}
}

@article{Ehrenpreis1960,
  title = {A Fundamental Principle for Systems of Linear Partial Differential Equations with Constant Coefficients and Some of Its Applications},
  author = {Ehrenpreis, Leon},
  date = {1960},
  journaltitle = {Proceedings Intern. Symp. on Linear Spaces},
  pages = {161--174.}
}

@book{Palamodov1970,
  title = {Linear {{Differential Operators}} with {{Constant Coefficients}}},
  author = {Palamodov, Victor P.},
  date = {1970},
  series = {Grundlehren Der Mathematischen {{Wissenschaften}}},
  number = {168},
  publisher = {Springer Berlin Heidelberg},
  location = {Berlin, Heidelberg},
  doi = {10.1007/978-3-642-46219-1},
  url = {http://link.springer.com/10.1007/978-3-642-46219-1},
  urldate = {2026-02-16},
  isbn = {978-3-642-46221-4}
}

@article{BhattacharyaLeonetti1991,
  title = {A New {{Poincaré}} Inequality and Its Application to the Regularity of Minimizers of Integral Functionals with Nonstandard Growth},
  author = {Bhattacharya, Tilak and Leonetti, Francesco},
  date = {1991},
  journaltitle = {Nonlinear Analysis: Theory, Methods \& Applications},
  shortjournal = {Nonlinear Analysis: Theory, Methods \& Applications},
  volume = {17},
  number = {9},
  pages = {833--839},
  issn = {0362546X},
  doi = {10.1016/0362-546X(91)90157-V},
  langid = {english}
}

@article{Decell1965,
  title = {An {{Application}} of the {{Cayley-Hamilton Theorem}} to {{Generalized Matrix Inversion}}},
  author = {Decell, Henry P.},
  date = {1965},
  journaltitle = {SIAM Review},
  volume = {7},
  number = {4},
  eprint = {2027771},
  eprinttype = {jstor},
  pages = {526--528},
  publisher = {{Society for Industrial and Applied Mathematics}},
  issn = {0036-1445},
  url = {https://www.jstor.org/stable/2027771},
  urldate = {2026-03-03}
}

@book{Maligranda1989,
  title = {Orlicz Spaces and Interpolation},
  author = {Maligranda, Lech},
  date = {1989},
  series = {Seminars in {{Mathematics}}},
  number = {5},
  publisher = {Departamento de Matemática, Universidade Estadual de Campinas},
  langid = {english},
  pagetotal = {206}
}

@article{Zygmund1956,
  title = {On a Theorem of {{Marcinkiewicz}} Concerning Interpolation of Operators},
  author = {Zygmund, A},
  date = {1956},
  journaltitle = {J. Math. Pures Appl.},
  volume = {35},
  pages = {223--248},
  url = {https://cir.nii.ac.jp/crid/1570572699787513344},
  urldate = {2022-07-12}
}

@article{ornstein1962non,
  title={A non-inequality for differential operators in the $L_1$ norm},
  author={Ornstein, Donald},
  journal={Archive for Rational Mechanics and Analysis},
  volume={11},
  number={1},
  pages={40--49},
  year={1962},
  publisher={Springer}
}

@article{conti2005new,
  title={A new approach to counterexamples to L 1 estimates: Korn’s inequality, geometric rigidity, and regularity for gradients of separately convex functions},
  author={Conti, Sergio and Faraco, Daniel and Maggi, Francesco},
  journal={Archive for rational mechanics and analysis},
  volume={175},
  number={2},
  pages={287--300},
  year={2005},
  publisher={Springer}
}

@article{beck2024gradient,
  title={Gradient integrability for bounded $\mathrm{BD}$-minimizers},
  author={Beck, Lisa and Eitler, Ferdinand and Gmeineder, Franz},
  journal={arXiv preprint arXiv:2412.16131},
  year={2024}
}

@article{guerra2020necessity,
  title={On the necessity of the constant rank condition for $ L^{p} $ estimates},
  author={Guerra, Andr{\'e} and Rai{\c{t}}{\u{a}}, Bogdan},
  journal={Comptes Rendus. Math{\'e}matique},
  volume={358},
  number={9-10},
  pages={1091--1095},
  year={2020}
}

@article{arroyo2021new,
  title={New projection and Korn estimates for a class of constant-rank operators on domains},
  author={Arroyo-Rabasa, Adolfo},
  journal={arXiv preprint arXiv:2109.14602},
  year={2021}
}

@article{iwaniec1999nonlinear,
  title={Nonlinear Hodge theory on manifolds with boundary},
  author={Iwaniec, Tadeusz and Scott, Chad and Stroffolini, Bianca},
  journal={Annali di Matematica pura ed applicata},
  volume={177},
  number={1},
  pages={37--115},
  year={1999},
  publisher={Springer}
}

@article{harkonen2024syzygies,
  title={Syzygies, constant rank, and beyond},
  author={H{\"a}rk{\"o}nen, Marc and Nicklasson, Lisa and Rai{\cb{t}}{\u{a}}, Bogdan},
  journal={Journal of Symbolic Computation},
  volume={123},
  pages={102274},
  year={2024},
  publisher={Elsevier}
}
    
\end{document}